\pdfoutput=1
\RequirePackage{ifpdf}
\ifpdf 
\documentclass[pdftex]{sigma}
\else
\documentclass{sigma}
\fi

\usepackage{comment}

\numberwithin{equation}{section}

\newtheorem{Theorem}{Theorem}[section]
\newtheorem{Corollary}[Theorem]{Corollary}
\newtheorem{Lemma}[Theorem]{Lemma}
\newtheorem{Proposition}[Theorem]{Proposition}
\newtheorem{Conjecture}[Theorem]{Conjecture}
\newtheorem{Question}[Theorem]{Question}
\newtheorem{Problem}[Theorem]{Problem}

 { \theoremstyle{definition}
\newtheorem{Definition}[Theorem]{Definition}
\newtheorem{Example}[Theorem]{Example}
\newtheorem{Remark}[Theorem]{Remark} }

\usepackage{enumerate}

\newcommand{\etype}[1]{\renewcommand{\labelenumi}{(#1{enumi})}}
\def\eroman{\etype{\roman}}
\newcommand{\C}{\mathbb{C}}
\newcommand{\Q}{\mathbb{Q}}

\newcommand{\R}{\mathbb{R}}
\newcommand{\HH}{\mathbb{H}}
\newcommand{\ssl}{\operatorname{sl}}
\newcommand{\so}{\operatorname{so}}
\newcommand{\ssp}{\operatorname{sp}}

\newcommand{\G}{\operatorname{G}}
\newcommand{\ff}{\operatorname{F}}
\newcommand{\su}{\operatorname{su}}

\newcommand{\Image}{\operatorname{Im}}
\newcommand{\Char}{\operatorname{Char}}
\newcommand{\Cent}{\operatorname{Cent}}
\newcommand{\tr}{\operatorname{tr}}
\newcommand{\diag}{\operatorname{Diag}}

\newcommand{\spann}{\operatorname{span}}
\newcommand{\GL}{\operatorname{GL}}

\newcommand{\N}{\mathbb{N}}
\newcommand{\Discr}{\operatorname{Discr}}
 
 \DeclareMathOperator{\End}{End}

\def\Cent{\operatorname{Cent}}

\def\UD{\operatorname{UD}}
\def\sl{\operatorname{sl}}
\def\chara{\operatorname{char}}
\def\a{\alpha}
\def\la{\lambda}
\def\bfi{\mathbf{i}}
\def\bfj{\mathbf{j}}
\newcommand{\discr}{\operatorname{Discr}}
\newcommand{\D}{\operatorname{Dist}}

\def\gl{\operatorname{gl}}
\newcommand{\ad}{\operatorname{ad}}

\newcommand{\SL}{\operatorname{SL}}
\newcommand{\real}{\operatorname{Re}}
\newcommand{\ve}{\operatorname{Ve}}

\begin{document}

\allowdisplaybreaks

\newcommand{\arXivNumber}{1909.07785}

\renewcommand{\thefootnote}{}

\renewcommand{\PaperNumber}{071}

\FirstPageHeading

\ShortArticleName{Evaluations of Noncommutative Polynomials on Algebras}

\ArticleName{Evaluations of Noncommutative Polynomials\\ on Algebras: Methods and Problems,\\ and the L'vov--Kaplansky Conjecture)\footnote{This paper is a~contribution to the Special Issue on Algebra, Topology, and Dynamics in Interaction in honor of Dmitry Fuchs. The full collection is available at \href{https://www.emis.de/journals/SIGMA/Fuchs.html}{https://www.emis.de/journals/SIGMA/Fuchs.html}}}
 
\Author{Alexei KANEL-BELOV~$^\dag$, Sergey MALEV~$^\ddag$, Louis ROWEN~$^\S$ and Roman YAVICH~$^\ddag$}

\AuthorNameForHeading{A.~Kanel-Belov, S.~Malev, L.~Rowen and R.~Yavich}

\Address{$^\dag$~Bar-Ilan University, MIPT, Israel}
\EmailD{\href{mailto:kanelster@gmail.com}{kanelster@gmail.com}}

\Address{$^\ddag$~Department of Mathematics, Ariel University of Samaria, Ariel, Israel}
\EmailD{\href{mailto:sergeyma@ariel.ac.il}{sergeyma@ariel.ac.il}, \href{mailto:romany@ariel.ac.il}{romany@ariel.ac.il}}

\Address{$^\S$~Department of Mathematics, Bar Ilan University, Ramat Gan, Israel}
\EmailD{\href{mailto:rowen@math.biu.ac.il}{rowen@math.biu.ac.il}}

\ArticleDates{Received September 18, 2019, in final form July 08, 2020; Published online July 27, 2020}

\Abstract{Let $p$ be a polynomial in several non-commuting variables with coefficients in a~field $K$ of arbitrary characteristic. It has been conjectured that for any $n$, for $p$ multilinear, the image of $p$ evaluated on the set $M_n(K)$ of $n$ by $n$ matrices is either zero, or the set of scalar matrices, or the set ${\rm sl}_n(K)$ of matrices of trace 0, or all of $M_n(K)$. This expository paper describes research on this problem and related areas. We discuss the solution of this conjecture for $n=2$ in Section~2, some decisive results for $n=3$ in Section~3, and partial information for $n\geq 3$ in Section~4, also for non-multilinear polynomials. In addition we consider the case of~$K$ not algebraically closed, and polynomials evaluated on other finite dimensional simple algebras (in particular the algebra of the quaternions). This review recollects results and technical material of our previous papers, as well as new results of other researches, and applies them in a~new context. This article also explains the role of the Deligne trick, which is related to some nonassociative cases in new situations, underlying our earlier, more straightforward approach. We pose some problems for future generalizations and point out possible generalizations in the present state of art, and in the other hand providing counterexamples showing the boundaries of generalizations.}


\Classification{16H05; 16H99; 16K20; 16R30; 16R40; 17B99}

\begin{flushright}
\it On occasion of the 80-th birthday of D.B.~Fuchs
\end{flushright}

{\small \tableofcontents}

\renewcommand{\thefootnote}{\arabic{footnote}}
\setcounter{footnote}{0}

\section{Introduction}

In this review, we present a systematized exposition including
results obtained in previous papers. In addition, we have
systematized ideas and methods of proofs. Surprisingly, for quite
elementary looking results, we have used the Deligne trick, and a
technique for working with central simple algebras dating back to
Amitsur. At the end of each section, we present open problems and
our ideas, not all of which we were able to complete. We would be
happy if other scientists will succeed.

$K\langle x_1,\dots,x_m\rangle$ denotes the free $K$-algebra generated by noncommuting variables
$x_1,\dots,x_m$; we refer to the elements of $K\langle
x_1,\dots,x_m\rangle$ as {\it polynomials}. Consider any algebra $R$
over a~field~$K$. A polynomial $p\in K\langle x_1,\dots,x_m\rangle$
is called a {\it polynomial identity} (PI) of the algebra $R$ if
$p(a_1,\dots,\allowbreak a_m)=0$ for all $a_1,\dots,a_m\in R$; $p\in K\langle
x_1,\dots,x_m\rangle$ is a {\it central polynomial} of $R$, if for
any $a_1,\dots,a_m\in R$ one has $\mbox{$p(a_1,\dots,a_m)\in
\Cent(R)$}$ but $p$ is not a PI of $R$.

For any polynomial $p\in K\langle x_1,\dots,x_m\rangle$, the $image$
of $p$ (in $R$) is defined as
\begin{gather*}
\Image p: =\{r\in R:
\ \text{there exist}\ a_1,\dots,a_m\in R\ \text{such that}\
p(a_1,\dots,a_m)=r \}.
\end{gather*}

\begin{Remark}\label{cong-2}$\Image p$ is invariant under conjugation, since
\begin{gather*}
sp(x_1,\dots,x_m)s^{-1}=p\big(s x_1s ^{-1},s x_2s ^{-1},\dots,s x_ms^{-1}\big)\in\Image p,
\end{gather*}
 for any invertible element $s$.
\end{Remark}

Images of polynomials evaluated on algebras play an important role
in noncommutative algebra. In particular, various challenging
problems related to the theory of polynomial identities have been
settled after the construction of central polynomials by Formanek~\cite{F1} and Razmyslov~\cite{Ra1}.

This survey consists of the main results by the authors describing
the possible images of polynomials, especially in connection with
the partial solution of a conjecture attributed to L'vov and
Kaplansky concerning evaluations of polynomials on matrices that was formulated in~\cite{Dn}:

\begin{Conjecture}[L'vov--Kaplansky]\label{lvov-int}Let $p$ be a~multilinear polynomial. Then the set of values of $p$
on the matrix algebra $M_n(K)$ over an infinite field $K$ is a~vector space.
\end{Conjecture}

\begin{Remark}\label{linear-2} It is not difficult to ascertain the linear span of the values of any multilinear polynomial. Indeed, the linear span of its values comprises a Lie ideal since, as is well known,
 \begin{gather*}
 [a, p(a_1, \dots, a_n)]
 = p([a,a_1],a_2 \dots, a_n)+ p(a_1,[a,a_2] \dots, a_n)+ \cdots +p(a_1, \dots,[a,
 a_n]),
 \end{gather*}
 and Herstein \cite{Her} characterized Lie ideals of a
 simple ring $R$ as either being contained in the center or
 containing the commutator Lie ideal $[R,R]$. Another proof is
 given in \cite{BK}; also see \cite[Lemma 4]{BMR1}. It
is considerably more difficult to determine the actual image set
 $\Image p$, rather than its linear span.
\end{Remark}

Thus Conjecture~\ref{lvov-int} is equivalent to the following:

\begin{Conjecture}\label{Polynomial image-int}
If $p$ is a multilinear polynomial evaluated on the matrix ring
$M_n(K)$, then $\Image p$ is either $\{0\}$, $K$, $\ssl_n(K)$, or
$M_n(K)$. Here $K$ indicates the set of scalar matrices and
$\ssl_n(K)$ is the set of matrices of trace zero.
\end{Conjecture}

Note that these options are mutually exclusive when $\chara(K)$ does
not divide~$n$. While lacking a verification of this conjecture, one
is led to the following more general question:

\begin{Question} Given a polynomial $p$ $($not necessarily
multilinear$)$, what is its possible image set? Which polynomial of
minimal degree produces one of these image sets?
\end{Question}

Recall the \textbf{standard polynomial} \begin{gather*}
 s_k: = \sum _{\pi \in S_k}
\operatorname{sgn}(\pi) x_{\pi(1)}\cdots x_{\pi(k)}.
\end{gather*}

A polynomial $p$ is {\it trace vanishing} if each of its evaluations
has trace 0. For example, it is easy to see that $s_{2k}$ is trace
vanishing.

\begin{Example}\label{ex_multilin-int}
$\Image p$ can indeed equal $\{0\}$, $K$, $\ssl_n(K)$, or
$M_n(K)$.

\begin{enumerate}\itemsep=0pt\eroman \item $\Image p =
\{0\}$ iff the polynomial $p$ is a PI, and $s_{2n}$ is an
example of such a polynomial by the Amitsur--Levitzki
theorem~\cite{AL}. To determine all such PI's leads us to Specht's
finite basis problem, which for multilinear polynomial identities is
not yet settled in nonzero characteristic. Also Kemer's solution is
non-computational, so we do not yet have explicit generators of the
T-ideal of PI's for $n>2$.

 \item If the polynomial $p$ is central,
then its image is $K$, and examples of such polynomials can be found
in
 \cite{Ra1} and in \cite{F1}. For $n=2$ the central polynomial of smallest
 degree is the multilinearization of $[x,y]^2$. The central polynomial of smallest
 degrees are known for $n = 3,4$ but not in general.

\item When $p = x_1x_2-x_2x_1$, $\Image p = \ssl_n(K)$ by a theorem of Albert and
Muckenhoupt~\cite{AM}. Obviously $p$ has the lowest possible degree.

In general, if $\Image p \subseteq \ssl_n(K)$, then $p$ is trace
vanishing, which raises the issue, when does this imply that $\Image
p$ consists of all commutators? We shall investigate this issue, and
obtain counterexamples for non-multilinear polynomials.

\item $\Image p=M_n(K)$ for $p=x$.
\end{enumerate}
\end{Example}

When $K$ is a finite field, Chuang \cite{Ch} proved that any
subset $S \subseteq M_n(K)$ containing $0$ is the image of a
polynomial with constant term zero, if and only if $S$ is invariant
under conjugation. Later Chuang's result was generalized by Kulyamin
\cite{Ku1, Ku2} for graded algebras.

The research detailed in this paper focuses on associative
algebras, mostly matrix algebras~$M_n(K)$ over an infinite field
$K$, for $n =2,3$. We also have density results for arbitrary $n$.
In~\cite{DyKl} Dykema and Klep obtain an affirmative answer for
the L'vov--Kaplansky conjecture, for multilinear polynomials of
degree $3$ when $n$ is either even, or odd and $\leq 15$.

In \cite{Mes} Mesyan conjectured that if $n\geq m-1$, then any
multilinear polynomial of degree $m$ evaluated on $M_n(K)$ takes all
values of trace zero, and proved it for $m=3$. In~\cite{BW} Buzinski
and Winstenley proved Mesyan's conjecture for $m=4$.

 In \cite{Fag}
Fagundes denotes by $UT_n^{(k)}$ for $k\geq 0$ the set of strictly
upper triangular matrices which, besides the main diagonal, also
have $k$ zero diagonals located above the main diagonal, and proves
that if $p$ is a multilinear polynomial evaluated on~$UT_n^{(0)}$ of
degree $m$ then its image is either $\{0\}$ or $UT_n^{(m-1)}$. In
particular, $\Image p$ is a vector space.

In \cite{FM} Fagundes and de Mello describe the images of
multilinear polynomials of degree $\le 4$ on the upper triangular
matrices.

Other works on upper triangular matrices include \cite{San,W1,W2,W3}.

Vitas \cite{V} proved for any nonzero multilinear polynomial~$p$, that if $A$ is an algebra with a~surjective inner derivation, such as the Weyl algebra, then $\Image p = A$.

A \textit{Lie polynomial} is an element of the free Lie algebra in
the alphabet $\{ x_i\colon i \in~I\}$, cf.~\cite[p.~8]{Ra4}. In other
words, a Lie polynomial is a sum of Lie monomials $\alpha_j h_j$,
where $h_j$ is a Lie word, built inductively: each letter $x_i$ is a
Lie word of degree~1, and if $h_j, h_k$ are distinct Lie words of
degree $d_j$ and~$d_k$,
then $[h_j, h_k]$ is a Lie word of degree $d_j+d_k$. A Lie
polynomial $p$ is \textit{multilinear} if each letter appearing in
$p$ appears exactly once in each of its Lie monomials.

 In \cite{S1}
\v{S}penko proved the L'vov--Kaplansky conjecture for Lie
polynomials $p$ of degree $\leq 4$ evaluated on matrix algebra $M_n$.

In \cite{AEV} Anzis, Emrich and Valiveti proved the L'vov--Kaplansky
conjecture for multilinear Lie polynomials of degree $3$ and $4$ evaluated on the Lie algebras $\su(n)$ of traceless skew-Hermitian matrices and $\so(n)$ of
skew-symmetric matrices.

In \cite{MO} Ma and Oliva proved that the image of any multilinear
Jordan polynomial of degree $3$ evaluated on the Jordan algebras of
real and complex symmetric matrices forms a vector space.

In \cite{LT} Li and Tsui proved that if $R$ is a central simple
algebra of degree $n$ over its center~$F$ and $\chara(F) = 0$, with
$\lambda\in F\setminus \{0, -1\}$, then there exist $a, c \in R$
such that for any element $r \in R$ of reduced trace $0$ there is an
element $b \in R$ such that $r = [a, [c, b]] + \lambda[c, [a, b]]$.

In \cite{Bre} Bre\v{s}ar proved that for any unital algebra $A$ over
a field $F$ of characteristic~$0$, if $1\in[A,A]$ then
$[A,A]\subseteq\spann f(A)$ for every nonconstant polynomial~$f$.
Also he investigated the set $f(A)-f(A)$ of differences of
evaluations
 and proved that for any algebraically closed field $F$
of characteristic $0$ and any noncommutative polynomial $f$, the set
$f(A)-f(A)$ on $A=M_n(F)$ with $n\geq 2$ contains all square-zero
matrices.

Papers on polynomial maps evaluated on matrix algebras include
\cite{GK, Wat}, who investigated maps that preserve zeros of
multilinear polynomials.

Research into polynomial image sets has strong connections with the
PI-theory. In particular, a polynomial identity is a polynomial whose
image set is $\{0\}$, and a central polynomial is a
polynomial whose images are central elements (in matrix algebras,
central elements are scalar matrices). The methods of working with PI
are set out in \cite{BrPS,DPP,Dr2,DF, DP, F2, GZ,Go, Ha, I, Ke1, Ke2, LeZh, P2, Ra2, Sa},
for more detailed exposition of PI theory and related
references see~\cite{BKR}. For combinatorial questions see~\cite{BBL}.

In the study of evaluations of polynomials on algebras, the approach associated with the investigation of normal bases of algebras seems significant. In this context, the research into Gr\"obner--Shirshov bases by the Bokut school (see \cite{Bokut, BokutChen, BK2}) is of interest, as is the study of evaluations of
polynomials on vertex algebras.

\subsection{Evaluations of words}

 The parallel topic in group theory (the images of words in
 groups) also has been studied extensively, particularly in recent
 years.
Investigation of the image sets of words in pro-$p$-groups is
related to the investigation of Lie polynomials and helped Zelmanov
\cite{Ze} to prove that the free pro-$p$-group cannot be embedded in
the algebra of $n\times n$ matrices when $p\gg n$. (For $p>2$, the
impossibility of embedding the free pro-$p$-group into the algebra
of $2\times 2$ matrices had been proved by Zubkov \cite{Zu}.) The
general problem of nonlinearity of the free pro-$p$-group is related
on the one hand with images of Lie polynomials and words in groups,
and on the other hand with problems of Specht type, which is of
significant current interest.

Let $w=w(x_1,\ldots, x_m)$ be an element of the free group $F_m(X)$, where
$X=\{x_1,x_2,\dots, x_m\}$. Given a group $G$, we consider
the corresponding evaluation map $f_{w, G}\colon G^m \rightarrow G$ corresponding to the word $w$.
This map is called a {\it word map}, which for convenience we also
notate as $w$ instead of $f_{w, G}$. Note that the identity
matrix~$I$ belongs to the image of any word map.

The major question under consideration is the size of the image
$w(G)\subseteq G$. Surjectivity of the map $w$ means that $w(G)=G$,
i.e., the equations $w(x_1,\ldots,x_m)=g$ can be solved for
each$g\in G$. This is, of course, a rare phenomenon even for
``good'' classes of groups. So usually one has to vary the word $w$
and group $G$ to obtain a reasonable estimate of $w(G)$. The typical
classes of groups which provide such estimates are simple algebraic
groups, simple and perfect finite groups, and some others.

The theorem of Borel~\cite{B} (also cf.~\cite{La}) states that for any
simple (semisimple) algebraic group~$G$ and any word $w$ of the free
group on $m$ variables, the word map $w\colon G^m\to G$ is dominant. If
the ground field~$K$ is algebraically closed this implies
immediately that \mbox{$w(G(K))^2=G$}, that is, every element $g\in G$ is
a product of two elements from~$w(G)$.

For an arbitrary infinite field $K$ and arbitrary word $w$,
Hui--Larsen--Shalev \cite{HLS} proved that $w(G(K))^4=G$, that is,
every element $g\in G$ is a product of four elements from~$w(G)$.
This estimate was improved by Egorchenova--Gordeev in~\cite{EG1} to
$w(G(K))^3=G$.

The most challenging open problem for word maps on semisimple
algebraic groups is the following, see~\cite{KKMP}:

\begin{Conjecture}Let $G={\rm PSL}_2(\mathbb C)$, and let $w=w(x,y)$ be an arbitrary
non-identity word in~$F(x,y)$. Then the word map $w\colon
{\rm PSL}_2(\mathbb C)\times {\rm PSL}_2(\mathbb C)\to {\rm PSL}_2(\mathbb C)$ is
surjective. In other words, the equation
\begin{gather*}
w(x_1,x_2)=a
\end{gather*}
has a solution for every $a\in {\rm PSL}_2(\mathbb C)$.
\end{Conjecture}

This conjecture is still widely open,
 there being only several partial results, see
 \cite{BGG,BanZar,GKP2,GKP3,GKP1,KKMP}. It is a special case of the conjecture in \cite[Question~2]{KBKP}.
\begin{Problem}
Let $\mathbb G$ be the class of simple groups $G$ of the form $G=
G(K)$ where $K=\bar K$ is an algebraically closed field and $ G$ is
a semisimple adjoint linear algebraic group. Is it true that word
maps evaluated on groups from $\mathbb G$ are surjective for all
nontrivial non-power words?
\end{Problem}
The latter problem for groups of type A$_n$ can be reformulated as
follows:

\begin{Problem}\label{pr:slozhaja}
Is the word map $w\colon {\rm PSL}_n(\mathbb C)\times {\rm PSL}_n(\mathbb C)\to
{\rm PSL}_n(\mathbb C)$ surjective for any non-trivial $w(x,y)\in
F_2(x,y)$?
\end{Problem}

Now we turn from simple algebraic groups to finite simple groups.
Ongoing interest to this area was initially stimulated by the
positive solution of Ore's problem: Every element of a finite simple
group is a single commutator (see \cite{EG} and the final solution
in Liebeck--O'Brien--Shalev--Tiep \cite{LOST}; a survey is given in~\cite{Mall}).

Formidable progress in the description of images of word maps on
finite simple groups was obtained by M.~Larsen and A.~Shalev, who
stimulated the development of this area of research under the name
``Waring type problems''. The latest result of Larsen--Shalev--Tiep~\cite{LaST1} (see also \cite{LaS,LaST2, LaT}), is as
follows:
\begin{Theorem}[\cite{LaST1}]\label{th:fg}
Let $w$ be an arbitrary non-trivial word of $F(x_1,\ldots, x_n)$.
There exists a~constant $N = N(w)$ such that for all finite non-abelian
simple groups of order greater than $N$ one has
\begin{gather*}w(G)^2 = G.\end{gather*}
\end{Theorem}

Different aspects of word maps are considered in a vast and
extended literature; we refer to the papers
\cite{BGK,BGKP,BaKu,BLS,BLS1,GKP2,GKP3, GKP1,KBKP,LaST1,Sha1,S,Sha3,Th}
 for details, surveys and further explanations.
Waring type questions for rings were considered by Matei Bre\v{s}ar
\cite{Bre}.

In \cite {BGEY} Kanel-Belov, Grigoriev, Elishev and Yu prove the
possibility of lifting a symplectomorphism to an automorphism of the
power series completion of the Weyl algebra of the corresponding
rank. They study the problem of lifting polynomial
symplectomorphisms in characteristic zero to automorphisms of the
Weyl algebra, by means of approximation by tame automorphisms.

\subsection{Non-multilinear polynomials}

As noted above, the analog to the L'vov--Kaplansky conjecture
formulated for any polynomial fails when $K$ is a finite field, so
we may assume that $K$ is infinite. The situation is considerably subtler for images of non-multilinear polynomials.

\begin{Definition}\label{basdef} A polynomial $p$ (written as a sum of monomials)
is called {\it semi-homogeneous of weighted degree $d$} with
(integer) {\it weights} $(w_1,\dots,w_m)$ if for each monomial $h$
of $p$, taking $d_j$ to be the degree of $x_{j}$ in $h$, we have{\samepage
\begin{gather*}d_1w_1+\dots+d_nw_n=d.\end{gather*} A
semi-homogeneous polynomial with weights $(1,1,\dots, 1)$ is called
$\it{homogeneous}$ of degree $d$.}

A polynomial $p$ is {\it completely homogeneous} of multidegree
$(d_1,\dots,d_m)$ if each variable $x_i$ appears the same number of
times $d_i$ in all monomials.
A polynomial $p\in K\langle x_1,\dots,x_m\rangle$ is multilinear
iff it is homogeneous of multidegree $(1,1,\dots,1)$. Thus, a
polynomial is multilinear if it is a polynomial of the form
\begin{gather*}p(x_1,\dots,x_m)=\sum_{\sigma\in S_m}c_\sigma
x_{\sigma(1)}\cdots x_{\sigma(m)},\end{gather*} where $S_m$ is the
symmetric group in $m$ letters and the coefficients $c_\sigma$ are
constants in $K$. \end{Definition}

\subsection{The main theorems}\label{2q70}

Here is a general combinatorial result.
\begin{Theorem}[{\cite[Theorem 1]{BMR3}}]\label{thmB1-pcp}
Let $p(x_1,\dots,x_m)$ be any multilinear polynomial evaluated on
$n\times n$ matrices over an infinite field. Assume that $p$ is
neither scalar nor PI. Then $\Image p$ contains a~matrix of the form
$c_n e_{n,1} + \sum\limits_{i=1}^{n-1} c_i e_{i,i+1}$ where $c_1,\dots,
c_n\neq 0$. When $\chara(K)$ is $0$ or prime to~$n$, $\Image p$
contains a~matrix with eigenvalues
$\big\{c,c\varepsilon,\dots,c\varepsilon^{n-1}\big\}$ for some $0 \ne c \in
K$, where $\varepsilon$ is a~primitive~$n$ root of~$1$.
\end{Theorem}

\subsubsection[The main theorems for $n=2$]{The main theorems for $\boldsymbol{n=2}$}\label{2q7}

Our most decisive results are for $n=2$, given in \cite{BMR1}, for
which we settle Conjecture \ref{Polynomial image-int}, proving the
following results (see \cite[Section~2]{BMR1} for terminology). We call
a field $K$ {\it quadratically closed} (with respect to the
polynomial $p$) if every nonconstant polynomial in~$K[x]$ in one
variable, of degree $\le 2\deg p$, has a root in~$K$.

\begin{Theorem}[{\cite[Theorem~1]{BMR1}}]\label{imhom2-int}
Let $p(x_1,\dots,x_m)$ be a semi-homogeneous polynomial evaluated
on the algebra $M_2(K)$ of $2\times 2$ matrices over a quadratically
closed field. Then $\Image p$ is either $\{0\}$, $K$, the set of all non-nilpotent matrices having
trace zero, $\ssl_2(K)$, or a~dense subset of $M_2(K)$ $($with respect to Zariski topology$)$.
\end{Theorem}

(We also give examples to show how $p$ can have these images.)

\begin{Theorem}[{\cite[Theorem~2]{BMR1}}]\label{main2-int} If $p$ is a multilinear polynomial evaluated on the
matrix ring~$M_2(K)$ $($where $K$ is a quadratically closed field$)$,
then $\Image p$ is either $\{0\}$, $K$, $\ssl_2$, or~$M_2(K)$.
\end{Theorem}

The L'vov--Kaplansky conjecture is proved in \cite{M} for $M_2(\mathbb{R})$, together with a~partial solution settling the major part of L'vov--Kaplansky's conjecture in this case, proving the following result:

\begin{Theorem}[{\cite[Theorem 1]{M}}] \label{main-2r} If $p$ is a multilinear polynomial evaluated on the matrix ring~$M_2(K)$ $($where $K$ is an arbitrary field$)$, then $\Image p$ is either $\{0\}$, or $K$ $($the set of scalar matrices$)$, or $\ssl_2\subseteq\Image p$. If $K=\mathbb{R}$ then $\Image p$ is either $\{0\}$, or $K$, or $\ssl_2$ or $M_2$.
\end{Theorem}

\begin{Remark} Assume that $p$ is a multilinear polynomial evaluated on $2\times 2$ matrices. According to
 Theorem \ref{main-2r}, $\Image p$ is $\{0\}$, or $K$, or $\ssl_2(K)$ or $\ssl_2(K)\subsetneqq\Image p$. In the last case it is clear that $\Image p$ must be Zariski dense in~$M_2(K)$, because otherwise $\dim(\Image p)=3$ and $\Image p$ is reducible, a~contradiction.
\end{Remark}

The situation is considerably subtler for images of non-multilinear,
completely homogeneous polynomials than for multilinear polynomials,
but nevertheless a classification of the possible images of
homogeneous polynomials evaluated on~$2\times 2$ matrices is
provided:

\begin{Theorem}[{\cite[Theorem 2]{M}}]\label{homogen-int}
 Let $p(x_1,\dots,x_m)$ be a semi-homogeneous polynomial eva\-lua\-ted on $2\times2$ matrices with real entries.
 Then $\Image p$ is either:
\begin{itemize}\itemsep=0pt \item $\{0\}$, \item the set $\R_{\geq 0}$, i.e., the matrices $\lambda I$ for $\lambda\geq 0$,
\item the set $\R$ of scalar matrices,
\item the set $\R_{\leq 0}$, i.e., the matrices $\lambda I$ for $\lambda\leq 0$,
\item the set $\ssl_{2,\geq0}(\R)$ of trace zero matrices with non-negative discriminant,
\item the set $\ssl_{2,\leq 0}(\R)$ of trace zero matrices with non-positive discriminant,
\item the set $\ssl_2(\R)$,
 \item or is Zariski dense in $M_2(\R)$.\end{itemize}
\end{Theorem}

\subsubsection[The main theorems for $n=3$]{The main theorems for $\boldsymbol{n=3}$}

 We have not classified the possible images of all homogeneous or
semi-homogeneous polynomials, but for $n=3$ we do describe all
possible images of trace vanishing semi-homogeneous polynomials:

\begin{Theorem}[{\cite[Theorem 3]{BMR2}}]\label{semi_tr0_3-3}
Let $p(x_1, \dots, x_m)$ be a semi-homogeneous polynomial which is
trace vanishing on $3\times 3$ matrices. Then $\Image p$ is one of
the following:
\begin{itemize}\itemsep=0pt
\item $\{0\}$, \item the set of scalar matrices $($which can occur
only if $\Char K=3)$,
 \item a dense subset of $\sl_3(K)$, or \item
the set of $3$-scalar matrices, i.e., the set of matrices having
eigenvalues
$\big\{\gamma,\gamma\varepsilon,\gamma\varepsilon^2\big\}$, where~$\varepsilon$ is a primitive cube root of~$1$.
\end{itemize}
\end{Theorem}

All of the cases in Theorem~\ref{semi_tr0_3-3} occur, and we give
an example of a completely homogeneous $3$-scalar polynomial.
Unfortunately the question of whether there exists a $3$-scalar
multilinear polynomial remains open.

Although we do not settle the L'vov--Kaplansky conjecture completely, we describe the
possible images of multilinear polynomials.

\begin{Theorem}[{\cite[Theorem 4]{BMR2}}]\label{multi_tr=0_3}
Let $p$ be a multilinear polynomial which is trace vanishing on
$3\times 3$ matrices over a field $K$ of arbitrary characteristic.
Then $\Image p$ is one of the following:
\begin{itemize}\itemsep=0pt
\item $\{0\}$, \item the set of scalar matrices,
\item the set of $3$-scalar matrices, or
\item for each triple $\lambda_1+\lambda_2+\lambda_3=0$ there exist a matrix $M\in\Image
p$ with eigenvalues~$\lambda_1$,~$\lambda_2$ and~$\lambda_3$.
\end{itemize}
\end{Theorem}

\begin{Theorem}[{\cite[Theorem~2]{BMR2}}]\label{main3-int}
If $p$ is a multilinear polynomial evaluated on $3\times 3$
matrices over algebraically closed field~$K$,
then $\Image p$ is one of the
following:
\begin{itemize}\itemsep=0pt
\item $\{0\}$,
\item the set of scalar matrices, \item
 $\sl_3(K)$, $($perhaps lacking the diagonalizable matrices of discriminant
 $0)$,
 \item
a dense subset of $M_3(K)$, \item the set of $3$-scalar matrices, or
\item the set of sums of scalar and $3$-scalar matrices.
\end{itemize}
\end{Theorem}

\subsubsection[Higher $n$]{Higher $\boldsymbol{n}$}

Let us improve the estimates of the dimension of $\Image p$, first for $n\ge 5$ and then for $n=4$.

\begin{Theorem}[{\cite[Theorem 4]{BMR3}}]\label{no-mpc-pcp} Assume that the characteristic of $K$
does not divide $n$, and $K = F[\varepsilon]$, where
$\varepsilon$ is a primitive $n$-th root of~$1$.
 Let $p$ be any multilinear polynomial evaluated on $n\times n$ matrices which is not PI or central.
 If $n\geq 5$, then the image of~$p$ is at least $\big(n^2-n+3\big)$-dimensional.
\end{Theorem}

\begin{Theorem}[{\cite[Theorem 5]{BMR3}}]\label{harmonic-4-int}
 Let $p$ be any multilinear polynomial evaluated on $4\times 4$ matrices, which is neither PI nor central.
 Assume that $\Char K \ne 2$.
 Then $\dim \Image p \ge 14$, equality holding only if the following conditions are satisfied:

 \begin{itemize}\itemsep=0pt
\item For any matrix units $a_i$, if $p(a_1,\dots,a_m)$ is diagonal
 then it has eigenvalues $(c,c,-c,-c)$ for some $c$.
\item
Any value of $p$ has eigenvalues
 $(\lambda_1,\lambda_2,-\lambda_1,-\lambda_2)$.
\end{itemize}
\end{Theorem}

In \cite{BMR4} we investigated possible images of homogeneous Lie
polynomials evaluated on $2\times 2$ matrices $n=2$, and obtained
the following result:

\begin{Theorem}[{\cite[Theorem 3]{BMR4}}]\label{mainlie} For any algebraically closed field $K$ of characteristic \mbox{$\ne 2$},
 the image of any Lie polynomial $f$ $($not necessarily homogeneous$)$ evaluated on $\ssl_2(K)$
 is either~$\ssl_2(K)$, or~$\{0\}$, or the set of trace zero non-nilpotent matrices.
\end{Theorem}

We also provide examples, showing the existence of completely
homogeneous Lie polynomials with exactly these image sets. In
\cite{BMR4} adjoint maps are used to construct an important example
of the completely homogeneous Lie polynomial which image is the set
of all trace vanishing matrices except for the nilpotents.

\subsubsection{Polynomials over quaternion algebras}

One can generalize the L'vov--Kaplansky conjecture for an arbitrary finite dimensional simple algebra.

\begin{Definition}
By the quaternion algebra $\HH$ we mean the four-dimensional
algebra $\langle 1,i,j,k\rangle_\R$ such that
\begin{gather*}i^2=j^2=k^2=-1,\qquad ij=-ji=k,\qquad jk=-kj=i,\qquad ki=-ik=j.\end{gather*} In this algebra, $\R= \R 1$ are called {\it scalars}, and $V=\R i + \R j + \R k$ are called {\it pure quaternions};
 $\{1,i,j,k\}$ are called {\it basic quaternions}. We also will use the standard
quaternion functions: the {\it norm} $\left\vert\left\vert
a+bi+cj+dk\right\vert\right\vert=\sqrt{a^2+b^2+c^2+d^2}$, the {\it
real part} $\real (a+bi+cj+dk)=a$, and the {\it pure quaternion}
part $\ve (a+bi+cj+dk)=bi+cj+dk$.
\end{Definition}

Any quaternion can be uniquely written as a sum of a scalar and a
pure quaternion, hence the functions of real and pure quaternion parts are well
defined. The norm function is multiplicative.

In \cite{A} Almutairi proved that if $p$ is a non-central
multilinear polynomial then its image contains all pure quaternions.
In Section~\ref{quat} we provide a complete classification of~$\Image p$,
settling the L'vov--Kaplansky conjecture for the quaternion algebra

\begin{Theorem}[{\cite[Theorem 1]{M2}}]\label{main-q}
 If $p$ is a multilinear polynomial evaluated on
the quaternion algebra~$\HH(\R)$, then $\Image p$ is either
$\{0\}$, $\R $ $($the space of scalar quaternions$)$, or $V $ $($the
space of pure quaternions$)$, or $\HH(\R)$.
\end{Theorem}

Also, the matrix representation of the quaternions ring is used in
order to provide a classification of possible images of
semi-homogeneous polynomials evaluated on the quaternion algebra.

\begin{Theorem}[{\cite[Theorem 2]{M2}}]\label{homogen-q}
If $p$ is a semi-homogeneous polynomial
evaluated on the quaternion algebra $\HH(\R)$, then $\Image p$
is either $\{0\}$, or $\R$, or $\R_{\geq 0}$, or $\R_{\leq 0}$, or
$V$, or some Zariski dense subset of $\HH$.
\end{Theorem}

\subsection{Some basic tools}\label{def1-7}

We recall the following elementary graph-theoretic lemma.

\begin{Lemma} \label{graph-2r} Let $p$ be a multilinear polynomial.
If the $a_i$ are matrix units, then $p(a_1,\dots,a_m)$ is
either~$0$, or~$c\cdot e_{ij}$ for some $i\neq j$, or a diagonal
matrix.
\end{Lemma}
The proof of this lemma is presented in Section~\ref{Euler-ch}.

\begin{Lemma}\label{scalar-2}
If $\Image p$ consists only of diagonal matrices, then $\Image p$
is either~$\{0\}$ or the set $K$ of scalar matrices.
\end{Lemma}
\begin{proof}
Suppose that some nonscalar diagonal matrix
$A=\diag\{\lambda_1,\dots,\lambda_n\}$ is in the image. Therefore
$\lambda_i\neq\lambda_j$ for some $i$ and $j$. The matrix
$A'=A+e_{ij}$ (here $e_{ij}$ is the matrix unit) is conjugate to $A$
so by Remark~\ref{cong-2} also belongs to $\Image p$. However $A'$
is not diagonal, a contradiction.
\end{proof}

The proofs of our theorems use some algebraic-geometric tools in
conjunction with these ideas from graph theory. The final part of
the proofs of Theorems~\ref{main2-int} and~\ref{main3-int} uses the
First fundamental theorem of invariant theory (that in case $\Char
K=0$, invariant functions evaluated on matrices are polynomials
involving traces), proved by Helling~\cite{Hel}, Procesi~\cite{P},
and Razmyslov \cite{Ra3}. The formulation in positive
characteristic, due to
 Donkin~\cite{D}, is somewhat more intricate.
The group $\GL_n(K)$ acts on $m$-tuples of $n\times n$-matrices by
simultaneous conjugation.
\begin{theorem*}[Donkin~\cite{D}]
For any $m, n\in \N$, the algebra of polynomial invariants
\[ K[M_n(K)^m]^{\GL_n(K)}\] under $\GL_n(K)$ is generated by the
trace functions
\begin{gather}\label{Donk-2} T_{i,j}(x_1,x_2,\dots,x_m) =
\operatorname{Trace}\bigg(x_{i_1}x_{i_2}\cdots
x_{i_r},\bigwedge\nolimits^jK^n\bigg),\end{gather} where
$i=(i_1,\dots,i_r)$, all $i_l\leq m$, $r\in\N$, $j>0$, and
$x_{i_1}x_{i_2}\cdots x_{i_r}$ acts as a linear transformation on
the exterior algebra $\bigwedge^jK^n$.
\end{theorem*}

\begin{Remark}For $n=2$, we have a polynomial function in expressions of the form
\[ \operatorname{Trace}\bigg(A,\bigwedge^2K^2\bigg)\] and $\tr A$ where $A$ is
monomial. ($\tr A$ denotes the trace.)
 Note that
 \[ \operatorname{Trace}\bigg(A,\bigwedge^2K^2\bigg)=\det A.\]
\end{Remark}

We also need the first fundamental theorem of invariant theory (see \cite[Theorem~1.3]{P}).
\begin{Proposition}\label{procesi-2r}
Any polynomial invariant of $n\times n$ matrices $A_1,\dots,A_m$ is
a polynomial in the invariants $\tr(A_{i_1}A_{i_2}\cdots A_{i_k})$,
taken over all possible $($noncommutative$)$ products of the $A_i$.
\end{Proposition}

(The second fundamental theorem, dealing with relations between
invariants, was proved by Procesi \cite{P} and Razmyslov \cite{Ra3}
in the case $\Char K=0$ and by Zubkov \cite{Zu} in the case $\Char
K>0$. We recommend reader to read the book \cite{P3}.)

\subsubsection{Generic matrices}

 Another major tool is Amitsur's theorem \cite[Theorem~3.2.6, p.~176]{Row}.
{\it Generic matrices} over $K$ are $n\times n$ matrices whose
entries are commuting indeterminates over $K$. By Proposition
\ref{Am1-2}) the algebra of generic matrices is a domain UD, whose
ring of fractions is a division algebra of degree $n$,
\begin{gather*}\widetilde{\UD} \subseteq M_n\big( F\big(\xi _{i,j}^{(k)}\big)\colon 1\le i,j
\le n,\ k\ge 1\big).\end{gather*} \ In fact, we need a slight modification of
this theorem, which is well known. We can define the reduced
characteristic coefficients of elements of $\widetilde{\UD}$, which
by \cite[Remark~24.67]{Row2} lie in $F_1$.

\begin{Proposition}\label{Am1-2} The algebra of generic matrices with traces
is a domain which can be embedded in the division algebra
$\widetilde{\UD}$ of central fractions of Amitsur's algebra of
generic matrices. Likewise, all of the functions in Donkin's theorem
can be embedded in $\widetilde{\UD}$.
\end{Proposition}
\begin{proof}
 Any trace function can be expressed as the ratio of two
central polynomials, in view of \cite[Theorem 1.4.12]{Row}; also see
\cite[Theorem~J, p.~27]{BR} which says for any characteristic
coefficient~$\omega_k $ of the characteristic polynomial
\begin{gather*}\lambda^t + \sum_{k=1}^t (-1)^k \omega _k \lambda ^{t-k}\end{gather*} that
\begin{gather}\label{trace2pol0-2}
\omega_k f(a_1, \dots, a_t, r_1, \dots, r_m) = \sum _{k=1}^t
f\big(T^{k_1}a_1, \dots, T^{k_t} a_t, r_1, \dots, r_m\big) ,
\end{gather}
summed over all vectors $(k_1, \dots, k_t)$ where each $k_i \in \{
0, 1 \}$ and $\sum k_i = k$, where $f$ is any $t$-alternating
polynomial (and $t = n^2$). In particular, taking $k =1$, so that
$\omega_k = \tr(T)$, we have
\begin{gather*}
\tr(T)f(a_1, \dots, a_t, r_1, \dots, r_m) = \sum _{j=1}^t f(a_1,
\dots, a_{j-1}, Ta_j, a_{j+1} , \dots, a_t, r_1, \dots, r_m) ,
\end{gather*}
so the trace
\[ \tr(T) = \frac{ \sum\limits_{j=1}^t f(a_1, \dots, a_{j-1},
Ta_j, a_{j+1} , \dots, a_t, r_1, \dots, r_m) }{f(a_1, \dots, a_t,
r_1, \dots, r_m)}
\]
 belongs to $\widetilde{\UD}$. In general, the function
\eqref{Donk-2} of Donkin's theorem is a matrix invariant and thus~can be written in terms of characteristic coefficients of matrices,
so we can apply equation~\eqref{trace2pol0-2}.
\end{proof}

\begin{Lemma}\label{nilp-2} If $\Char K$ does not divide $n$, then any non-identity $p(x_1, \dots, x_m)$ of $M_n(K)$ must either be a central polynomial or take on a value which is a matrix whose eigenvalues are not all
the same.\end{Lemma}
\begin{proof} Otherwise $p(x_1, \dots, x_m)- \frac 1n \tr(p(x_1, \dots,
x_m))$
is a nilpotent element in the algebra of generic matrices with
traces, so by Proposition~\ref{Am1-2} is 0, implying $p$ is central.
\end{proof}

For $n>2$, we have an easy consequence of the theory of division
algebras.

\begin{Lemma}\label{div-3} Suppose for some polynomial $p$ and some number $q< n$,
 that $p^q$ takes on only scalar values in $M_n(K)$, over an infinite field $K$, for $n$ prime. Then $p$
 takes on only scalar values in $M_n(K)$.
\end{Lemma}
\begin{proof} We can view $p$ as an element of the generic
division algebra $\widetilde{\UD}$, and we adjoin a $q$-root of 1 to
$K$ if necessary. Then $p$ generates a subfield of dimension 1 or
$n$ of $\widetilde{\UD}$. The latter is impossible, so the dimension
is 1; i.e., $p$ is already central.
\end{proof}

We also require one basic fact from linear algebra:
\begin{Lemma}\label{dim2-2r}
Let $V_i$ $($for $1\leq i\leq m)$ and $V$ be linear spaces over
arbitrary field $K$. Let $f(T_1,\dots,T_m)\colon \prod\limits_{i=1}^m
V_i\rightarrow V$ be a multilinear mapping $($i.e., linear with respect
to each $T_i)$. Assume there exist two points in $\Image f$ which
are not proportional. Then $\Image f$ contains a $2$-dimensional
plane. In particular, if $V$ is $2$-dimensional, then $\Image f=V$.
\end{Lemma}
\begin{proof}
 Let us denote for $\mu=(T_1\dots,T_m)$ and $\nu=(T_1',\dots,T_m')\in \prod\limits_{i=1}^m V_i$
 \begin{gather*}\D(\mu,\nu)=\#\{i\colon T_i\neq T_i'\}.\end{gather*}
 Consider $k=\min\big\{d\colon$ there exists $\mu,\nu\in\prod\limits_{i=1}^m V_i$ such that $f(\mu)$ is not proportional to
 $f(\nu)$ and $\D(\mu,\nu)=d\big\}$.
 We know $k\leq m$ by assumptions of lemma. Also $k\geq 1$ since any element of $V$ is proportional to itself.
 Assume $k=1$. In this case there exist $i$ and $T_1,\dots,T_m,T_i'$ such that $f(T_1,\dots,T_m)$ is not proportional to $f(T_1,\dots,T_{i-1},T_i',T_{i+1},\dots,T_m)$.
 Therefore \begin{gather*}\langle f(T_1,\dots,T_m),f(T_1,\dots,T_{i-1},T_i',T_{i+1},\dots,T_m)\rangle\subseteq\Image p\end{gather*} is $2$-dimensional.
 Hence we can assume $k\geq 2$.
 We can enumerate variables and consider $\mu=(T_1,\dots,T_m)$ and $\nu=(T_1',\dots,T_k',T_{k+1},\dots,T_m)$, $v_1=f(\mu)$ is not proportional to $v_2=f(\nu)$.
 Take any $a,b\in K$.
 Consider $v_{a,b}=f(aT_1+bT_1',T_2+T_2',\dots,T_k+T_k', T_{k+1},\dots,T_m)$. Let us open the brackets.
 We have
 \begin{gather*}v_{a,b}=av_1+bv_2+\sum_{\varnothing\subsetneqq S\subsetneqq\{1,\dots,k\}} c_S f(\theta_S),\end{gather*}
 where $c_S$ equals $a$ if $1\in S$ and $b$ otherwise,
 and $\theta_S=\big(\tilde T_1,\dots,\tilde T_k,T_{k+1},\dots,T_m\big)$ for $\tilde T_i=T_i$ if $i\in S$ or $T_i'$ otherwise.
 Note that any $\theta_S$ in the sum satisfies $\D(\theta_S,\mu)< k$ and $\D(\theta_S,\nu)<k$ therefore
 $f(\theta_S)$ must be proportional to both $v_1$ and $v_2$ and thus $f(\theta_S)=0$.
 Therefore $v_{a,b}=av_1+bv_2$ and hence $\Image f$ contains a $2$-dimensional plane.
\end{proof}

\begin{Lemma}\label{div2-3} The multiplicity of any eigenvalue of an element $a$ of
$\widetilde{\UD}$ must divide~$n$. In particular, when $n$ is odd,
$a$ cannot have an eigenvalue of multiplicity~$2$.
\end{Lemma}
\begin{proof} Recall \cite[Remark~4.106]{Row1} that for any element~$a$ in a division
algebra, represented as a~matrix, the eigenvalues of~$a$ occur with
the same multiplicity, which thus must divide~$n$.
\end{proof}

\begin{Lemma}\label{divAm-int} Assume that an element $a$ of $\widetilde{\UD}$ has
a unique eigenvalue~$\a$ $($of multiplicity~$n)$.

If $\Char(K)=0$, then $a$ is scalar.

If $\Char(K)=k\neq 0$, then $a$ is $k^l$-scalar for some $l$.
\end{Lemma}
\begin{proof} If $\Char(K)=0$, then $\a$ is an element of $\widetilde{\UD}$ and $a -\a I$ is nilpotent, and thus
$0$.

If $\Char(K)=k$ then $\a^{k^l}$ is an element of $\widetilde{\UD}$,
therefore $a^{k^l} -\a^{k^l} I$ is nilpotent, and thus $0$. Thus $a$
is $k^l$-scalar. This is impossible unless $k$ divides $n$.
\end{proof}

Over any field $K$, applying the structure theory of division rings to
Amitsur's theorem, it is not difficult to get an example of a~completely homogeneous polynomial $f$, noncentral on~$M_3(K)$, whose
values all have third powers central; clearly its image does not
comprise a~subspace of~$M_3(K)$. Furthermore, in the
(non-multilinear) completely homogeneous case, the set of values
could be dense without including all matrices. (Analogously,
although the finite basis problem for multilinear identities is not
yet settled in nonzero characteristic, there are counterexamples
for completely homogeneous polynomials, cf.~\cite{Belov01}.)

\subsubsection{Generic elements}

\begin{Definition}Assume that $K$ is an arbitrary field and $F\subseteq K$ is a
subfield. The set $\{\xi_1,\dots,\xi_k\}\subseteq K$ is called {\it
generic} (over $F$) if $f(\xi_1,\dots,\xi_k)\neq 0$ for any
commutative polynomial $f\in F[x_1,\dots,x_k]$ that takes nonzero
values.
\end{Definition}
\begin{Lemma}\label{many-gen-2r}
 Assume that $K$ has infinite transcendence degree over $F$.
 Then for any $k\in\N$ there exists a set of generic elements
 $\{\xi_1,\dots,\xi_k\}\subset K$.
\end{Lemma}
\begin{proof}
 $K$ has infinite transcendence degree over $F$. Therefore, there exists an element $\xi_1\in K\setminus \bar F$,
 where $\bar F$ is an algebraic closure of $F$.
 Now we consider $F_1=F[\xi_1]$.
 $K$ has infinite transcendence degree over $F$ and thus has infinite transcendence degree over $F_1$.
 Therefore there exists an element $\xi_2\in K\setminus \bar F_1$.
 And we consider the new base field $F_2=F_1[\xi_2]$.
 We can continue up to any natural number $k$.
\end{proof}

\begin{Remark} Note according to Lemma \ref{many-gen-2r} that if $K$ has infinite transcendence degree over~$F$
 we can take as many generic elements as we need. In particular we can take as many generic matrices as we need.
\end{Remark}
\begin{Lemma}\label{gen-real-2r}
 Assume $f\colon H\rightarrow\R$ $($where $H\subseteq\R^k$ is an open set in $k$-dimensional Euclidean space$)$
 is a function that is continuous in a neighborhood of the point $(y_1,\dots,y_k)\in
 H$, with
 $f(y_1,\dots,y_k)<q$.
 Let $c_i$ be real numbers $($in particular the coefficients of some
 polynomial~$p)$.
 Then there exists a set of elements $\{x_1,\dots,x_k\}\subseteq\R$ generic over $F=\Q[c_1,\dots,c_N]$ such that
 $(x_1,\dots,x_k)\in H$ and $f(x_1,\dots,x_k)<q$.
\end{Lemma}
\begin{proof}
The $\delta$-{\it neighborhood} $N_\delta(x)$ of $x\in\R$ denotes
the interval $(x-\delta,x+\delta)\subseteq\R$.
 Fix some small $\delta>0$ such that the product of $\delta$-neighborhoods of $y_k$ lays in $H$.
 For this particular $\delta$ we consider the $\delta$-neighborhood $N_\delta(y_1)$ of $y_1$:
 the interval $(y_1-\delta,y_1+\delta)$ is an uncountable set, and therefore there exists
 $x_1\in N_\delta(y_1)\setminus \bar F$. We consider $F_1=F[x_1]$ and analogically chose $x_2\in N_\delta (y_2)\setminus \bar F_1$
 and
 take $F_2=F_1[x_2]$. In such a way we can take generic elements $x_k\in N_\delta(y_k)$.
 Note that if $\delta$ is not sufficiently small $f(x_1,\dots,x_k)$ can be larger than $q$,
 but
 \begin{gather*}\mathop{f(x_1,\dots,x_k)\rightarrow f(y_1,\dots,y_k)}_{\delta\rightarrow 0}.\end{gather*}
 Thus there exists sufficiently small $\delta$ and generic elements $x_i\in N_\delta(y_i)$ such that
 $f(x_1,\dots,x_k)\allowbreak <q$.
\end{proof}

\begin{Remark} Note that $f$ can be a function defined on a set of matrices. In this case we consider it as a function
 defined on the matrix entries.
\end{Remark}

\begin{Remark}\label{patch1-int} Assume that $\Char (K)=0$. Suppose $t$ is a commuting indeterminate, and $f(x_1,
\dots, x_m;t)$ is a polynomial taking values under matrix
substitutions for the~$x_i$ and scalars for $t$. If there exists
unique $t_0$ such that $f(x_1, \dots, x_m;t_0)=0$, then $t_0$ is a
rational function with respect to the entries of $x_i$. If this
$t_0$ is fixed under simultaneous conjugation of the matrices
$x_1,\dots , x_m$, then $t_0$ is in the center of Amitsur's generic
division algebra~$\widetilde{\UD}$, implying $f \in
\widetilde{\UD}$. If $\Char(K)=k\neq 0$, then $t_0^{k^l}$ is a
rational function for some $l\in\mathbb{N}_0$. For details see
\cite[Remark~$2$]{BMR1}.
\end{Remark}
\begin{Remark}
In Remark \ref{patch1-int} we could take a system of polynomial
equations and polynomial inequalities. If $t_0$ is unique, then it
is a rational function (or $t_0^{k^l}$ if $\Char(K)=k$).
\end{Remark}

\subsubsection{Basic facts about cones}\label{def1-2}

Here is one of the main tools for our investigation.

\begin{Definition}
A {\it cone} of $M_n(K)$ is a subset closed under multiplication by
nonzero constants. An {\it invariant cone} is a cone invariant
under conjugation. An invariant cone is {\it irreducible} if it
does not contain any nonempty invariant cone.
\end{Definition}

\begin{Example}\label{coneex-2} Examples of invariant cones of $M_n(K)$ include:
\begin{enumerate}\itemsep=0pt\eroman \item The set of diagonalizable matrices.
 \item The set of non-diagonalizable matrices.
 \item The set $K$ of scalar matrices.
 \item The set of nilpotent matrices.
 \item The set $\sl_n$ of matrices having trace zero.
\end{enumerate}
\end{Example}

Let us continue with the following easy but crucial lemma.

\begin{Lemma}\label{cone_conj-2} Suppose the field $K$ is closed
under $d$-roots. If the image of a semi-homogeneous polynomial $p$
of weighted degree $d$ intersects an irreducible invariant cone $C$
nontrivially, then $C \subseteq \Image p$.
\end{Lemma}
\begin{proof}

If $A\in\Image p$ then $A=p(x_1,\dots,x_m)$ for some $x_i\in
M_n(K)$. Thus for any $c\in K$,
$cA=p\big(c^{w_1/d}x_1,c^{w_2/d}x_2,\dots,c^{i_m/d}x_m\big) \in\Image p$,
where $(w_1,\dots,w_m)$ are the weights. This shows that $\Image p$
is a cone.
\end{proof}
\begin{Remark} When the polynomial $p$ is multilinear, the image of any multilinear
polynomial is an invariant cone, without any assumption on $K$.
\end{Remark}

\subsection{Some general observations}

\begin{Lemma}\label{p1-3}
Let $K$ be an algebraically closed field of characteristic $0 $, and
let $I$ be an ideal of $K[X_1,\ldots,X_n]$, and $V(I) = \{x\in K^n\colon
f(x)=0\ \forall\, f\in I\}$. Let $\pi\colon K^n \to K^{n-1}$ be the
projection onto the first $n-1$ coordinates. Let $I'$ denote the
ideal $I\cap K[X_1,\ldots,X_{n-1}]$ of $K[X_1,\ldots, X_{n-1}]$.
Then:
\begin{enumerate}\itemsep=0pt
\item[$(1)$] $\pi(V(I))$ is a Zariski dense subset of $V(I')$;
\item[$(2)$] if there exists a Zariski dense subset $W$ of $V(I')$ such that
 the pre-image $\pi^{-1}(p)\cap V(I)$ of each point $p\in W$ consists of
one point, then there exists a rational $K$-valued function~$\phi$
on~$V(I')$ such that all points of a Zariski-dense subset of $V(I)$
have the form $(p,\phi(p))$ where $p\in V(I')$.
\end{enumerate}

If $\Char (K)=k>0$ then there exists non-negative integer $\ell$
such that $(p,a)\in V(I)$ satisfy $ \phi(p) = a^{k^\ell}$ on a
Zariski-dense subset of $V(I)$.
\end{Lemma}

\begin{proof}(1) is by \cite[Chapter~3, Section~2, Theorem~3]{CLO} and the subsequent
remarks.

To prove~(2), note that by (1) $\pi$ induces a field homomorphism
(hence an embedding) $K(V(I'))\to K(V(I))$ between the fields of
rational functions on the respective varieties. It is enough to show
that this is an isomorphism. Indeed, $K(V(I))$ is generated by~$K(V(I'))$ and~$X_n$. Moreover, $X_n$ is algebraic over $K(V(I'))$.
Let $h$ be the minimal polynomial of~$X_n$ over~$K(V(I'))$, of
degree $d$. The derivative~$h'$ has degree $d-1$, and the
discriminant $\Discr(h)$ is, up to a scalar, the resultant of~$h$
and~$h'$; it is non-zero since $h$ is irreducible implies that~$h$ and~$h'$ are relatively prime. Let now $U$ be the open subset of~$V(I')$ in which $\Discr(h)\ne 0$ and the coefficients of $h$ are
defined. Then each point of $U$ has precisely $d$ distinct
$\pi$-preimages in~$V(I)$. It follows that $d=1$, as required.

If $\Char(K)=k>0$, we take $\ell$ such that $h(x)=h_1\big(x^{k^\ell}\big)$
but $h_1'$ is not identically zero.
\end{proof}

\begin{Lemma}\label{divAm-3} Assume that $\Char(K)=0$. If an element $a$ of $\widetilde{\UD}$ has
a unique eigenvalue $\a$ $($i.e., of multiplicity $n)$, then $a$ is
scalar. If $\Char(K)=k\neq 0$ then $a$ is $k^l$-scalar for some~$l$.
\end{Lemma}
\begin{proof} If $\Char(K)=0$, then $\a$ is an element of $\widetilde{\UD}$ and $a -\a I$ is nilpotent, and thus~$0$.

If $\Char(K)=k$ then $\a^{k^l}$ is an element of $\widetilde{\UD}$,
therefore $a^{k^l} -\a^{k^l} I$ is nilpotent, and thus $0$. Thus $a$
is $k^l$-scalar. This is impossible if $k$ is not the divisor of the
size of the matrices $n$.
\end{proof}

\begin{Remark}\label{dim-f0-pcp} The variety of $n\times n$ matrices with a given set of $n$ distinct
eigenvalues has dimension $n^2-n$.
 \end{Remark}

\begin{Remark}\label{dim-f1-pcp} Assume for some matrix units $a_i$ that $p(a_1,\dots,a_m)$ is a diagonal matrix.
 Then~$f$ as constructed in \eqref{mapping_sets-pcp} in the proof of Theorem \ref{thmB1-pcp}
 is diagonal. If the dimension of~$\Image f$ is~$\delta$, then each evaluation
$M$ of $f$ has some set
 of eigenvalues, and (if the point is generic and the eigenvalues are distinct), then
 any matrix with this set of eigenvalues is similar to $M$ and therefore
 belongs to $\Image p$ .
Therefore by Remark~\ref{dim-f0-pcp}, $\Image p$ has dimension at
least $n^2-n+\delta$.
\end{Remark}

Other works on polynomial maps evaluated on matrix algebras include
\cite{GK,Wat}, who investigated maps that preserve zeros of
multilinear polynomials.

\section{The low rank case}\label{2q0}

Let $p$ be a multilinear polynomial in several non-commuting
variables with coefficients in a~quad\-ra\-tically closed field $K$ of
arbitrary characteristic. In this section we prove the
L'vov--Kaplansky conjecture for $n=2$ in several general cases, and
show that although the analogous assertion fails for completely
homogeneous polynomials, one can salvage most of the conjecture by
including the set of all non-nilpotent matrices of trace zero and
also permitting dense subsets of~$M_n(K)$.

\begin{Remark}For $n=2$, Donkin's theorem provides a polynomial function in
expressions of the form $\operatorname{Trace}\big(A,\bigwedge^2K^2\big)$ and
$\tr A$ where $A$ is monomial. Note that
$\operatorname{Trace}\big(A,\bigwedge^2K^2\big)=\det A$.
\end{Remark}

Next we introduce the cones of main interest to us, drawing from
Example~\ref{coneex-2}.

\begin{Example}\label{coneex1-2}\quad

\begin{enumerate}\itemsep=0pt\eroman
\item The set of nonzero nilpotent matrices comprise an
 irreducible invariant cone, since these all have the same
minimal and characteristic polynomial~$x^2$.
 \item The set of nonzero
 scalar
 matrices is an irreducible invariant cone.
 \item $\tilde K$ denotes the set of non-nilpotent,
non-diagonalizable matrices in~$M_2(K)$. Note that $A\in \tilde K$
precisely when $A$ is non-scalar, but with equal nonzero
eigenvalues, which is the case if and only if $A$ is the sum of a
nonzero scalar matrix with a nonzero nilpotent matrix. These are all
conjugate when the scalar part is the identity, i.e., for matrices
of the form \begin{gather*}\left(\begin{matrix}1 & a \\0 &
1\end{matrix}\right),\qquad a \ne 0\end{gather*} since these all have the same
minimal and characteristic polynomials, namely $x^2 - 2x+ 1$. It
follows that $\tilde K$ is an irreducible invariant cone.
 \item $\hat K$ denotes the set of non-nilpotent
 matrices in~$M_2(K)$ that have trace zero.

 When $\Char K \ne 2$, $\hat K$ is an irreducible invariant
cone,
 since any such matrix has distinct eigenvalues and thus is conjugate to
$\left(\begin{smallmatrix}\lambda & 0\\0 &
-\lambda\end{smallmatrix}\right)$.

 When $\Char K = 2$, $\hat K$ is an irreducible invariant cone,
 since any such matrix is conjugate to
$\left(\begin{smallmatrix}\lambda & 1\\0 & \lambda
\end{smallmatrix}\right)$.
 \item $\sl_2(K)\setminus \{ 0 \}$ is the union of the two irreducible invariant cones of (i) and~(iv). (The cases $\Char K \ne 2$ and
 $\Char K = 2$ are treated separately.)

 \item Let $C$ denote the set of nonzero matrices which are the sum of a scalar and a nilpotent
 matrix. Then $C$ is the union of the following
 three irreducible invariant cones: The nonzero scalar matrices,
 the nilpotent matrices, and the nonzero scalar multiples of
 non-identity
 unipotent matrices. (All non-identity unipotent matrices are conjugate.)
\end{enumerate}
\end{Example}

\subsection[$2\times 2$ matrices]{$\boldsymbol{2\times 2}$ matrices}
\subsubsection{The case of a quadratically closed field}\label{2q}

In this subsection we assume that $K$ is a quadratically closed
field. In particular, all of the eigenvalues of a matrix $A\in
M_2(K)$ lie in $K$. We start with the semi-homogeneous case (which
includes the completely homogeneous case), and finally give the
complete picture for the multilinear case.

\begin{Theorem}\label{imhom-2} Let $p(x_1,\dots,x_m)$ be a semi-homogeneous polynomial evaluated
on the algebra~$M_2(K)$ of $2\times 2$ matrices over a quadratically closed field. Then $\Image p$ is either $\{0\}$, $K$, the set of all non-nilpotent matrices having trace zero, $\sl_2(K)$, or a~dense subset of $M_2(K)$ $($with respect to Zariski topology$)$.
\end{Theorem}

(We also give examples to show how $p$ can have these images.)

\begin{Theorem}\label{main-2} If $p$ is a multilinear polynomial evaluated on the
matrix ring $M_2(K)$ $($where $K$ is a quadratically closed field$)$,
then $\Image p$ is either $\{0\}$, $K$, $\sl_2$, or~$M_2(K)$.
\end{Theorem}

Whereas one has a decisive answer for multilinear polynomials, the
situation is ambiguous for homogeneous polynomials, since, as we
shall see, certain invariant sets cannot occur as their images. For
the general non-homogeneous case, the image of a polynomial need not
be dense, even if it is non-central and takes on values of nonzero
trace, as we see in Example~\ref{nondense-2}.

One of our main ideas is to consider some invariant of the matrices
in $\Image(p)$, and study the corresponding invariant cones. Here is
the first such invariant that we consider.

\begin{Remark}\label{Phi-2}
Any non-nilpotent $2\times 2$ matrix $A$ over a quadratically closed
field has two eigenvalues $\lambda_1$ and $\lambda_2$, such that at
least one of which is nonzero. Therefore one can define the ratio of
eigenvalues, which is well-defined up to taking reciprocals:
$\frac{\lambda_1}{\lambda_2}$ and $\frac{\lambda_2}{\lambda_1}$.
Thus, we will say that two non-nilpotent matrices have {\it
different ratios} of eigenvalues if and only if their ratios of
eigenvalues are not equal nor reciprocal.

We do have a well-defined mapping $\Pi\colon M_2(K)\rightarrow K$
given by $A\mapsto
\frac{\lambda_1}{\lambda_2}+\frac{\lambda_2}{\lambda_1}$. This
mapping is algebraic because
\begin{gather*}\frac{\lambda_1}{\lambda_2}+\frac{\lambda_2}{\lambda_1}=-2+\frac{(\tr A)^2}{\det A}.\end{gather*}
\end{Remark}

\begin{Remark}\label{Phi2-2} The set of non-scalar
 diagonalizable matrices with a fixed nonzero ratio $r$ of eigenvalues (up to taking reciprocals) is an irreducible invariant cone. Indeed, this
is true since any such diagonalizable matrix is conjugate to
\begin{gather*}\lambda\left(\begin{matrix} 1 & 0\\0 & r\end{matrix}\right).\end{gather*}
\end{Remark}

{\bf Images of semi-homogeneous polynomials.}

\begin{Lemma}\label{linear1-2}
Suppose $K$ is closed under $d$-roots, as well as being
quadratically closed. If the image $\Image p$ of a
semi-homogeneous polynomial $p$ of weighted degree $d$ contains an
element of $\tilde K$, then $\Image p$ contains all of $\tilde K$.
\end{Lemma}
\begin{proof} This is clear from
Lemma~\ref{cone_conj-2} together with Example~\ref{coneex1-2}(iii),
since $\tilde K$ is an irreducible invariant cone.
\end{proof}
For the proof of Theorem~\ref{imhom-2} see Section~\ref{Deligne}.

We illuminate this result with some examples to show that certain
cones are obtained from specific completely homogeneous polynomials.
\begin{Example}\label{coneex2-2}\quad
\begin{enumerate}\itemsep=0pt\eroman \item
The polynomial $g(x_1, x_2) = [x_1, x_2]^2$ has the property that
$g(A,B) = 0$ whenever $A$ is scalar, but $g$ can take on a
nonzero value whenever $A$ is non-scalar.
 Thus, $g(x_1, x_2)x_1$ takes on all values except
scalars. This polynomial is completely homogeneous, but not
multilinear. (One can linearize in $x_2$ to make $g$ linear in
each variable except $x_1$, and the same idea can be applied to
Formanek's construction \cite{F1} of a central polynomial for any~$n$.)

\item Let $S$ be any finite subset of $K$. There exists a
completely homogeneous polynomial $p$ such that $\Image p$ is the
set of all $2\times 2$ matrices except the matrices with ratio of
eigenvalues from $S$. The construction is as follows. Consider
\begin{gather*}f(x)=x\cdot\prod_{\delta\in
S}(\lambda_1-\lambda_2\delta)(\lambda_2-\lambda_1\delta),\end{gather*} where
$\lambda_{1,2}$ are eigenvalues of $x$. For each $\delta$ the
product $(\lambda_1-\lambda_2\delta)(\lambda_2-\lambda_1\delta)$
is
a polynomial of $\tr x$ and $\tr x^2$. Thus $f(x)$ is a
polynomial with traces, and, by \cite[Theorem~1.4.12]{Row}), one can rewrite each trace in~$f$ as a fraction of
multilinear central polynomials.
After that we multiply the expression by the product of all the
denominators, which we can take to have value~1. We obtain a
completely homogeneous polynomial~$p$ which image is the cone under
$\Image f$ and thus equals~$\Image f$. The image of~$p$ is the set
of all non-nilpotent matrices with ratios of eigenvalues not
belonging to~$S$.

\item The image of a completely homogeneous polynomial evaluated
on $2\times 2$ matrices can also be $\hat K$. Take
$f(x,y)=[x,y]^3$. This is the product of $[x,y]^2$ and $[x,y]$.
$[x,y]^2$ is a central polynomial, and therefore $\tr f=0$.
However, there are no nonzero nilpotent matrices in $\Image p$ because if
$[A,B]^3$ is nilpotent then $[A,B]$ (which is a scalar multiple of
$[A,B]^3$) is nilpotent and therefore $[A,B]^2=0$ and $[A,B]^3=0$.

\item Consider the polynomial \begin{gather*}\qquad p(x_1,x_2,y_1,y_2) =
\big[(x_1x_2)^2,(y_1y_2)^2\big]^2 +
\big[(x_1x_2)^2,(y_1y_2)^2\big][x_1y_1,x_2y_2]^2.\end{gather*} Then $p$ takes on all
scalar values (since it becomes central by specializing $x_1
\mapsto x_2$ and $y_1 \mapsto y_2$), but also takes on all
nilpotent values, since specializing $x_1 \mapsto I + e_{12}$,
$x_2 \mapsto e_{22}$, and $y_1 \mapsto
 e_{12}$, and $y_2 \mapsto
e_{21} $ sends $p$~to \begin{gather*} \big[(e_{12}+e_{22})^2,
e_{11}^2\big]^2 + \big[(e_{12}+e_{22})^2, e_{11}^2\big][ e_{12}, e_{21}] =
0 - e_{12}( e_{11}- e_{22}) = e_{12}.\end{gather*}

We claim that $\Image p$ does not contain any matrix~$a= p(\bar
x_1,\bar x_2,\bar y_1,\bar y_2)$ in~$ \tilde K$. Otherwise, the
matrix
 $\big[(\bar x_1\bar x_2)^2,(\bar y_1\bar y_2)^2\big][\bar x_1\bar y_1,\bar x_2\bar y_2]^2$ would be the difference
of a matrix having equal eigenvalues and a scalar matrix, but of
trace 0, and so would have both eigenvalues 0 and thus be
nilpotent. Thus $\big[(\bar x_1\bar x_2)^2,(\bar y_1\bar y_2)^2\big]$
would also be nilpotent, implying the scalar term $\big[(\bar x_1\bar
x_2)^2,(\bar y_1\bar y_2)^2\big]^2$ equals zero, implying $a$ is
nilpotent, a contradiction.

$\Image p$ also contains all matrices having two distinct
eigenvalues. We conclude that $\Image p = M_2(K) \setminus \tilde
K$.
\end{enumerate}
\end{Example}

\begin{Remark}\label{Phi} In Example~\ref{coneex2-2}(iv),
The intersection $S$ of $\Image p$ with the discriminant surface is
defined by the polynomial \begin{gather*}\tr(p(x_1, \dots, x_m))^2- 4 \det(p(x_1,
\dots, x_m))= (\lambda_1 -\lambda_2)^2.\end{gather*} $S$ is the union of two
irreducible varieties (its scalar matrices and the nonzero nilpotent
matrices), and thus $S$ is a reducible variety. Thus, we see that
the discriminant surface of a~polynomial~$p$ of the algebra of
generic matrices can be reducible, even if it is not divisible by
any trace polynomial. Such an example could not exist for~$p$
multilinear, since then, by the same sort of argument as given in
the proof of Theorem~\ref{imhom-2}, the discriminant surface would
give a generic zero divisor in Amitsur's generic division algebra
$\widetilde{\UD}$ of Proposition~\ref{Am1-2}, a contradiction. In
fact, we will also see that the image of a multilinear polynomial
cannot be as in Example~\ref{coneex2-2}(iv).
\end{Remark}

{\bf Results for arbitrary polynomials.}

\begin{Lemma}\label{2two-2}
If $A,B\in \Image p$ have different ratios of eigenvalues,
then $\Image p$ contains matrices having arbitrary ratios of
eigenvalues $\frac{\lambda_1}{\lambda_2}\in K$.
\end{Lemma}
\begin{proof}
If $A = p(x_1,\dots,x_m)$, $B = p(y_1,\dots,y_m)\in \Image p$
have different ratios of eigenvalues, then we can lift the
$x_1,\dots,x_m$, $y_1,\dots,y_m$ to generic matrices, and then
$p(x_1,\dots,x_m)=\tilde A$ and $p(y_1,\dots,y_m)=\tilde B$ also
have different ratios of eigenvalues. Then take
\begin{gather*}f(T_1,T_2,\dots,T_m)=p(\tau_1 x_1+t_1y_1,\dots,\tau_m x_m+t_my_m),\end{gather*}
where $T_i=(t_i,\tau_i)\in K^2$. The polynomial $f$ is linear with
respect to all $T_i$.

In view of Remark~\ref{Phi}, it is enough to show that the ratio
$\frac{(\tr f)^2}{\det f}$ takes on all values. Fix
$T_1,\dots,T_{i-1},T_{i+1},\dots,T_m$ to be generic pairs where
$i$ is such that $\frac{(\tr f)^2}{\det f}$ is not constant with
respect to $T_i$. Such $i$ exist because otherwise all matrices in
the image (in particular, $A$ and $B$) have the same ratio of
eigenvalues.
 But
$\frac{(\tr f)^2}{\det f}$ is the ratio of quadratic polynomials,
and $K$ is quadratically closed.

 If there is a point $T_i$ such
that $\tr f=\det f=0$, then $f$ evaluated at this $T_i$ is
nilpotent. Since $\tr f$ is a linear function, the equation $\tr
f = 0$ has only one root, which is a rational function on the
other parameters. Thus $f$ evaluated at this $T_i$ is~0, by
Amitsur's theorem. We conclude that the ratio of eigenvalues does
not depend on $T_i$, contrary to our assumption on~$i$. Hence, we
can solve $\frac{(\tr f)^2}{\det f} =c$ for any $c \in K$.
\end{proof}

\begin{Lemma}\label{2general_one-2}
If there exist $\lambda_1\neq\pm\lambda_2$ with a collection of
matrices $(A_1,A_2,\dots,A_m)$ such that $p(A_1,A_2,\dots,A_m)$
has eigenvalues $\lambda_1$ and $\lambda_2$, then all
diagonalizable matrices lie in $\Image p$.
\end{Lemma}
\begin{proof}Applying Lemma~\ref{graph-2r} to the hypothesis, there is a matrix
\begin{gather*}\left(\begin{matrix}\lambda_1 &
0 \\0 & \lambda_2\end{matrix}\right)\in\Image p,\qquad
\lambda_1\neq\pm\lambda_2\end{gather*} which is an evaluation of $p$ on
matrix units $e_{ij}$. Consider the following mapping $\chi$
acting on the indices of the matrix units:
$\chi(e_{ij})=e_{3-i,3-j}$. Now take the polynomial
\begin{gather*}f(T_1,T_2,\dots,T_m)=p(\tau_1 x_1+t_1\chi(x_1),\dots,\tau_m x_m+t_m\chi(x_m)),\end{gather*}
where $T_i=(t_i,\tau_i)\in K^2$, which is linear with respect to
each $T_i$. Let us open the brackets. We obtain $2^m$ terms and
for each of them the degrees of all vertices stay even. (The edge
$12$ becomes $21$ which does not change degrees, and the edge $11$
becomes $22$, which decreases the degree of the vertex $1$ by two
and increases the degree of the vertex $2$ by two.) Thus all terms
remain diagonal. Consider generic pairs $T_1,\dots,T_m\in K^2$.
For each $i$ consider the polynomial $\tilde
f_i(T_i^*)=f(T_1,\dots,T_{i-1},T_i+T_i^*,T_{i+1},\dots,T_m)$. For
at least one $i$ the ratio of eigenvalues of $\tilde f_i$ must be
different from $\pm 1$. (Otherwise the ratio of eigenvalues of
$\tilde f_i$ equal $\pm 1$ all $i$, implying $\lambda_1= \pm
\lambda_2\}$, a contradiction.)

 Fix $i$ such that the ratio of eigenvalues of $\tilde f_i$ is not $\pm 1$.
 By linearity, $\Image\big(\tilde f_i\big)$ takes on values with all possible ratios of
 eigenvalues; hence, the cone under $\Image(\tilde f_i)$ is the set of all diagonal
matrices. Therefore by Lemma~\ref{cone_conj-2} all diagonalizable
matrices lie in the image of~$p$.
\end{proof}

{\bf Images of multilinear polynomials.}

\begin{Lemma}\label{thm1-2} If $p$ is a multilinear polynomial evaluated on the
matrix ring $M_2(K)$ over a~quad\-ra\-tically closed field $K$, then
$\Image p$ is either $\{0\}$, $K$, $\sl_2$, $M_2(K)$, or
$M_2(K)\setminus\tilde K$.
\end{Lemma}
\begin{proof} If $\Image p$ does not contain
a non-scalar matrix, then $p$ is either PI or central, and we are done. Hence, without loss of generality we can assume that $\Image p$ contains a~non-scalar matrix. By Remark~\ref{linear1-2} the linear span of
$\Image p$ is $\sl_2$ or $M_2(K)$. We treat the characteristic 2
and characteristic $\ne 2$ cases separately.

Case I: $\Char K=2$. Consider the set
\begin{gather*}\Theta= \{ p(e_1,\dots,e_m) \ \text{where the } e_j \text{ are matrix units}\}.\end{gather*}
 If
the linear span of the image is not $\sl_2$, then $\Theta$ contains
 at least one non-scalar diagonal matrix
$\diag\{\lambda_1,\lambda_2\}$, so $\lambda_1\neq -\lambda_2$
(since $+1 = -1$). Hence~by Lemma~\ref{2general_one-2}, all
diagonalizable matrices belong to $\Image p$. Thus, $\Image p$
contains $M_2(K)\setminus\tilde K$.

If the linear span of the image of $p$ is $\sl_2$, then by Lemma~\ref{graph-2r} the identity matrix (and thus all scalar matrices)
and $e_{12}$ (and thus all nilpotent matrices) belong to the
image. On the other hand, in characteristic 2, any matrix $\sl_2$
is conjugate to a matrix of the form $\lambda _1 I + \lambda_2
e_{1,2}$, and we consider the invariant $\frac
{\lambda_2}{\lambda_1}$. Take $x_1,\dots,x_m$ to be generic
matrices. If $p(x_1,\dots,x_m)$ were nilpotent then $\Image p$
would consist only of nilpotent matrices, which is impossible. By
Example~\ref{coneex1-2}(v), $p(x_1,\dots,x_m)$ is not scalar and not
nilpotent, and thus is a matrix from $\tilde K$. Hence, $\tilde K
\subset \Image p$, by Remark~\ref{linear1-2}. Thus, all trace zero
matrices belong to~$\Image p$.

Case II: $\Char K\neq 2$. Again assume that the image is not
$\{0\}$ or the set of scalar matrices. Then $e_{12}\in\Image p$ by Lemma~\ref{graph-2r}. Thus all nilpotent matrices
lie in $\Image p$. If the image consists only of matrices of trace
zero, then there is at least one matrix in
the image with a nonzero diagonal entry. By Lemma~\ref{graph-2r}
there is a set of matrix units that maps to a nonzero diagonal
matrix which, by assumption, is of trace zero and thus is
$\left(\begin{smallmatrix}c & 0
\\0 & -c\end{smallmatrix}\right)$. By Lemma \ref{cone_conj-2} and Example~\ref{coneex1-2}, $\Image p$ contains
 all trace zero $2\times 2$ matrices.

Assume that the image contains a matrix with nonzero trace. Then by Remark~\ref{linear1-2}
the linear span of the image is $M_2(K)$, and together with Lemma~\ref{graph-2r} we have at least two diagonal linearly independent
matrices in the image. Either these matrices have ratios of
eigenvalues $(\lambda_1:\lambda_2)$ and $(\lambda_2:\lambda_1)$
for $\lambda_1\neq\pm \lambda_2$ or these matrices have
non-equivalent ratios. In the first case we can use Lemma~\ref{2general_one-2} which says that all diagonalizable matrices lie
in the image. If at least one of these matrices have ratio not
equal to $\pm 1$, then in the second case we also use Lemma~\ref{2general_one-2} and obtain that all diagonalizable matrices lie
in the image. If these matrices are such that the ratios of their
eigenvalues are respectively $1$ and $-1$, then we use Lemma~\ref{2two-2} and obtain that all diagonalizable matrices with
distinct eigenvalues lie in the image. By assumption, in this
case, scalar matrices also belong to the image. Therefore we
obtain that for any ratio $(\lambda_1:\lambda_2)$ there is a
matrix $A\in\Image p$ having such a ratio of eigenvalues. Using
Lemmas~\ref{cone_conj-2} and~\ref{linear1-2}, we obtain that the image
of $p$ can be either $\{0\}$, $K$, $\sl_2$, $M_2(K)$, or $M_2(K)\setminus\tilde K$.
\end{proof}
\begin{Lemma}\label{char2-2} If $p$ is a multilinear polynomial evaluated on the
matrix ring $M_2(K)$, where $K$ is a quadratically closed field of
characteristic $2$, then $\Image p$ is either $\{0\}$, $K$,
$\sl_2$, or $M_2(K)$.
\end{Lemma}
\begin{proof}In view of Lemma \ref{thm1-2}, it suffices to assume that the image
of $p$ is $M_2(K)\setminus\tilde K$. Let
$x_1,\dots,x_m$, $y_1,\dots,y_m$ be generic matrices. Consider the
polynomials
\begin{gather*}b_i=p(x_1,\dots,x_{i-1},y_i,x_{i+1},\dots,x_m).\end{gather*} Let
$p_i(x_1,\dots,x_m,y_i)=p\tr(b_i)+\tr(p)b_i$. Hence $p_i$ can be
written as
\begin{gather*}p_i=p(x_1,\dots,x_{i-1},x_i\tr(b_i)+y_i\tr(p),x_{i+1},\dots,x_m).\end{gather*}
Therefore $\Image p_i\subseteq\Image p$. Also if $a\in\Image p_i$,
then
\begin{gather*}\tr(a)=\tr(p\,\tr(b_i)+\tr(p)\,b_i)=2\,\tr(p)\,\tr(b_i)=0.\end{gather*}
Thus, $\Image p_i$ consists only of trace-zero matrices which belong
to the image of $p$. Excluding~$\tilde K$, the only trace zero
matrices are nilpotent or scalar. Thus, for each~$i$,
$p_i(x_1,\dots,x_m,y_i)$ is either scalar or nilpotent. However, the
$p_i$ are the elements of the algebra of generic matrices with
traces, which is a domain. Thus, $p_i(x_1,\dots,x_m,y_i)$ cannot be
nilpotent. Hence for all $i=1,\ldots, m$, $p_i(x_1,\dots,x_m,y_i)$
is scalar. In this case, changing variables leaves the plane
$\langle p,I\rangle$ invariant. Therefore, $\dim (\Image p) = 2$, a
contradiction.
\end{proof}
\begin{Lemma}\label{sklem}\label{char_n_2} If $p$ is a multilinear polynomial evaluated on
the matrix ring $M_2(K)$ $($where $K$ is a quadratically closed
field of characteristic not $2)$, then $\Image p$ is either
$\{0\}$, $K$, $\sl_2$, or $M_2(K)$.
\end{Lemma}
The proof of Lemma \ref{sklem} is given in Section~\ref{Deligne}.

Finally, Theorem \ref{main-2} follows from Lemmas~\ref{char2-2} and~\ref{char_n_2}.

\subsubsection{Images of arbitrary non-homogeneous polynomials}

We consider briefly the general non-homogeneous case. One can write any
polynomial $p(x_1,\dots,\allowbreak x_m)$ as $p=h_k+\dots+h_n$, where the~$h_i$ are semi-homogeneous polynomials of weighted degree~$i$.

\begin{Proposition} Notation as above,
assume that there are weights $(w_1,\dots,w_m)$ that $\Image {h_n}$ is dense in $M_2(K)$. Then $\Image p$ is
dense in $M_2(K)$.
\end{Proposition}

\begin{proof} Consider
\begin{gather*}p\big(\lambda^w_1x_1,\dots,\lambda^w_mx_m\big)=\sum\limits_{i=k}^n
h_i\lambda^i.\end{gather*} One can write $\tilde
P=\lambda^{-n}p\big(\lambda^w_1x_1,\dots,\lambda^w_mx_m\big)$ as a
polynomial in $x_1,\dots,x_m$ and $\varepsilon=\frac{1}{\lambda}$.
The matrix polynomial is the set of four polynomials
$p_{1,1}$, $p_{1,2}$, $p_{2,1}$, $p_{2,2}$, which we claim are independent.
If there is some polynomial $h$ in four variables such that
$h(p_{1,1},p_{1,2},p_{2,1},p_{2,2})=0$ then $h$ should vanish on
four polynomials of $\tilde P$ for each~$\varepsilon$, in
particular for $\varepsilon=0$, a contradiction.
\end{proof}

\begin{Remark} The case remains open where
 $p(x_1,\dots,x_m)$ is a polynomial for which there are no
weights $(w_1,\dots,w_m)$ such that one can write
$p=h_k+\dots+h_n$, where $h_i$ is semi-homogeneous of weighted
degree $i$ and $h_n$ has dense image in~$M_2$.\end{Remark}

\begin{Example}\label{nondense-2} For $\Char K\neq 2$ we give an example of such a polynomial
whose middle term has dense image in~$M_2(K)$. Take the polynomial
\begin{gather*}f(x,y)=[x,y]+[x,y]^2.\end{gather*} It is not hard to check that $\Image f$ is
the set of all matrices with eigenvalues $c^2+c$ and $c^2-c$.
Consider
$p(\alpha_1,\alpha_2,\beta_1,\beta_2)=f\big(\alpha_1+\beta_1^2,\alpha_2+\beta_2^2\big)$.
The polynomials $f$ and $p$ have the same images. Now let us open
the brackets. The term of degree 4 is
$h_4=[\alpha_1,\alpha_2]^2+\big[\beta_1^2,\beta_2^2\big]$. The image of
$h_4$ is all of $M_2(K)$, because $[\alpha_1,\alpha_2]^2$ can be
any scalar matrix and $\big[\beta_1^2,\beta_2^2\big]$ can be any trace
zero matrix. However the image of~$p$ is the set of all matrices
with eigenvalues $c^2+c$ and $c^2-c$.
\end{Example}

\subsubsection{The cases of an arbitrary real field}\label{2r}

Let $p$ be a multilinear polynomial in several non-commuting
variables with coefficients in an arbitrary field~$K$. Kaplansky
conjectured that for any $n$, the image of $p$ evaluated on the set~$M_n(K)$ of $n$ by $n$ matrices is either zero, or the set of scalar matrices, or the set~$\sl_n(K)$ of matrices of trace~$0$, or all of~$M_n(K)$. This conjecture was proved for $n=2$ when~$K$ is closed under quadratic extensions. In this section the conjecture is
verified for $K=\mathbb{R}$ and $n=2$, also for semi-homogeneous
polynomials~$p$, with a partial solution for an arbitrary field~$K$.

In Section~\ref{2q} the field $K$ was required to be
quadratically closed. Even for the field $\mathbb{R}$ of real
numbers L'vov--Kaplansky's question remained open, leading people
to ask what happens if the field is not quadratically closed? This
subsection provides a positive partial answer.

\subsubsection[Images of multilinear polynomials evaluated on $M_2(K)$]{Images of multilinear polynomials evaluated on $\boldsymbol{M_2(K)}$}\label{im-of-pol-2r}

Assume that $p$ is a multilinear polynomial evaluated on $2\times 2$
matrices over any field $K$. Assume also that $p$ is neither PI nor
central. Then, by Lemma~\ref{graph-2r} there exist matrix units $a_1,\dots,a_m$ such that
$p(a_1,\dots,a_m)=e_{12}$. Let us consider the mapping $\chi$
defined on matrix units that switches the indices~$1$ and $2$, i.e.,
$e_{11}\leftrightarrow e_{22}$ and $e_{12}\leftrightarrow e_{21}$.
Now let us consider the mapping $f$ defined on $m$ pairs
$T_i=(t_i,\tau_i):$
\begin{gather*}f(T_1,\dots,T_m)=p(t_1a_1+\tau_1\chi(a_1),t_2a_2+\tau_2\chi(a_2),\dots,t_ma_m+\tau_m\chi(a_m)).\end{gather*}
Now let us open the brackets. We show in the
proof of Theorem~\ref{thmB1-pcp} that any matrix of the image of $f$ can be
written as $c_1e_{12}+\dots+c_{n-1}e_{n-1,n}+c_ne_{n,n-1}$. In our
case $n=2$ and the image of $f$ contains only matrices of the type
$c_1e_{12}+c_2e_{21}$. Note that the matrices~$e_{12}$ and~$e_{21}$ both belong to the image of~$f$ since
$p(a_1,\dots,a_m)=e_{12}$ and
$p(\chi(a_1),\dots,\chi(a_m))=e_{21}$. According to Lemma~\ref{dim2-2r}
 the image of~$f$ is at least $2$-dimensional, and lies in the $2$-dimensional plane~$\langle e_{12}, e_{21}\rangle$. Therefore this plane is exactly the image of~$f$. Now we are ready to prove the following:
\begin{Lemma}\label{genfield-2r}
 If $p$ is a multilinear polynomial evaluated on the
 matrix ring $M_2(K)$ $($for an arbitrary field~$K)$,
 then $\Image p$ is either $\{0\}$ or~$K$, or $\ssl_2\setminus K\subseteq\Image p$.
\end{Lemma}
\begin{proof}
 Let $A$ be any trace zero, non-scalar matrix. Take any vector $v_1$ that is not an eigenvector of $A$.
 Consider the vector $v_2=Av_1$. Note that \begin{gather*}Av_2=A^2v_1=-\det(A)v_1,\end{gather*} and therefore the matrix $A$ with respect to
 the base $\{v_1,v_2\}$ has the form $c_1e_{12}+c_2e_{21}$, for some~$c_i$. Hence $A$ is similar to $c_1e_{12}+c_2e_{21} \in\Image p$, implying $A\in\Image p$.
\end{proof}
\begin{Remark}\label{chr-n2-2r}
 Note that for $\Char(K)\neq 2$ (in particular for $K=\mathbb{R}$), \begin{gather*}(\ssl_2\setminus K) \cup\{0\}=
 \ssl_2\subseteq\Image p.\end{gather*}
\end{Remark}

{\bf The real case.}
Now we assume that $K=\mathbb{R}$. We already know that either $p$ is PI, or central, or
$\ssl_2\subseteq\Image p$. Assume that $\ssl_2\subsetneqq\Image p$.
We will use the following lemma:
\begin{Lemma}\label{ineq-2r}
 Let $p$ be any multilinear polynomial satisfying $\ssl_2\subsetneqq\Image p$.
 For any $q\in\mathbb{R}$ there exist generic matrices $x_1,\dots, x_m,y_1,\dots,y_m$ such that
 for $X=p(x_1,\dots,x_m)$ and $Y=p(y_1,\dots,y_m)$ we have the following:
 \begin{gather*}\frac{\det X}{\tr^2 X}\leq q\leq \frac{\det Y}{\tr^2 Y},\end{gather*}
 where $\tr^2 M$ denotes the square of the trace of $M$.
\end{Lemma}
\begin{proof} We know that $\ssl_2\subseteq\Image p$, in particular for the matrices
 $\Omega=e_{11}-e_{22}$ and $\Upsilon=e_{12}-e_{21}$
 there exist matrices $a_1,\dots,a_m,b_1,\dots,b_m$ such that
 $p(a_1,\dots,a_m)=\Omega$ and
 $p(b_1,\dots,b_m)=\Upsilon$.
 Note $\frac{\det M}{\tr^2 M}\leq q$ if $M$ is close to $\Omega$ and
 $\frac{\det M}{\tr^2 M}>q$ if $M$ is close to $\Upsilon$.
 Now we consider a very small $\delta>0$ such that for any matrices $x_i\in N_\delta(a_i)$ and $y_i\in N_\delta(b_i)$
 \begin{gather*}\frac{\det X}{\tr^2 X}\leq q\leq \frac{\det Y}{\tr^2 Y},\end{gather*} where
 $X=p(x_1,\dots,x_m)$ and $Y=p(y_1,\dots,y_m)$.
 Here by $N_\delta(x)$ we denote a $\delta$-neighborhood of $x$, under the
 max norm $\Arrowvert A \Arrowvert=\max\limits_{i,j} \arrowvert a_{ij} \arrowvert$.
 According to Lemma~\ref{gen-real-2r} one can choose generic matrices with such property.
\end{proof}

 Now we are ready to prove that the image of $g(x_1,\dots,x_m)=\frac{\det p}{\tr^2 p}$ is all of $\mathbb{R}$:
 \begin{Lemma}\label{anyq-2r}
 Let $p$ be any multilinear polynomial satisfying $\ssl_2\subsetneqq\Image p$.
 Then for any $q\in\mathbb{R}$ there exists a set of matrices $a_1,\dots, a_m$ such that
 \begin{gather}\label{eq-2r}
\frac{\det p(a_1,\dots,a_m)}{\tr^2 p(a_1,\dots,a_m)}= q.
 \end{gather}
 \end{Lemma}
\begin{proof}
 Let $q$ be any real number. According to Lemma \ref{ineq-2r}
 there exist generic matrices $x_1,\dots, x_m$, $y_1,\dots,y_m$ such that
 for $X=p(x_1,\dots,x_m)$ and $Y=p(y_1,\dots,y_m)$ we have the following:
 \begin{gather*}\frac{\det X}{\tr^2 X}\leq q\leq \frac{\det Y}{\tr^2 Y}.\end{gather*}
 Consider the following matrices:
 $A_0=p(\tilde x_1,x_2, \dots,x_m)$, where
 $\tilde x_1$ is either $x_1$ or $-x_1$, such that $\tr A_0>0$.
 $A_1=p(\tilde y_1,x_2,\dots,x_m)$, where
 $\tilde y_1$ is either $y_1$ or $-y_1$ such that $\tr A_1>0$.
 Assume that $A_i$, $\tilde x_1$, $\tilde y_1,\dots,\tilde y_i$ are defined.
 Let \begin{gather*}A_{i+1}=p(\tilde y_1,\dots,\tilde y_i,\tilde y_{i+1},x_{i+2},\dots,x_m),\end{gather*} where
 $\tilde y_{i+1}=\pm y_{i+1}$ is such that $\tr A_{i+1}>0$.
 In such a way we defined matrices $A_i$ for $0\leq i\leq m$.
 Note that for any $2\times 2$ matrix~$M$,
 \begin{gather*}\frac{\det M}{\tr^2 M}=\frac{\det (-M)}{\tr^2 (-M)}.\end{gather*}
 Note that $A_0=\pm p(x_1,\dots,x_m)$ and $A_m=\pm p(y_1,\dots, y_m);$ hence
 \begin{gather*}\frac{\det A_0}{\tr^2 A_0}\leq q\leq \frac{\det A_m}{\tr^2 A_m}.\end{gather*}
 Therefore there exists $i$ such that
 \begin{gather*}\frac{\det A_i}{\tr^2 A_i}\leq q\leq \frac{\det A_{i+1}}{\tr^2 A_{i+1}}.\end{gather*}
Since
 $A_{i}=p(\tilde y_1,\dots,\tilde y_i,x_{i+1},x_{i+2},\dots,x_m)$ and
 $A_{i+1}=p(\tilde y_1,\dots,\tilde y_{i+1},x_{i+2},\dots,x_m)$,
 we can consider the matrix function
 \begin{gather*}M(t)=(1-t)A_i+tA_{i+1}=p(\tilde y_1,\dots,\tilde y_i,(1-t)x_{i+1}+t\tilde y_{i+1},x_{i+2},\dots,x_m).\end{gather*}
 Then $\Image M\subseteq\Image p$, $M(0)=A_i$, $M(1)=A_{i+1}$ both $M(0)$ and $M(1)$ have positive trace, and
 $M$ is an affine function. Therefore for any $t\in [0,1]$ $M(t)$ has positive trace.
 Therefore the function $\psi(t)=\frac{\det M(t)}{\tr^2 M(t)}$ is well defined on $[0,1]$ and continuous.
 Also we have $\psi(0)\leq q\leq \psi(1)$. Thus there exists $\tau\in [0,1]$ such that $\psi(\tau)=q$ and thus
 $M(\tau)\in\Image p$ satisfies equation~\eqref{eq-2r}.
\end{proof}
\begin{Lemma}\label{discr-not-zero-2r}
 Let $p$ be a multilinear polynomial satisfying $\ssl_2\subsetneqq\Image p$.
 Then any matrix with distinct eigenvalues $($i.e., matrix of nonzero discriminant$)$ belongs to $\Image p$.
\end{Lemma}
\begin{proof} Let $A$ be any matrix with nonzero discriminant. Let us show that $A\in\Image p$.
 Let $q=\frac{\det A}{\tr^2 A}$.
 According to Lemma~\ref{anyq-2r} there exists a set of matrices $a_1,\dots, a_m$ such that
 $\frac{\det \tilde A}{\tr^2 \tilde A}= q$, where $\tilde A=p(a_1,\dots,a_m)$.
 Take $c\in\R$ such that $\tr\big(c\tilde A\big)=\tr A$. Note $c\tilde A=p(ca_1,a_2,\dots,a_m)$ belongs to $\Image p$.
 Thus \begin{gather*}\frac{\det \big(c\tilde A\big)}{\tr^2 \big(c\tilde A\big)}=q=\frac{\det A}{\tr^2 A},\end{gather*} and $\tr A=\tr\big(c\tilde A)$.
 Hence, $\det\big(c\tilde A\big)=\det(A)$. Therefore the matrices $c\tilde A$ and $A$ are similar since
 they are not from the discriminant surface. Therefore $A\in\Image p$.
\end{proof}
\begin{Lemma}\label{unip-2r} Let $p$ be a multilinear polynomial satisfying $\ssl_2\subsetneqq\Image p$. Then any non-scalar matrix with zero discriminant belongs to $\Image p$.
\end{Lemma}
\begin{proof} Let $A$ be any non-scalar matrix with zero discriminant. Let us show that $A\in\Image p$.
 The eigenvalues of $A$ are equal, and therefore they must be real.
 Thus $A$ is similar to the matrix
 $\tilde A = \left(
 \begin{smallmatrix}
 \lambda & 1
 \\ 0 & \lambda
 \end{smallmatrix}
 \right)
 $. If $A$ is nilpotent then $\lambda=0$ and $\tilde A=e_{12}$, and it belongs to $\Image p$ by Lemma~\ref{graph-2r}.
 If $A$ is not nilpotent then we need to show that at least one non-nilpotent matrix of such type belongs to $\Image p$, and
 all other are similar to it.
 We know that the matrices $e_{11}-e_{22}=p(a_1,\dots,a_m)$ and $e_{12}-e_{21}=p(b_1,\dots,b_m)$ for some $a_i$ and $b_i$.
 Note that $e_{11}-e_{22}$ has positive discriminant and $e_{12}-e_{21}$ has negative discriminant.
 Take generic matrices $x_1,x_2,\dots,x_m$, $y_1,\dots,y_m$ such that $x_i\in N_\delta(a_i)$ and $y_i\in N_\delta(b_i)$ where
 $\delta>0$ is so small that $p(x_1,\dots,x_m)$ has positive discriminant and $p(y_1,\dots,y_m)$ has negative discriminant.
 Consider the following matrices:
 \begin{gather*}A_0=p(x_1,x_2, \dots,x_m) ,\qquad A_i=p(y_1,\dots,y_i, x_{i+1},\dots,x_m),\qquad 1 \le i \le m.\end{gather*}
We know that $\discr A_0>0$ and $\discr A_m<0$, and therefore there exists $i$ such that $\discr A_i>0$ and $\discr A_{i+1}<0$.
 We can consider the continuous matrix function
 \begin{gather*}M(t)=(1-t)A_i+tA_{i+1}=p(y_1,\dots,y_i,(1-t)x_{i+1}+ty_{i+1},x_{i+2},\dots,x_m).\end{gather*}
 We know that $M(0)$ has positive discriminant and $M(1)$ has negative discriminant.
 Therefore for some $\tau$, $M(\tau)$ has discriminant zero.
 Assume there exists $t$ such that $M(t)$ is nilpotent.
 In this case either $t$ is unique or there exists $t'\neq t$ such that $M(t')$ is also nilpotent.
 If $t$ is unique then it equals to some rational function with respect to other variables (entries of matrices $x_i$ and $y_i$).
 In this case $t$ can be considered as a function on matrices $x_i$ and $y_i$ and as soon as it is invariant, according to
 the Proposition~\ref{procesi-2r} $t$ is an element of $\widetilde{\UD}$
 and thus
 $M(t)$ is the element of $\widetilde{\UD}$. Therefore $M(t)$ cannot be nilpotent since UD is a domain according to Remark~\ref{Am1-2}.
 If there exists $t'\neq t$ such that $M(t')$ is also nilpotent then for any $\tilde t\in \R$ $M(\tilde t)$ is the combination
 of two nilpotent (and thus trace vanishing) matrices $M(t)$ and $M(t')$. Hence $M(0)$ is trace vanishing and thus
 $\Image p\subseteq\ssl_2$, a contradiction.

 Recall that we proved $M(\tau)$ has discriminant zero that for some $\tau$. Note that $M(\tau)$ cannot be nilpotent.
 Assume that the matrix $M(\tau)$ is scalar.
 Hence $(1-\tau)A_i+\tau A_{i+1}=\lambda I$ where $\lambda\in\R$ and $I$ is the identity matrix.
 Thus,
 $A_{i+1}=\frac{1-\tau}{\tau}A_i+cI$.
 Note that for any matrix $M$ and any $c\in\R$ we have $\discr(M)=\discr(M+cI)$.
 Therefore the discriminant of $A_{i+1}$ can be written as
 \begin{gather*}\discr(A_{i+1})=\discr\left(\frac{1-\tau}{\tau}A_i\right)=\left(\frac{1-\tau}{\tau}\right)^2\discr(A_i),\end{gather*}
 a contradiction, since $\discr A_i>0$ and $\discr(A_{i+1})<0$.
 Therefore the matrix $M(\tau)$ is similar to $A$.
\end{proof}
\begin{Lemma}\label{scalar-2r}
 Let $p$ be a multilinear polynomial satisfying $\ssl_2\subsetneqq\Image p$.
 Then every scalar matrix belongs to $\Image p$.
\end{Lemma}
\begin{proof}
 Note that it is enough to show that at least one scalar matrix belong to the image of $p$.
 According to Lemmas \ref{graph-2r} and Remark~\ref{linear1-2} there are matrix units $a_1,\dots,a_m$ such that $p(a_1,\dots,a_m)$ is
 diagonal with nonzero trace.
 Assume that it is not scalar, i.e., $p(a_1,\dots,a_m)=\lambda_1e_{11}+\lambda_2e_{22}$.
 We define again the mapping $\chi$ and $f(T_1,\dots,T_m)$ as in the beginning of Section~\ref{im-of-pol-2r} and return to the
 proof of Lemma~\ref{2general_one-2} where we proved that $\Image f$ consists only of diagonal matrices or only of matrices with zeros
 on the diagonal.
 In our case the image of $f$ consists only of diagonal matrices, which is a $2$-dimensional variety.
 We know that both $p(a_1,\dots,a_m)=\lambda_1e_{11}+\lambda_2e_{22}$ and
 $p(\chi(a_1),\dots,\chi(a_m))=\lambda_1e_{22}+\lambda_2e_{11}$ belong to the image of $f$, and
 therefore every diagonal matrix belong
 to the image of $f$, in particular every scalar matrix.
\end{proof}

Now we are ready to prove a major result.

\begin{proof}[Proof of Theorem~\ref{main-2r}]
 The second part follows from Lemmas \ref{genfield-2r}, \ref{discr-not-zero-2r}, \ref{unip-2r} and \ref{scalar-2r}.
 In the first part we need to prove that if $p$ is neither PI nor central then $\ssl_2(K)\subseteq\Image p$.
 According to Lemma \ref{genfield-2r}, $\ssl_2(K)\setminus K\subseteq\Image p$, and therefore according to Remark \ref{chr-n2-2r} we need consider only the case $\Char(K)=2$.
Then we need to prove that the scalar matrices belong to the image of $p$.
By Lemma~\ref{graph-2r} and Remark~\ref{linear1-2} there are matrix units $a_1,\dots,a_m$ such that $p(a_1,\dots,a_m)$ is diagonal.
 Assume that it is not scalar. Then we consider the mappings $\chi$ and $f$ as described in the beginning of Section~\ref{im-of-pol-2r}.
 According to Lemma~\ref{dim2-2r} the image of $f$ will be the set of all diagonal matrices, and in particular the scalar matrices belong to it.
\end{proof}
\begin{Remark}
 Assume that $p$ is a multilinear polynomial evaluated on $2\times 2$ matrices over an arbitrary infinite field $K$. Then, according to
 Theorem \ref{main-2r}, $\Image p$ is $\{0\}$, or $K$, or $\ssl_2(K)$ or $\ssl_2(K)\subsetneqq\Image p$. In the last case it is clear that
 $\Image p$ must be Zariski dense in~$M_2(K)$, because otherwise $\dim(\Image p)=3$ and $\Image p$ is reducible, a contradiction.
\end{Remark}

\begin{Remark} Note that the proof of Theorem \ref{main-2r} does not work when $n>2$ since for this case we will need to take more than one function (two functions for $n=3$ and more for $n>3$). In our proof we used that we have only one function: we proved that it takes values close to~$\pm\infty$ and after that used continuity. This does not work for $n\geq 3$. However one can use this idea for the question of possible images of trace vanishing multilinear polynomials evaluated on $3\times 3$ matrices. In this case one function will be enough, and one can take $g=\frac{\omega_3^2}{\omega_2^3}$. (One can find the definitions of~$\omega_i$ in the proof of Theorem \ref{semi_tr0_3-3}.) Moreover by Lemma \ref{graph-2r} there are matrix units~$a_i$ such that $p(a_1,\dots,a_m)$ is a~diagonal, trace vanishing, nonzero real matrix, which cannot be $3$-scalar since it will have three real eigenvalues. Therefore $p$ cannot be $3$-central polynomial. However the question of possible images of~$p$ remains open.
\end{Remark}

{\bf Images of semi-homogeneous polynomials evaluated on $\boldsymbol{2\times 2}$ matrices with real entries.}
Here we provide a classification of the possible images of semi-homogeneous
polynomials evaluated on $2\times 2$ matrices with real entries. Let
us start with the definitions.

\begin{Definition} A {\it semi-cone} of $M_n(\R)$ is a subset closed under multiplication by positive
constants. An {\it invariant semi-cone} is a semi-cone invariant
under conjugation. An invariant semi-cone is {\it irreducible} if
it does not contain any nonempty invariant semi-cone.
\end{Definition}

\begin{Remark}\label{semcone-2r}
Let $p$ be any semi-homogeneous polynomial of weighted degree $d\neq
0$ with weights $(w_1,\dots,w_m)$. Thus if $A=p(x_1,\dots,x_m)$ then
for any $c\in\R$ we have $p\big(c^{w_1}x_1,{\dots},c^{w_m}x_m\big)\allowbreak =c^dA$.
Hence $\Image p$ is a~semi-cone, for any~$d$.
\end{Remark}

\begin{proof}[Proof of Theorem~\ref{homogen-int}]
 Consider the function $g(x_1,\dots,x_m)=\frac{\det p}{\tr^2 p}$.
 If this function is not constant, then $\Image p$ is Zariski dense.
 Assume that it is constant; i.e., $\frac{\det p}{\tr^2 p}=c$.
 Then the ratio $\frac{\lambda_1}{\lambda_2}=\hat c$ of eigenvalues is also a constant.
 If $\hat c\neq -1$ then we can write $\lambda_1$ explicitly as
 \begin{gather*}\lambda_1=\frac{\lambda_1}{\lambda_1+\lambda_2}\tr p=\frac{1}{1+\frac{\lambda_2}{\lambda_1}}\tr p
 =\frac{1}{1+\frac{1}{\hat c}}\tr p,\end{gather*}
 Therefore $\lambda_1$ is an element of $\widetilde{\UD}$, and $\lambda_2=\tr p-\lambda_1$ also.
 According to the Hamilton--Cayley equation, $(p-\lambda_1)(p-\lambda_2)=0$ and therefore, since, by Remark~\ref{Am1-2},
 $\widetilde{\UD}$ is a domain,
 one of the terms $p-\lambda_i$ is a PI. Therefore $p$ is central or PI.
 Therefore we see that any semi-homogeneous polynomial is either PI, or central, or trace vanishing (if the ratio of eigenvalues is
 $-1$ then the trace is identically zero), or $\Image p$ is Zariski dense.
 If $p$ is PI then $\Image p=\{0\}$.
 If $p$ is central then, by Remark~\ref{semcone-2r}, $\Image p$ is a semi-cone, therefore
 $\Image p$ is either $\R_{\geq 0}$, or $\R_{\leq 0}$, or $\R$.
 If $p$ is trace vanishing, then any trace zero matrix $A\in\ssl_2(\R)$ is similar to $-A$.
 Therefore $\Image p=-\Image p$ is symmetric.
 Together with Remark \ref{semcone-2r} we have that $\Image p$ must be a cone.
 The determinant cannot be identically zero since otherwise the polynomial is nilpotent, contrary to Remark~\ref{Am1-2}.
 Hence there exists some value with nonzero determinant.
 All the trace zero matrices of positive determinant are pairwise similar, and
 all the trace zero matrices of negative determinant are pairwise similar.
 Therefore in this case all possible images of $p$ are $\ssl_{2,\geq0}(\R)$, $\ssl_{2,\leq0}(\R)$ and~$\ssl_{2}(\R)$.
\end{proof}

\begin{Example} $\Image p$ can be the set of non-negative scalars. Take any central polynomial, say $p(x,y)=[x,y]^2$ and consider
 $p^2=[x,y]^4$. If one takes $-p^2=-[x,y]^4$, then its image
 is the set~$\R_{\leq 0}$.

 The question remains open of whether or not there exists an example of a trace zero polynomial with non-negative (or non-positive) discriminant.

 There are many polynomials with Zariski dense image which are not dense with respect to the usual Euclidean topology.
 For example the image of the polynomial $p(x)=x^2$ is the set of matrices with two positive eigenvalues,
 or two complex conjugate eigenvalues; in particular any matrix $x^2$ has non-negative determinant.
The image of the polynomial $p(x,y)=[x,y]^4+\big[x^4,y^4\big]$ is the set
of matrices with non-negative trace.
 The question of classifying possible semi-homogeneous Zariski dense images
 remains open.
\end{Example}

\subsubsection[Multilinear polynomials evaluated on $\HH$]{Multilinear polynomials evaluated on $\boldsymbol{\HH}$}\label{quat}

The L'vov--Kaplansky conjecture can fail for a non-simple finite
dimensional algebra. In particular, it fails for the Grassmann
algebra over a linear space of finite dimension more than or equal
to $4$, and a field of characteristic $\neq 2$, which is a finite
dimensional (but not simple) associative algebra. In this case one
can consider the multilinear polynomial $p(x,y)=x\wedge y-y\wedge
x$. Then $e_1\wedge e_2=p(1/2e_1,e_2)$ and $e_3\wedge e_4$ both
belong to the image of $p$, but their sum does not. Thus the image
is not a vector space.

Recently, it has been conjectured that the evaluation of any
multilinear polynomial on a simple algebra is a vector space.
However, according to~\cite{ML2} this conjecture fails even for some
(infinite dimensional) division algebras and the polynomial
$p(x,y)=xy-yx$. Cohn~\cite{C} constructed a~division ring~$D$ in
which every element is a commutator, i.e.,~$p(D)=D$.

It is interesting to investigate the Kaplansky conjecture for finite
dimensional simple algebras. In this section we deal with
quaternions.
Note that the algebra of split quaternions, defined also by
$4$-dimensional vector space $\langle 1,i,j,k \rangle_\R$ with
multiplication defined by
\begin{gather*}ij=-ji=k,\qquad jk=-kj=-i,\qquad ki=-ik=j,\qquad i^2=-1,\qquad j^2=k^2=1,\end{gather*}
 is isomorphic to $M_2(\R)$.
Let us start with proving the following straightforward but
important lemmas:

\begin{Lemma}\label{elem-q} Let $p$ be a multilinear polynomial.
If $a_i$ are basic quaternions $(1,i,j,k)$, then $p(a_1,\dots,a_m)$
is $c\cdot q$ for some basic quaternion $q$ and some scalar $c\in\R$
$($which can equal~$0)$.
\end{Lemma}
\begin{proof}
Note that for any two basic quaternions $q_1$ and $q_2$, $q_1q_2=\pm
q_2q_1$. Therefore, taking products of $m$ basic quaternions we
obtain the same result (up to $\pm$). Thus the sum of these results
multiplied by scalars must be a basic quaternion multiplied by some
scalar coefficient.
\end{proof}
\begin{Lemma}\label{cone-1-q}
For any multilinear polynomial $p$, $\Image p$ is a self-similar
cone, i.e.,~for any invertible $h\in\HH$, any scalar $c\in\R$ and
any element $\alpha\in\Image p$, $ch\alpha h^{-1}\in\Image p$.
\end{Lemma}
\begin{proof}
If $p(x_1,\dots,x_m)=\sum\limits_{\sigma\in S_m}c_\sigma
x_{\sigma(1)}\cdots x_{\sigma(m)}$, then
\begin{gather*}
p\big(hx_1h^{-1},hx_2h^{-1}\dots,hx_mh^{-1}\big)=
\sum_{\sigma\in S_m}c_\sigma
hx_{\sigma(1)}h^{-1}\cdots hx_{\sigma(m)}h^{-1} =hp(x_1,\dots,x_m)h^{-1},
\end{gather*}
 and
thus $p\big(chx_1h^{-1},hx_2h^{-1}\dots,hx_mh^{-1}\big)=ch\alpha h^{-1}\in\Image p$.
\end{proof}
\begin{Lemma}\label{cone-q} The set of pure quaternions $V$ is an irreducible self-similar cone,
i.e., equals all of its conjugates.
\end{Lemma}
\begin{proof}It is enough to show that any self-similar cone including the element $i$ contains $V$.
Take $h(y,z)=1+yj+zk$, thus $h^{-1}=\frac{1-yj-zk}{1+y^2+z^2}$.
Thus a minimal self-similar cone $C$ containing~$i$ contains
all elements $c\cdot hih^{-1}$, in particular it contains all elements $(1+yj+zk)i(1-yj-zk)=\big(1-y^2-z^2\big)i+2zj-2yk$.
Consider an arbitrary pure quaternion $ai+bj+ck$. If $b=c=0$ then this vector belongs to $C$ because it is a multiple of $i$. Assume that at least one of $b$ and $c$ is nonzero. Then $b^2+c^2>0$, hence one can take $l=\frac{a+\sqrt{a^2+b^2+c^2}}{2}$,
$y=-\frac{c}{a+\sqrt{a^2+b^2+c^2}}$, $z=\frac{b}{a+\sqrt{a^2+b^2+c^2}}$, these numbers are well defined. Thus the element $(1+yj+zk)i(1-yj-zk)=\big(1-y^2-z^2\big)i+2zj-2yk=ai+bj+ck$ belongs to $C$.
\end{proof}

Now we are ready to prove the main theorem:

\begin{proof}[Proof of Theorem~\ref{main-q}]
Let us substitute basic quaternions in the polynomial~$p$. Lemma~\ref{elem-q} implies that we obtain multiples of basic quaternions.
Consider the four possible cases: all these results vanish, or among
these results there are scalars only,
 pure quaternions only, and both pure quaternions and scalars.
The first two cases quickly lead to answers about the image of the
polynomial $p$: in the first case $p$ is PI, its image is $\{0\}$,
and in the second case it is central polynomial, and its image is
$\R$. In the third case the image is $V$, by Lemma \ref{cone-q}.

Therefore the most interesting case is the fourth one. We assume
that there are basic quaternions $x_1,\dots,x_m$ and $y_1,\dots,y_m$
such that $p(x_1,\dots,x_m)=k\in\R\setminus\{0\}$ and
$p(y_1,\dots,y_m)=v\in V\setminus\{0\}$. We will show that in this
case the image of $p$ is the set of all the quaternions. For that
let us consider the following $m+1$ evaluations depending on
$y_1,\dots,y_m$: $A_1=p(x_1,x_2,\dots,x_m);
A_2=p(y_1,x_2,\dots,x_m); A_3=p(y_1,y_2,x_3,\dots,x_m);\dots;
A_{m+1}=p(y_1,\dots,z_m)$. Note that $A_1$ is a constant taking only
one possible value (which is a nonzero scalar), for any~$i$ $\Image
A_i\subseteq\Image A_{i+1}$ and $\Image A_{m+1}=\Image p$ includes
nonscalar values. Therefore, there exists~$i$ such that $\Image
A_i\subseteq\R$ and $\Image A_{i +1}\not\subseteq\R$. Thus there
exist some collection of quaternions $r_1,r_2,\dots,r_m,r_i^*$ such
that $p(r_1,r_2,\dots,r_m)=r\in\R\setminus\{0\}$. and
$p(r_1,r_2,\dots,r_{i-1}, r_i^*, r_{i+1},\dots,r_m) \allowbreak \notin \R$.
Assume that $p(r_1,r_2,\dots,r_{i-1}, r_i^*, r_{i+1},
\dots,r_m)=a+v$ for $a\in\R$ and $v\in V$. Then $v\neq 0$. If
$a=c\cdot p(r_1,r_2,\dots,r_m)$ we can take $\tilde r_i=r_i^*-cr_i$,
and
\begin{gather*}p(r_1,r_2,\dots,r_{i-1}, \tilde r_i, r_{i+1},\dots,r_m)=v\in V\setminus\{0\}.\end{gather*}
Note that for arbitrary real numbers $x$ and $y$ we have an element
\begin{gather*}p(r_1,r_2,\dots,r_{i-1}, xr_i+yr_i\tilde r_i, r_{i+1},\dots,r_m)=xr+yv.\end{gather*} Consider an arbitrary quaternion $a+bi+cj+dk$, where $a,b,c,d\in\R$. Let us take an $x$ such that $xr=a$. By Lemma \ref{cone-q}, $V$ is an irreducible self-similar cone, thus there exist $h\in\HH$ and $y\in\R$ such that $yhvh^{-1}=bi+cj+dk$.
Hence,
\begin{gather*} p\big(hr_1h^{-1},hr_2h^{-1} ,\dots,hr_{i-1}h^{-1},h(xr_i+y\tilde
r_i)h^{-1},hr_{i+1}h^{-1},\dots,hr_mh^{-1}\big) \\
\qquad{} =
h(xr+yv)h^{-1}=xr+yhvh^{-1}=a+bi+cj+dk. \end{gather*} Therefore,
$\Image p=\HH$.
\end{proof}

\subsubsection[Semi-homogeneous polynomials evaluated on $\HH$]{Semi-homogeneous polynomials evaluated on $\boldsymbol{\HH}$}

Now we can consider the semi-homogeneous case and present the proof of Theorem \ref{homogen-q}:

\begin{proof}[Proof of Theorem~\ref{homogen-q}]
There exists an isomorphism
$\Phi\colon \HH(\R)\otimes \C_{\R}\rightarrow M_2(\C)_{\R}$ such that for any $q=a+bi+cj+dk\in\HH$ and any $z\in\C$
\begin{gather*}\Phi(q\otimes z)= z\cdot
\begin{bmatrix}
 a+bi & c+di \\
 -c+di & a-bi
\end{bmatrix}.\end{gather*}

Note that $\Phi(\HH\otimes 1_\C)$, i.e., the set of matrices{\samepage
\begin{gather*}\begin{bmatrix}
 a+bi & c+di \\
 -c+di & a-bi
\end{bmatrix}\end{gather*}
is Zariski dense in $M_2(\C)$.}

Note that there exist numbers
$w_1,\dots,w_m$ and $d\neq 0$ such that for any $c\in\R$
$p\big(c^{w_1}x_1,c^{w_2}x_2,\allowbreak \dots,c^{w_m}x_m\big)=c^dp(x_1,\dots,x_m)$.
Therefore $\Image p$ must be a cone with respect to positive real
multipliers, i.e., $\lambda x\in\Image p$ for any $x\in\Image p$ and any $\lambda>0$.

 By Theorem \ref{imhom-2}, the Zariski closure of the image of
polynomial evaluated on $2\times 2$ matrices with complex entries
must be either $\{0\}$, or $\C$, or $\ssl_2 (\C)$ or $M_2 (\C)$.

 In the first case $p$ is PI, and its image is $\{0\}$. In the second
case $p$ is a central polynomial, therefore the image of $p$ being a~cone with respect to positive multipliers must be either $\R$, or~$\R_{\geq 0}$, or~$\R_{\leq 0}$. In the third case, $p$ takes only
pure quaternion values. According to Lemma~\ref{cone-q} we know that $V$ is
an irreducible self-similar cone (up to real multipliers). Hence
$\Image p$ is a~self-similar cone up to positive real multipliers. Therefore $\Image p\cup (-\Image p)=V$, i.e., for any pure quaternion $v\in
V$, either $v$, or $-v$ belongs to $\Image p$. Without loss of
generality, assume that $i\in\Image p$. Hence $jij^{-1}\in\Image p$,
and $jij^{-1}=-i$. Hence for any $c\in\R$, $ci\in V$, and thus any
element $v\in V$ is conjugate to some $ci$, as we showed in the
proof of Lemma~\ref{cone-q}. Therefore, $\Image p=V$. In the forth
case an image of $p$ is a Zariski dense subset of $\HH$.
\end{proof}

Let us show some interesting examples of multilinear and homogeneous polynomials:
\begin{Example}\quad
\begin{enumerate}\itemsep=0pt\eroman
\item According to Amitsur--Levitsky theorem $s_4$ evaluated on $2\times 2$ matrices is a PI, and thus is PI on quaternions, as is its multilinearization.
\item The Lie bracket $p(x,y)=xy-yx$ is a Lie polynomial and its image is $V$.
\item The polynomial $[x_1,x_2][x_3,x_4]+[x_3,x_4][x_1,x_2]$, the multilinearization of $[x_1,x_2]^2$, evaluated on $\HH$, is a multilinear central polynomial.
\item An example where $\Image p=\HH(\R)$ can be constructed trivially.
\item The central polynomial $p(x,y)=[x,y]^4$ provides an example of complete homogeneous polynomial evaluated of $2\times 2$ matrices with real entries and taking only positive central values I took a square of $p$. Note that evaluated on quaternions we do not need to take a square. This polynomial takes only non positive values: indeed, $(ai+bj+ck)^2=-a^2-b^2-c^2$, and $\Image p$ is $\R_{\leq 0}$. Hence, $-p(x,y)=-[x,y]^2$ is an example of the polynomial with image set $\R_{\geq 0}$.

\item The polynomial $p(x,y)=[x,y]^2+\big[x^2,y^2\big]$ is the sum of two polynomials, the image of the first term is $\R_{\leq 0}$, and the second one has image $V$,
thus $\Image p$ is the set of quaternions with non-positive real part. Of course $-p$ has the opposite image: the set of quaternions with non-negative real part.

\item In Example~\ref{coneex2-2}(i), the polynomial $g(x_1 , x_2 ) = [x_1 , x_2 ]^2$ has the property that $g(A, B) = 0$
whenever $A$ is scalar, but $g$ can take on a non-zero value whenever $A$ is non-scalar. Thus, $g(x_1 , x_2 )x_1$ takes all values except scalars.

\item Any quaternion $q=a+v$, $a\in\R$, $v\in V$ can be considered as a $2\times 2$ matrix with complex entries, whose eigenvalues are $\lambda_{1,2}=a\pm ni$, where $n=\left\vert\left\vert v\right\vert\right\vert$ is the norm of the vector part of the quaternion. In particular
$\tr q=2a=2\real q$ and $\det q=a^2+n^2=\left\vert\left\vert q\right\vert\right\vert$. In Example~\ref{coneex2-2}(ii) we provided an example of the polynomial taking all possible values except for those where the ratio of eigenvalues belongs to some set~$S$. However here we should have a polynomial with real coefficients. Nevertheless this is possible:
let $S$ be any finite subset of~$S^2$ -- a unit circle on the complex plane. There exists a
completely homogeneous polynomial~$p$ such that $\Image p$ is the
set of all quaternions except the quaternions with ratio of
eigenvalues from $S$. Any $c\in S$ can be written in the form $c=\frac{a+bi}{a-bi}$ for some $a,b\in\R$. The construction of the polynomial is as follows. Consider
\begin{gather*}f(x)=x\cdot\prod_{c\in
S}((a+bi)\lambda_1-(a-bi)\lambda_2)((a+bi)\lambda_2-(a-bi)\lambda_1),\end{gather*} where
$\lambda_{1,2}$ are eigenvalues of $x$ and $c=\frac{a+bi}{a-bi}$. For each $c$ the
product $(a+bi)\lambda_1-(a-bi)\lambda_2)((a+bi)\lambda_2-(a-bi)\lambda_1)=-a^2(\lambda_1-\lambda_2)^2-b^2(\lambda_1+\lambda_2)^2$
is
a polynomial with real coefficients in $\tr x$ and $\tr x^2$. Thus
$f(x)$ is a polynomial with traces, and by \cite[Theorem~1.4.12]{Row}, one can rewrite each trace in $f$ as a fraction of
multilinear central polynomials.

After that we multiply the expression by the product of all the
denominators, which we can take to have value 1. We obtain a
completely homogeneous polynomial $p$ which image is the cone
under $\Image f$ and thus equals $\Image f$. The image of $p$ is
the set of all quaternions with ratios of eigenvalues
not belonging to $S$.
\end{enumerate}
\end{Example}
\subsection{The Deligne trick}\label{Deligne}

One of the strongest tools we use in our investigation is a famous
 trick of Deligne, given below in Remark~\ref{Del}. We use it when we give a proof by contradiction; in
some cases we obtain an element of $\widetilde{\UD}$ whose
characteristic polynomial can be decomposed, and this contradicts to
 Amitsur's theorem. Here are some examples of when we use it.
\subsubsection{Evaluations of word maps on matrix groups}

It is important to investigate possible images of word maps on
matrix groups, and the famous Deligne trick technique can be used in
this area. A word map $w(x_1,\dots,x_m)$ is called dominant if its
image is Zariski dense. A famous theorem of Borel states:
\begin{Theorem}Any nontrivial word map $w$ evaluated on $\SL_n(K)$ is dominant.
\end{Theorem}
\begin{proof}Let us use induction with respect to $n$. The case $n=1$ is trivial,
since $\SL_1(K)$ is a~trivial group, and any word map (including the
trivial one) is dominant. Assume that Borel's theorem holds for some
$n$ and let us show that it holds for $n+1$. Note that
$\SL_n\leq\SL_{n+1}$, because one can take
$\varphi\colon \SL_n\rightarrow\SL_{n+1}$ defined in the following way:
for any matrix $M\in\SL_n$ we will take an $(n+1)\times (n+1)$
matrix from $\SL_{n+1}$, having entries of $M$ in the first $n$ rows
and columns, zeros on the last row and on the last column except an
entry $(n+1,n+1)$ having $1$ there. Therefore the Zariski closure of
the image of $w$ evaluated on $\SL_{n+1}$, according to the induction
assumption, is at least $\Image \varphi$. Note that $\Image w$ as
well as its Zariski closure is closed under conjugations, therefore
it will include all matrices from $\SL_{n+1}$ having at least one
eigenvalue equal to $1$. Let us show that this set cannot be Zariski
closure of $\Image w$. Assume that it is. {\it Note that $w$ can be
considered as an element of $\widetilde{\UD}$, and if $1$ is its
eigenvalue in any evaluation, then its characteristic polynomial has
a divisor $\lambda-1$, which is impossible, since
$\chi_w(w)=(w-1)f(w)=0$ gives us that either $w=1$ $($which is
impossible if $w$ is not trivial$)$, or $f(M)=0$ for all $M$ being
values of the map~$w$. Note that $\deg f=n$ and this is impossible
since it contradicts Jordan's theorem. Here we use the Deligne
trick.}
\end{proof}

If $K$ is an algebraically closed field and $\Char K=0$ then there
is a working hypothesis that any nontrivial word map evaluated on
${\rm PSL}_n(K)$ must be surjective; here we will consider it for the
case $n=2$:
\begin{Conjecture}\label{word-PSL} If the field $K$ is algebraically closed
of characteristic $0$, then the image of any nontrivial group word
$w(x_1,\dots,x_m)$ on the projective linear group ${\rm PSL}_2(K)$ is
${\rm PSL}_2(K)$.
\end{Conjecture}

\begin{Remark} Note that if one takes the group $\SL_2$ instead of ${\rm PSL}_2$, Conjecture~\ref{word-PSL}
fails, since the matrix $-I+e_{12}$ does not belong to the image of
the word map $w=x^2$.
\end{Remark}

\begin{Example}
When $\Char K=p>0$, the image of the word map $w(x)=x^p$ evaluated
on ${\rm PSL}_2(K)$ is not ${\rm PSL}_2(K)$. Indeed, otherwise the matrix
$I+e_{12}$ could be written as $x^p$ for $x\in{\rm PSL}_2(K)$.
 If the eigenvalues of $x$ are equal, then $x=I+n$ where $n$ is nilpotent. Therefore $x^p=(I+n)^p=I+pn=I$.
 If the eigenvalues of $x$ are not equal, then $x$ is diagonalizable and therefore $x^p$ is also diagonalizable, a contradiction.
\end{Example}

\begin{Lemma}[Liebeck, Nikolov, Shalev, cf.~also \cite{BanZar,G}] $\Image w$ contains all matrices from ${\rm PSL}_2(K)$ which
are not unipotent.
\end{Lemma}

\begin{proof}According to \cite{B} the image of the word map $w$ must be Zariski
dense in $\SL_2(K)$. Therefore the image of $\tr w$ must be Zariski
dense in $K$. Note that $\tr w$ is a homogeneous rational function
and $K$ is algebraically closed. Hence, $\Image (\tr w)=K$. For any
$\lambda\neq\pm 1$ any matrix with eigenvalues $\lambda$ and
$\lambda^{-1}$ belongs to the image of $w$ since there is a matrix
with trace $\lambda+\lambda^{-1}$ in $\Image w$ and any two matrices
from $\SL_2$ with equal trace (except trace $\pm 2$) are similar.
\end{proof}

However the question of whether one of the matrices $(I+e_{12})$ or
$(-I-e_{12})$ (which are equal in ${\rm PSL}_2$) belongs to the image of
$w$ remains open.

The following Lie-algebraic counterpart of Borel's theorem holds:

\begin{Theorem}[\cite{BGKP}] \label{borel-lie}
Let $L$ be a split semisimple Lie algebra, $k$ a field. Suppose that a Lie polynomial $w(x_1,\ldots,x_n)$ is not an identity of the Lie algebra $\sl_2(k)$. Then the image of $w\colon L^n\to L$ is Zariski dense.
\end{Theorem}

The proof of this Theorem uses Deligne trick.

\begin{Remark}
At the best of our knowledge, the Deligne trick first appeared in the paper \cite{DS}. The Borel paper \cite{B} was published in the same issue.
\end{Remark}

\subsubsection{Evaluations of semihomogeneous polynomials on matrix rings}

\begin{proof}[Proof of Theorem~\ref{imhom-2}] Assume that there are matrices
$p(x_1,\dots,x_m)$ and $p(y_1,\dots,y_m)$ with different ratios of
eigenvalues in the image of $p$. Consider the polynomial matrix
$f(t)=p(tx_1+(1-t)y_1,tx_2+(1-t)y_2,\dots,tx_m+(1-t)y_m)$, and
$\Pi\circ f$ where $\Pi$ is defined in Remark~\ref{Phi-2}. Write
this nonconstant rational function $\frac{\tr^2 f}{\det f}$ in
lowest terms as $\frac{A(t)}{B(t)}$, where $A(t)$, $B(t)$ are
polynomials of degree $\leq 2\deg p$ in the numerator and
denominator.

An element $c\in K$ is in $\Image (\Pi\circ f)$ iff there exists
$t$ such that $A-cB=0$ (If for some
$t^*$ $A(t^*)-cB(t^*)=0$, then $t^*$ would be a common root of $A$ and $B$). Let $d_c =
\deg(A-cB)$. Then $d_c \le \max(\deg A,\deg B)\leq 2\deg p$, and
$d_c=\max(\deg A,\deg B)$ for almost all $c$. Hence, the
polynomial $A-cB$ is not
 constant and thus there is a root. Thus the image of $\frac{A(t)}{B(t)}$ is Zariski
dense, implying the image of $\frac{\tr^2 f}{\det f}$ is Zariski
dense.

Hence, we may assume that $\Image p$ consists only of matrices
having a fixed ratio $r$ of eigenvalues. If $r \ne \pm 1$, the
eigenvalues $\lambda_1$ and $\lambda_2$ are linear functions of $\tr p(x_1,\dots,x_m)$. Hence $\lambda_1$ and $\lambda_2$ are the
elements of the algebra of generic matrices with traces, which is
a domain by Proposition~\ref{Am1-2}. But the two nonzero elements
$p-\lambda_1I$ and $p-\lambda_2I$ have product zero, a~contradiction. Here we use the Deligne trick.

We conclude that $r = \pm 1$. First assume $r = 1$. If $\Char K \ne
2$, then $p$ is a~PI, by~Lemma~\ref{nilp-2}. If $\Char K=2$ then the
image is either $\sl_2(K)$ or $\hat K$, by Example~\ref{coneex1-2}(v).

Thus, we may assume $r = -1$ and $\Char K\neq 2$. Hence, $\Image
p$ consists only of matrices with $\lambda_1=-\lambda_2$. By
Lemma~\ref{nilp-2}, there is a non-nilpotent matrix in the image of
$p$. Hence, by Example~\ref{coneex1-2}(v), $\Image p$ is either
$\hat K$ or strictly contains it and is all of $\sl_2(K)$.
\end{proof}

\subsubsection{Evaluations of multilinear polynomials on matrix rings}

Recall Lemma~\ref{sklem}: {\it
 If $p$ is a multilinear polynomial evaluated on
the matrix ring $M_2(K)$ $($where~$K$ is a quadratically closed
field of characteristic not~$2)$, then $\Image p$ is either
$\{0\}$, $K$, $\sl_2$, or~$M_2(K)$.}

\begin{Remark}
Since the details are rather technical, we start by sketching the
proof. We assume that $\Image p=M_2(K)\setminus\tilde K$. The
linear change of the variable in position $i$ gives us the line
$A+tB$ in the image, where $A=p(x_1,\dots,x_m)$ and
$B=p(x_1,\dots,x_{i-1},y_i,x_{i+1},\dots,x_m)$. Take the function
that maps $t$ to $f(t) = (\lambda_1-\lambda_2)^2$, where
$\lambda_i$ are the eigenvalues of $A+tB$. Evidently \begin{gather*}f(t) =
(\lambda_1-\lambda_2)^2=(\lambda_1+\lambda_2)^2-4\lambda_1\lambda_2=(\tr(A+tB))^2-4\det
(A+tB),\end{gather*} so our function $f$ is a polynomial of $\deg\leq 2$
evaluated on entries of $A+tB$, and thus is a~polynomial in~$t$.

There are three possibilities: Either $\deg_t f \leq 1$, or $f$ is
the square of another polynomial, or~$f$ vanishes at two different
values of $t$ (say, $t_1$ and~$t_2$). (Note that here we use that
the field is quadratically closed). This polynomial $f$ vanishes
if and only if the two eigenvalues of $A+tB$ are equal, and this
happens in two cases (according to Lemma~\ref{thm1-2}): $A+tB$ is
scalar or $A+tB$ is nilpotent. Thus either both $A+t_iB$ are
scalar, or $A+t_1B$ is scalar and $A+t_2B$ is nilpotent, or both
$A+t_iB$ are nilpotent. In the first instance $A$ and $B$
are scalars, which is impossible. The second case instance that the
matrix $A+\frac{t_1+t_2}{2} B\in\tilde K$, which is also
impossible. The third instance implies that $\tr A=\tr B=0$ which we
claim is also impossible. If $\deg _t f\leq 1$, then for large~$t$
the difference $\lambda_1-\lambda_2$ of the eigenvalues of $A+tB$,
is much less than~$ t$, so the difference between eigenvalues of~$B$ must be~$0$, a contradiction.

It follows that $f(t) = (\lambda_1-\lambda_2)^2$ is the square of
a polynomial (with respect to~$t$). Thus
$\lambda_1-\lambda_2=a+tb$, where $a$ and $b$ are some functions
of the entries of the matrices $x_1,\dots,x_m,y_i$. Note that $a$
is the difference of eigenvalues of $A$ and $b$ is the difference
of eigenvalues of $B$, Thus
\begin{gather}\label{symmetry}
a(x_1,\dots,x_m,y_i)=b(x_1,\dots,x_{i-1},y_i,x_{i+1},\dots,x_m,x_i).
\end{gather}
Note \looseness=-1 that $(\lambda_1-\lambda_2)^2=a^2+2abt+b^2t^2$ which means
that $a^2$, $b^2$ and $ab$ are polynomials (note that here we use
$\Char K\neq 2$). Thus, $\frac{a}{b}=\frac{a^2}{ab}$ is a rational
function. Therefore there are polynomials $p_1$,~$p_2$ and $q$ such
that $a=p_1\sqrt q$ and $b=p_2\sqrt q$. Without loss of
generality, $q$ does not have square divisors. By~\eqref{symmetry}
we have that $q$ does not depend on $x_i$ and $y_i$. Now consider
the change of other variables. The function $a$ is the difference
of eigenvalues of $A=p(x_1,\dots,x_m)$ so it remains unchanged.
Thus $q$ does not depend on other variables also. That is why
$\lambda_1\pm\lambda_2$ are two polynomials and hence $\lambda_i$
are polynomials. One concludes with the last paragraph of the
proof.
\end{Remark}

\begin{proof}[Proof of Lemma~\ref{sklem}]
In view of Lemma \ref{thm1-2} it suffices to prove that the image of
$p$ cannot be $M_2(K)\setminus\tilde K$. Assume that the image of
$p$ is $M_2(K)\setminus\tilde K$. Consider for each variable~$x_i$
the line $x_i+ty_i,\ t\in K$. Then
$p(x_1,\dots,x_{i-1},x_i+ty_i,x_{i+1},\dots,x_m)$ is the line
$A+tB$, where $p(x_1,\dots,x_m)=A$ and
$p(x_1,\dots,x_{i-1},y_i,x_{i+1},\dots,x_m)=B$. Thus
$A+tB\notin\tilde K$ for any $t$. Since~$B$ is diagonalizable, we
can choose our matrix units $e_{i,j}$ such that $B$ is diagonal.
Therefore
\begin{gather*}B=\lambda_BI+\left(\begin{matrix}c & 0
\\0 & -c\end{matrix}\right),\qquad A=\lambda_AI+\left(\begin{matrix}x
& y \\z & -x\end{matrix}\right).\end{gather*} Hence
\begin{gather*}A+tB=(\lambda_A+t\lambda_B)I+\left(\begin{matrix}x+tc & y \\z
& -x-tc\end{matrix}\right).\end{gather*}

The matrix $\left(\begin{smallmatrix}x+tc & y
\\z & -x-tc\end{smallmatrix}\right)$
is nilpotent if and only if $(x+tc)^2 + yz = 0$,, which has the
solution
$t_{1,2}=\frac{1}{c}\big({-}x\pm\sqrt{-yz}\big)$. Thus, when $yz\neq 0$,
$\pi(A+t_jB)$ will be nilpotent for $j=1,2$, where
$\pi(X)=X-\frac{1}{2}\tr X$. However $\tr(A+t_jB)$ is nonzero for
one of these values of $t_j$, implying $A+t_jB\in\tilde K$, a~contradiction.

Thus, we must have $yz=0$. Without loss of generality we can
assume that $z=0$. Any mat\-rix~$M$ of the type
$qI+\left(\begin{smallmatrix}w & h
\\0 & -w\end{smallmatrix}\right)$ satisfies $\det M=q^2-w^2$ and
$q=\frac{1}{2}\tr M$. Thus $x=\sqrt{\frac{1}{4}(\tr A)^2-\det A}$
and $c=\sqrt{\frac{1}{4}(\tr B)^2-\det B}$. Consider the matrix
\begin{gather*}P_i=cA-xB=p(x_1,\dots,x_{i-1},cx_i-xy_i,x_{i+1},\dots,x_m),\end{gather*}
which must be scalar or nilpotent. It can be written explicitly
algebraically in terms of the entries of $x_i$ and $y_i$. Also,
$P_i=(c\lambda_B-x\lambda_A)I+(cy)e_{12}$, where $e_{12}$ is the
matrix unit. There are two cases. If $y=0$ then the line $A+tB$
includes a scalar matrix, and if $y\neq 0$ then
$(c\lambda_B-x\lambda_A)=0$ and all matrices on the line $A+tB$
have the same ratio of eigenvalues.

Let $S_1 = \{ i \colon P_i\in K\}$ and $S_2 = \{i \colon P_i\in
\sl_2(K)\}$. Without loss of generality we can assume for some
$k\leq m$ that $S_1
 =\{1,2,\dots,k\}$ and $\{k+1,\dots,m\}$. The
four entries of $p(x_1,\dots,x_m)$ are
\begin{gather*}p_{ij}(x_{1,(1,1)},x_{1,(1,2)},x_{1,(2,1)},x_{1,(2,2)},\dots,x_{m,(2,2)}),\end{gather*}
polynomials in the entries of $x_i$. Consider the scalar function
\begin{gather*}f_1(x_1,\dots,x_m)=\frac{\frac{1}{2}\tr p(x_1,\dots,x_m)}{R(x_1,\dots,x_m)},\end{gather*} where
$R(x_1,\dots,x_m)=\sqrt{\frac{1}{4}\tr^2p(x_1,\dots,x_m)-\det
p(x_1,\dots,x_m)}$. This function is defined everywhere except for
those $(x_1,\dots,x_m)$ for which $p(x_1,\dots,x_m)$ is a matrix
with equal eigenvalues, because $R$ is the half-difference of
eigenvalues. The function $f_1(x_1,\dots,x_m)$ does not depend on
$x_{k+1},\dots,x_m$ because for any $i\geq k+1$, substituting
$y_i$ instead of $x_i$ does not change the ratio of eigenvalues of
$p(x_1,\dots,x_m)$. Consider the matrix function
\begin{gather*}f_2(x_1,\dots,x_m)=\frac{p(x_1,\dots,x_m)-\frac{1}{2}\tr p(x_1,\dots,x_m)}{R(x_1,\dots,x_m)}.\end{gather*} This function is also
defined everywhere except for those $(x_1,\dots,x_m)$ such that
the eigenvalues of $p(x_1,\dots,x_m)$ are equal. The function
$f_2(x_1,\dots,x_m)$ does not depend on $x_i$, $i\leq k$, because
for any $i\leq k$ substituting $y_i$ instead of $x_i$ does not
change the basis in which $p(x_1,\dots,x_m)$ is diagonal. $R^2$ is
a polynomial: \begin{gather*}R^2=\frac{1}{4}\tr^2p(x_1,\dots,x_m)-\det
p(x_1,\dots,x_m).\end{gather*}
Write $R^2=r_1r_2r_3$ where $r_1$ is the product of all the
irreducible factors in which only $x_1,\dots,x_k$ occur, $r_2$ is
the product of all the irreducible factors in which only
$x_{k+1},\dots,x_m$ occur, $r_3$ is the product of the other
irreducible factors. We have that
\begin{gather*}\frac{\tr^2
p(x_1,\dots,x_m)}{r_1(x_1,\dots,x_m)r_2(x_1,\dots,x_m)r_3(x_1,\dots,x_m)}=f_1^2(x_1,\dots,x_m)\end{gather*}
does not depend on $x_{k+1},\dots,x_m$. Therefore if $\tr^2
p=q_1q_2q_3$ (again in $q_1$ only $x_1,\dots,x_k$ occur, in $q_2$
only $x_{k+1},\dots,x_m$ occur and $q_3$ is all the rest) then
$\frac{r_1r_2r_3}{q_1q_2q_3}$ does not depend on
$x_{k+1},\dots,x_m$. Hence $r_2=q_2$ and $r_3=q_3$(up to scalar
factors). As soon as $q_1q_2q_3$ is a square of a polynomial all
$q_i$ are squares therefore $r_2$ and $r_3$ are squares. Now
consider the function $\frac{p_{12}^2}{R^2}$. This is the square
of the $(1,2)$-entry in the matrix function $f_2$, so it does not
depend on $x_1,\dots,x_k$. Writing $p_{12}^2=q_1q_2q_3$(where,
again, only $x_1,\dots,x_k$ occur in $q_1$, only
$x_{k+1},\dots,x_m$ occur in $q_2$ and $q_3$ is comprised of all
the rest), then all the $q_i$ are squares and $q_1=r_1$, implying
$r_1$ is square. Thus the polynomial $r_1r_2r_3=R^2$ is the square
of a polynomial. Therefore $R$ is a polynomial. We conclude that
$\lambda_1-\lambda_2=2R$ is a polynomial (where we recall that
$\lambda_1$ and $\lambda_2$ are the eigenvalues of
$p(x_1,\dots,x_m)$). $\lambda_1+\lambda_2=\tr (p)$ is also a
polynomial and hence $\lambda_i$ are polynomials, which obviously
 are invariant under conjugation since
 any conjugation is the square of some other conjugation).
Hence, $\lambda_i$ are the polynomials of traces, by Donkin's
theorem quoted above. Now consider the polynomials~$(p-\lambda_1I)$
and $(p-\lambda_2I)$, which are elements of the algebra of free
matrices with traces, which we noted above is a domain. Both are not
zero but their product is zero, a contradiction.
\end{proof}

\begin{Remark}\label{Del}
This trick (when we use that the characteristic polynomial of a
non-central element of $\widetilde{\UD}$ cannot be decomposed) is
called the \textit{Deligne trick}.\end{Remark}

\subsection[Some important questions and open problems regarding evaluations of polynomials on low rank algebras]{Some important questions and open problems regarding evaluations\\ of polynomials on low rank algebras}

In this section we formulate some problems that remain open.

The first and one of the most important problems in this area is the investigation of the possible image sets of multilinear polynomials on matrix algebras. One of the main and the most perspective conjectures answering this question is the L'vov--Kaplansky conjecture. However it remains being an open problem. For the case of $2\times 2$ matrices we have Theorem \ref{main-2r}, nevertheless the case when $\ssl_2\subsetneq\Image p$ should be investigated, not only for $K=\R$.

\begin{Problem}
Give a full classification of all possible image sets of multilinear polynomials evaluated on $M_2(K)$ for arbitrary field $K$ in the case when $\ssl_2\subsetneq\Image p$.
\end{Problem}

Another important problem is to investigate possible images of semihomogeneous polynomials evaluated on $M_2(K)$. We some a good result when $K$ is quadratically closed (see Theorem \ref{imhom2-int}), however our result for $K=\R$ does not settle the problem. For $\R$ the Zariski topology is not so natural, and it would be much more important to investigate possible images with respect to the standard topology. This question remains open. The question about classification of such image sets for the arbitrary field remains being open as well. The same question can be asked for the quaternion algebra. We have Theorem \ref{homogen-q}, however the problem of classification of all possible Zariski dense images of
semi-homogeneous polynomials evaluated on $\HH$ remains being open.
\begin{Problem}
Give a classification of all possible image sets of semihomogeneous polynomials evaluated on $M_2(K)$ for arbitrary $K$, or at least for $K=\R$ with respect to the standard topology. Give a classification of all possible images of semihomogeneous polynomials evaluated on $\HH$ with respect to the standard topology.
\end{Problem}

The same question can be asked for arbitrary (non-homogeneous) polynomial. Let us provide an important example:
\begin{Example}
Let $h(x)$ be arbitrary polynomial in one variable, and consider a noncommutative polynomial $h([x,y])$ evaluated on $M_2(K)$. Note that $[x,y]$ is trace zero matrix, so its eigenvalues are $\pm\lambda$ for some $\lambda$. Thus, if $[x,y]$ is not nilpotent then $[x,y]$ is similar to the matrix $\lambda e_{11}-\lambda e_{22}$. Hence, $h([x,y])$ is similar to the matrix $h(\lambda)e_{11}+h(-\lambda)e_{22}$, and its eigenvalues are~$h(\pm\lambda)$. Therefore,
$\Image h([x,y])$ is the set of all matrices having pairs of eigenvalues $h(\pm\lambda)$.
\end{Example}

 The question whether all possible images of noncommutative polynomials evaluated on~$M_2(K)\!$ can be only $\{0\}$, $K$, $\ssl_2(K)$, $M_2(K)$ or the set of all matrices having pairs of eigenvalues $h(\pm\lambda)$ for some polynomial $h$ is an open problem. The same question can be asked for the evaluations of non homogeneous polynomials on the quaternion algebra.
\begin{Problem}
Give a classification of all possible image sets of non homogeneous $($arbitrary$)$ polynomials on~$M_2(K)$ $($and~$\HH)$, where~$K$ is an arbitrary field $($or for some partial cases such as $K=\R$, $K=\C)$ $K$ being any quadratically closed field.
\end{Problem}

Another important algebra is the algebra of Cayley numbers. This
algebra is nonassociative, and therefore polynomials evaluated on it
are also nonassociative. However, it is close to associative, and
part of our tools can work. The problem is to give a classification
of possible image sets of nonassociative multilinear polynomials
evaluated on the algebra of Cayley numbers and to check whether all
possible images are vector spaces.
\begin{Problem}
Give a classification of all possible evaluations of multilinear
polynomials on the Cayley numbers algebra, the next step would be to
consider the same question for semihomogeneous and for arbitrary
polynomials.
\end{Problem}
\begin{Question}
Is it true for a simple $3$-dimensional Lie algebra that it is
possible not only to provide a classification of possible image sets
of polynomials $($see Theorem~{\rm \ref{mainlie})} but also for any given
set $S$ to describe the set of all polynomials whose image sets are~$S$?
\end{Question}
\begin{Question}[see Conjecture~\ref{word-PSL}] If the field $K$ is algebraically
closed of characteristic $0$, then must the image of any nontrivial
group word $w(x_1,\dots,x_m)$ on the projective linear group
${\rm PSL}_2(K)$ be ${\rm PSL}_2(K)$?
\end{Question}

Malev and Pines~\cite{MP} have established that the evaluations of
multilinear polynomials on the rock-paper-scissors algebra is a~vector space.

\section{The case of matrices of rank 3}\label{3a}

Now we turn specifically to the case $n=3$ and the proof of Theorem~\ref{main3-int}.

\begin{Lemma}\label{extr-3}
We define functions $\omega _k \colon M_3(K) \to K$ as follows: Given a
matrix $a$, let $\lambda_1$, $\lambda_2$, $\lambda _3$ be the
eigenvalues of $a$, and denote
\begin{gather*}\omega_k : = \omega_k (a) =\sum\limits_{1\leq i_1<i_2<\dots<i_k\leq
3}\lambda_{i_1}\dots \lambda_{i_k}.\end{gather*} Let $p(x_1,\dots,x_m)$ be a~semi-homogeneous, trace vanishing polynomial.

 Consider the rational function
$H(x_1,\dots,x_m)
=\frac{\omega_2(p(x_1,\dots,x_m))^3}{\omega_3(p(x_1,\dots,x_m))^2}$
$($taking values in $K \cup \{\infty\})$. If $\Image H$ is dense in~$K$, then $\Image p$ is dense in $\sl_3$.
\end{Lemma}
\begin{proof} Note that $\omega_2(p)^3$ and $\omega_3(p)^2$ are
semi-homogeneous. Thus, $\Image H$ is dense in $K$ iff the image
of the pair $\big(\omega_2(p)^3,\omega_3(p)^2\big)$ is dense in~$K^2$. But
since $\omega_2$ and $\omega_3$ are algebraically independent, so
are $\omega_2(p)^3$ and $\omega_3(p)^2$, so we conclude that the
image of the pair $\big(\omega_2(p)^3,\omega_3(p)^2\big)$ is dense in
$K^2$. Thus, the set of characteristic polynomials of evaluations
of $p$ is dense in the space of all possible characteristic
polynomials of trace zero matrices. Therefore, the set of all
triples $(\lambda_1,\lambda_2,-\lambda_1-\lambda_2)$ of
eigenvalues of matrices from $\Image p$ is dense in the plane
$x+y+z=0$ defined in $K^3$, implying that $\Image p$ is dense in
$\sl_3$.
\end{proof}

\subsection[$3\times 3$ matrices over a field with a primitive cube root of 1]{$\boldsymbol{3\times 3}$ matrices over a field with a primitive cube root of 1}\label{3ab}

 Let $K$ be an algebraically closed field.
For $\chara(K) \ne 3$ we fix a primitive cube root $\varepsilon \ne
1$ of $1$; when $\chara(K) = 3$ we take $\varepsilon = 1$.

The proof of Theorem \ref{semi_tr0_3-3} will be presented in Section~\ref{Deligne-3}.

\begin{Example}
The element $ \big[x,[y,x]x[y,x]^{-1}\big]$ of $\widetilde{\UD}$ takes on
only $3$-scalar values (see \cite[Theorem~3.2.21, p.~180]{Row})
and thus gives rise to a homogeneous polynomial taking on only
$3$-scalar values.
\end{Example}

\subsection{Multilinear trace vanishing polynomials}\label{3ac}

We can characterize the possible image sets of multilinear trace vanishing polynomials.

\begin{Lemma}\label{diag-not-scal-3}
If $p$ is a multilinear polynomial, not PI nor central, then there
exists a collection of matrix units $(e_1,e_2,\dots,e_m)$ such that
$p(e_1,e_2,\dots,e_m)$ is a diagonal but not scalar matrix.
\end{Lemma}
\begin{proof}
According to Lemma \ref{graph-2r}, for any matrix units $e_i$ the value $p(e_1,e_2,\dots,e_m)$ must be either $ce_{i,j}$ or some diagonal matrix. In addition, let us note that according to the Remark~\ref{linear-2} the linear span of the image set $\langle p(M_3)\rangle$ (being a linear span of all evaluations of $p$ on sets of matrix units) is one of four possible sets: either is either $\{0\}$, $K$, $\ssl_3(K)$, or $M_3(K)$. We assumed that $p$ is not PI nor central, hence the linear span of the image set $\langle p(M_3)\rangle$ either is~$\ssl_n(K)$, or~$M_n(K)$. Thus,
the linear span of all
$p(e_1,e_2,\dots,e_m)$ for any matrix units $e_i$ such that
$p(e_1,e_2,\dots,e_m)$ is diagonal, includes all
$\diag\{x,y,-x-y\}$. In particular there exists a~collection of
matrix units $(e_1,e_2,\dots,e_m)$ such that $p(e_1,e_2,\dots,e_m)$
is a diagonal but not scalar matrix.
\end{proof}

The proof of Theorem \ref{multi_tr=0_3} will be presented in Section~\ref{Deligne-3}.

\begin{Remark}
Assume that $\chara (K) = 3$ and $p$ is a multilinear polynomial,
which is neither PI nor central. Then, by Lemma~\ref{diag-not-scal-3}
there exists a collection of matrix units $e_i$ such that
\begin{gather*}p(e_1,\dots,e_m)=\diag\{\alpha ,\beta ,\gamma\}\end{gather*} is diagonal but not
scalar. Without loss of generality, $\alpha \neq \beta $. Hence
$p^3(e_1,\dots,e_m)=\diag\big\{a^3,\allowbreak\beta ^3, \gamma^3\big\}$ and $\alpha
^3\neq \beta ^3$ because $\chara (K) = 3$. Therefore $p$ is not
$3$-scalar.
\end{Remark}

\begin{Theorem}\label{equation}
If $p$ is a multilinear polynomial such that $\Image p$ does not
satisfy the equation $\gamma\omega_1(p)^2=\omega_2(p)$ for
$\gamma=0$ or $\gamma=\frac{1}{4}$, then $\Image p$ contains a
matrix with two equal eigenvalues that is not diagonalizable and
of determinant not zero. If $\Image p$ does not satisfy any
equation of the form $\gamma\omega_1(p)^2=\omega_2(p)$ for any
$\gamma$, then the set of non-diagonalizable matrices of $\Image
p$ is Zariski dense in the set of all non-diagonalizable matrices,
and $\Image p$ is dense.
\end{Theorem}
\begin{proof}
If not, then by Lemma \ref{dim2-2r} there is at least one
variable (say, $x_1$) such that $a = p(x_1,x_2,\dots,x_m)$ does not
commute with $b = p(\tilde x_1,x_2,\dots,x_m)$. Consider the matrix
 $a+tb=p(x_1+t\tilde x_1,x_2,\dots,x_m)$, viewed as a polynomial in $t$.

Recall that the discriminant of a $3\times 3$ matrix with
eigenvalues $\la _1$, $\la_2$, $\la _3$ is defined as $\prod\limits_{1 \le i
< j \le 3}(\la_i - \la _j)^2$. Thus, the discriminant of $a+tb$ is
a polynomial~$f(t)$ of degree~$6$. If~$f(t)$ has only one root $t_0$, then this
root is defined in terms of the entries of $\tilde
x_1,x_1,x_2,\dots,x_m$, and invariant under the action of the
symmetric group, and thus is in Amitsur's division algebra~$\widetilde{\UD}$. By Lemma~\ref{divAm-3}, $a+t_0b$ is scalar, and
the uniqueness of~$t_0$ implies that $a$ and $b$ are scalar,
contrary to assumption.

Thus, $f(t)$ has at least two roots -- say, $t_1\neq t_2$, and
 the matrices $a+t_1b$ and $a+t_2b$ each must have multiple
eigenvalues. If both of these matrices are diagonalizable, then each
of $a+t_ib$ have a $2$-dimensional plane of eigenvectors. Therefore
we have two $2$-dimensional planes in $3$-dimensional linear space,
which must intersect. Hence there is a common eigenvector of both
$a+t_ib$ and this is a common eigenvector of $a$ and $b$. If $a$ and
$b$ have a common eigenspace of dimension~1 or~2, then there is at
least one eigenvector (and thus eigenvalue) of~$a$ that is uniquely
defined, implying $a \in \widetilde{\UD}$ by
Remark~\ref{patch1-int}, contradicting Lemma~\ref{divAm-3}. If $a$
and $b$ have a common eigenspace of dimension~$3$, then $a$ and $b$
commute, a contradiction.

We claim that there cannot be a diagonalizable matrix with equal
eigenvalues on the line $a+tb$. Indeed, if there were such a
matrix, then either it would be unique (and thus an element of
$\widetilde {UD}$, which cannot happen), or there would be at least two such
diagonalizable matrices, which also cannot happen, as shown above.

Assume that all matrices on the line $a+tb$ of discriminant
zero have determinant zero. Then either all of them are of the
type $\diag\{\lambda,\lambda,0\}+e_{12}$ or all of them are of the
type $\diag\{0,0,\mu\}+e_{12}$. (Indeed, there are three roots of
the determinant equation $\det(a+tb)=0$, which are pairwise
distinct, and all of them give a matrix with two equal
eigenvalues, all belonging to one of these types, since
otherwise one eigenvalue is uniquely defined and thus yields an
element of $\widetilde {UD}$, which cannot happen.

In the first case, all three roots of the determinant equation
$\det(a+tb)=0$ satisfy the equation
$(\omega_1(a+tb))^2=4\omega_2(a+tb)$. Hence, we have three
pairwise distinct roots of the polynomial of maximal degree $2$,
which can occur only if the polynomial is identically zero. It
follows that also $(\omega_1(a))^2-4\omega_2(a)=0$,
 so $(\omega_1(p))^2-4\omega_2(p)=0$ is
identically zero, which by hypothesis cannot happen.

In the second case we have the analogous situation, but $\omega_2(p)$ will be identically zero, a~contradiction.

\looseness=-1 Thus on the line $a+tb$ we have at least one matrix of the type
$\diag\{\lambda,\lambda,\mu\}+e_{12}$ and $\lambda\mu\neq 0$.
Consider the algebraic expression $\mu \lambda^{-1}$. If not
constant, then it takes on almost all values, so assume that it is a
constant $\delta$. Then $\delta\ne -2$, since otherwise this
matrix will be the unique matrix of trace 0 on the line $a+tb$ and
thus an element of $\widetilde {UD}$, contrary to
Lemmas~\ref{divAm-3} and \ref{div2-3}. Consider the polynomial
$q=p-\frac{\tr p}{\delta+2}$. At the same point $t$ it takes on the
value $\diag\{0,0,(\delta-1)\lambda\}+e_{12}$. Hence all three
pairwise distinct roots of the equation $\det q(x_1+t\tilde
x_1,x_2,\dots,x_m)=0$ will give us a matrix of the form
$\diag\{0,0,*\}+e_{12}$ (otherwise we have uniqueness and thus an
element of $\widetilde {UD}$), contradicting Lemma~\ref{div2-3}. Therefore~$q$
satisfies an equation $\omega_2(q)=0$. Hence, $p$ satisfies an
equation $\omega_1(p)^2-c\omega_2(p)=0$, for some constant~$c$,
a~contradiction. Hence almost all non-diagonalizable matrices
belong to the image of~$p$, and they are almost all matrices of
discriminant $0$ (a subvariety of~$M_3(K)$ of codimension~$1$). By
Amitsur's theorem, $\Image p$ cannot be a subset of the
discriminant surface. Thus, $\Image p$ is dense in~$M_3(K)$.
\end{proof}
\begin{Remark}\label{rem_sl3}Note that if $\omega_1(p)$ is identically zero, and $\omega_2 (p)$
is not identically zero, then $\Image p$ contains a matrix
similar to $\diag\{1,1,-2\}+e_{12}$. Hence $\Image p$ contains all
diagonalizable trace zero matrices (perhaps with the exception of
the diagonalizable matrices of discriminant $0$, i.e., matrices
similar to $\diag\{c,c,-2c\}$), all non-diagonalizable
non-nilpotent trace zero matrices, and all matrices $N$ for which
$N^2=0$. Nilpotent matrices of order $3$ also belong to the image of~$p$, as we shall see in Lemma \ref{V3inimage}.
\end{Remark}

\subsubsection{Multilinear trace vanishing polynomials over an algebraically closed field}\label{3ad}

\begin{Lemma}\label{V3} A matrix is $3$-scalar iff its eigenvalues are in $\big\{\gamma, \gamma\varepsilon, \gamma\varepsilon^2\colon \gamma \in K\big\}$,
where $\gamma^3\in K$ is its determinant. The variety $V_3$ of
$3$-scalar matrices has dimension~$7$.
\end{Lemma}
\begin{proof} The first assertion is immediate since the
characteristic polynomial is $x^3 -\gamma^3$. Hence~${V_3}$ is a
variety. The second assertion follows since the invertible
elements of $V_3$ are defined by two equations: $\tr (x)=0$ and
$\tr \big(x^{-1}\big)=0$ and thus a~$V_3$ is a variety of codimension~$2$.
\end{proof}
\begin{Lemma}\label{V3inimage}
Assume $\Char K\neq 3$. If $p$ is neither PI nor central, then the
variety $V_3$ is contained in~$\Image p$.
\end{Lemma}
\begin{proof}
According to Lemma~\ref{graph-2r} there exist matrix units
$e_1,e_2,\dots,e_m$ such that $p(e_1,e_2,\dots,\allowbreak e_m)=e_{1,2}$.
Consider the mapping $\chi$ described in the proof of Theorem~\ref{dense} (see Section~\ref{Euler-ch} for details). For any triples $T_i=(t_{1,i},t_{2,i},t_{3,i})$, let
\begin{gather*}f(T_1,T_2,\dots,T_m)=p\big(\dots,t_{1,i}e_i+t_{2,i}\chi(e_i)+t_{3,i}\chi^2(e_i),\dots\big).\end{gather*}
 $\Image f$ (a subset of $\Image p$) is a subset
of the $3$-dimensional linear space
\begin{gather*}L=\{\alpha e_{12}+\beta e_{23}+\gamma e_{31},\ \alpha ,\beta ,\gamma\in
K\}.\end{gather*} Since $e_{12}$, $e_{23}$ and $e_{31}$ belong to $\Image f$,
we see that $\Image f$ is dense in $L$, and hence at least one
matrix $a = \alpha e_{12}+\beta e_{23}+\gamma e_{31}$ for
$\alpha\beta \gamma\neq 0$ belongs to $\Image p$. Note that this
matrix is $3$-central. Thus the variety ${V_3}$, excluding the
nilpotent matrices, is contained in $\Image p$. The nilpotent
matrices of order $2$ also belong to the image of $p$ since they
are similar to $e_{12}$.

Let us show that all nilpotent matrices of order $3$ (i.e.,
matrices similar to $e_{12}+e_{23}$), also belong to $\Image p$.
We have the multilinear polynomial
\begin{gather*}f(T_1,T_2,\dots,T_m)=q(T_1,T_2,\dots,T_m)e_{12}+
r(T_1,T_2,\dots,T_m)e_{23}+s(T_1,T_2,\dots,T_m)e_{31},\end{gather*} therefore
$q$, $r$ and $s$ are three scalar multilinear polynomials. Assume
there is no nilpotent matrix of order $3$ in $\Image p$. Then we
have the following: if $q=0$ then either $rs=0$, if $r=0$ then
$sq=0$, and if $s=0$ then $qr=0$. Assume $q_1$ is the greatest
common divisor of $q$ and $r$ and $q_2=\frac{q}{q_1}$. Note that both
$q_i$ are multilinear polynomials defined on disjoint sets of
variables. If $q_1=0$ then $r=0$ and if $q_2=0$ then $s=0$.
Since there are no repeated factors, $r=q_1r'$ is a multiple
of $q_1$ and $s=q_2s'$ is a multiple of $q_2$. The polynomial $r'$
cannot have common devisors with $q_2$, therefore if we consider
any generic point $(T_1,\dots,T_m)$ on the surface $r'=0$ then
$r(T_1,\dots,T_m)=0$ and $q(T_1,\dots,T_m)\neq 0$. Hence
$s(T_1,\dots,T_m)=0$ for any generic $(T_1,\dots,T_m)$ from the
surface $r'=0$. Therefore $r'$ is the divisor of $s$. Remind both
$q_1$ and $q_2$ are multilinear polynomials defined on disjoint
subsets of $\{T_1,T_2,\dots,T_m\}$. Without loss of generality
$q_1=q_1(T_1,\dots,T_k)$, and $q_2=q_2(T_{k+1},\dots,T_m)$.
Therefore $r'=r'(T_{k+1},\dots,T_m)$ and it is divisor of $s$.
Also remind $s=s'q_2$ so $q_2(T_{k+1},\dots,T_m)$ is also divisor
of $s$. Hence $r'=cq_2$ where $c$ is constant. Thus
$r=q_1r'=cq_1q_2=cq$. However there exist $(T_{k+1},\dots,T_m)$
such that $q=0$ and $r=1$ (i.e., such that
$f(T_{k+1},\dots,T_m)=e_{23}$). A contradiction.
\end{proof}
\begin{Remark}
When $\Char K=3$, $V_3$ is the space of matrices with
equal eigenvalues (including also scalar matrices). The same proof
shows that all nilpotent matrices belong to the image of $p$, as
well as all matrices similar to $cI+e_{12}+e_{23}$. But we do not
know how to show that scalar matrices and matrices similar to
$cI+e_{12}$ belong to the image of $p$.
\end{Remark}

\begin{proof}[Proof of Theorem~\ref{main3-int}]
First assume that $\Char K\ne 3$. According to Lemma~\ref{V3inimage} the variety ${V_3}$ is contained in $\Image p$.
Therefore $\Image p$ is either the set of $3$-scalar matrices, or
some $8$-dimensional variety (with $3$-scalar subvariety), or is
$9$-dimensional (and thus dense).

It remains to classify the possible $8$-dimensional images. Let
us consider all matrices $p(e_1,\dots,\allowbreak e_m)$ where $e_i$ are matrix
units. If all such matrices have trace~0, then $\Image p$ is dense
in $\sl_3(K)$, by Theorem \ref{multi_tr=0_3}. Therefore we may
assume that at least one such matrix $a$ has eigenvalues~$\alpha$,~$\beta $ and~$\gamma$ such that $\alpha +\beta +\gamma\neq 0$.
By Theorem \ref{dense} we cannot have $\alpha +\beta +\gamma$, $\alpha +\beta \varepsilon+\gamma\varepsilon^2$ and $\alpha +\beta
\varepsilon^2+\gamma\varepsilon$ all nonzero. Hence $a$ either is
scalar, or a linear combination (with nonzero coefficients) of a~scalar matrix and $\diag\big\{1,\varepsilon,\epsilon^2\big\}$ (or with
$\diag\big\{1,\varepsilon^2,\epsilon\big\}$, without loss of generality --
with $\diag\big\{1,\varepsilon,\varepsilon^2\big\}$). By Theorem~\ref{equation}, if $\Image p$ is not dense, then $p$ satisfies an
equation of the type $(\tr(p))^2=\gamma\tr\big(p^2\big)$ for some
$\gamma\in K$. Therefore, if a scalar matrix belongs to $\Image
p$, then $\gamma=\frac{1}{3}$ and $\Image p$ is the set of
$3$-scalar plus scalar matrices. If the matrix~$a$ is not scalar,
then it is a~linear combination of a~scalar matrix and
$\diag\big\{1,\varepsilon,\varepsilon^2\big\}$. Hence, by Remark~\ref{1-codim}, $\Image p$ is also the set of $3$-scalar plus
scalar matrices. At any rate, we have shown that $\Image p$ is
either $\{0\}$, $K$, the set of $3$-scalar matrices, the set of
$3$-scalar plus scalar matrices (matrices with eigenvalues
$\big(\alpha +\beta ,\alpha +\beta \varepsilon,\alpha +\beta
\varepsilon^2\big)$), $\sl_3(K)$ (perhaps lacking nilpotent matrices
of order~$3$), or is dense in~$M_3(K)$.

If $\Char K=3$, then by Remark~\ref{char3-classification} the
multilinear polynomial $p$ is either trace vanishing or $\Image p$
is dense in $M_3(K)$. If $p$ is trace vanishing, then by
Theorem~\ref{multi_tr=0_3}, $\Image p$ is one of the following:
\{0\}, the set of scalar matrices, the set of $3$-scalar matrices,
or for each triple $\lambda_1+\lambda_2+\lambda_3=0$ there exists a
matrix $M\in\Image p$ with eigenvalues $\lambda_1,\ \lambda_2$ and
$\lambda_3$.
\end{proof}

\subsection{Deligne trick for algebras of rank 3}\label{Deligne-3}

Here we show proofs of two important theorems using the Deligne
trick (for details see Section~\ref{Deligne}).
\subsubsection[Evaluations of semihomogeneous polynomials on $3\times 3$ matrices]{Evaluations of semihomogeneous polynomials on $\boldsymbol{3\times 3}$ matrices}

\begin{proof}[Proof of Theorem~\ref{semi_tr0_3-3}]
 We define the functions $\omega _k \colon M_n(K) \to
K$ as in Lemma~\ref{extr-3}, and consider the rational function
$H=\frac{\omega_2(p(x_1,\dots,x_m))^3}{\omega_3(p(x_1,\dots,x_m))^2}$
(taking values in $K \cup \{\infty\}$).

If $\omega_2(p)=\omega_3(p)=0$, then each evaluation of $p$ is a
nilpotent matrix, contradicting Amitsur's theorem. Thus, either
$\Image H$ is dense in $K$, or $H$ must be constant.

If $\Image H$ is dense in $K$, then $\Image p$ is dense in $\sl_3$ by Lemma~\ref{extr-3}.

So we may assume that $H$ is a constant, i.e.,
 $\alpha \omega_2^3(p)+\beta \omega_3^2(p)=0$ for
some $\alpha ,\beta \in K$ not both $0$. Fix generic matrices
$Y_1,\dots,Y_m$. We claim that the eigenvalues
$\lambda_1$, $\lambda_2$, $-\lambda_1-\lambda_2$ of
$q:=p(Y_1,\dots,Y_m)$ are pairwise distinct. Otherwise either they are all
equal, or two of them are equal and the third is not, each of
which is impossible by Lemmas~\ref{divAm-3} and \ref{div2-3} since $q
\in \widetilde{\UD}$.

 Let
$\lambda_1'$, $\lambda_2'$, $-\lambda_1'-\lambda_2'$ be the eigenvalues of
another matrix $r\in\Image p$. Thus we have the following:
\begin{gather*}\alpha\omega_2^3(r)+\beta\omega_3^2(r)=\alpha\omega_2^3(q)+\beta\omega_3^2(q)=0.\end{gather*}
Therefore we have homogeneous equations on the eigenvalues.
Dividing by $\lambda_2^6$ and $\lambda_2'^6$ respectively, we have
the same two polynomial equations of degree 6 on
$\frac{\lambda_1}{\lambda_2}$ and $\frac{\lambda'_1}{\lambda'_2}$,
yielding six possibilities for $\frac{\lambda'_1}{\lambda'_2}$.
The six permutations of $\lambda_1$, $\lambda_2$, and
$\lambda_3=-\lambda_1-\lambda_2$ define six pairwise different
$\frac{\lambda'_1}{\lambda'_2}$ unless
$(\lambda_1,\lambda_2,\lambda_3)$ is a permutation (multiplied
by a scalar) of one of the following triples: $(1,1,-2)$, $(1,-1,0)$, $\big(1,\varepsilon,\varepsilon^2\big)$. The first instance is
impossible since the eigenvalues must be pairwise distinct. Here we use the Deligne trick.
The second instance give us an element of Amitsur's algebra
$\widetilde{\UD}$ with eigenvalue~$0$ and thus determinant~0,
contradicting Amitsur's theorem. Here we use the Deligne trick again. In the third instance the polynomial
$p$ is $3$-scalar. Thus, either~$p$ is $3$-scalar, or
each matrix from $\Image p$ will have the same eigenvalues up to
permutation and scalar multiple (which also holds when~$p$ is $3$-scalar).

Assume that for some $i\in\{2,3\}$ that $\tr p^i$ is not
identically zero. Then $\lambda_1^i$, $\lambda_2^i$, and
$\lambda_3^i$ are three linear functions on~$\tr p^i$. Hence we
have the PI (polynomial identity)
$\big(p^i-\lambda_1^i\big)\big(p^i-\lambda_2^i\big)\big(p^i-\lambda_3^i\big)$. Thus by
Amitsur's theorem, one of the factors is a PI. Hence $p^i$ is a
scalar matrix. However $i\neq 2$ by Lemma~\ref{div2-3}. Hence $i=3$,
implying $\Image p$ is the set of matrices with
eigenvalues $\big\{\big(\gamma,\gamma\varepsilon,\gamma\varepsilon^2\big)\colon
\gamma \in K\big\}$.

Thus, we may assume that $p$ satisfies $\tr(p^i)=0$ for $i=1, 2$
and $3$. Now $\omega_1(p)=\tr(p)=0$ and
$2\omega_2(p)=(\tr(p))^2-\tr\big(p^2\big)=0$.

Hence $\omega_1=\omega_2=0$ if $\chara(K) \ne 2$; in this case
$\omega_3$ is either $0$ (and hence $p$ is PI) or not~$0$ (and
hence $p$ is $3$-scalar).

So assume that $\chara (K)=2$. Recall that
\begin{gather*}0=\tr
\big(p^3\big)=\lambda_1^3+\lambda_2^3+\lambda_3^3=\lambda_1^3+\lambda_2^3+\lambda_3^3 -3\lambda_1\lambda_2\lambda_3+3\lambda_1\lambda_2\lambda_3.\end{gather*}
But
$\lambda_1^3+\lambda_2^3+\lambda_3^3-3\lambda_1\lambda_2\lambda_3$
is a multiple of $\lambda_1+\lambda_2+\lambda_3$ (seen by
substituting $-(\lambda_1+\lambda_2)$ for $\lambda_3)$ and thus
equals $0$. Thus, $0 = 3 \lambda_1\lambda_2\lambda_3 =
\lambda_1\lambda_2\lambda_3 =\omega_3(p)$, and the Hamilton--Cayley
equation yields $p^3+\omega_2p=0$. Therefore, $p\big(p^2+\omega_2\big)=0$
and by Amitsur's theorem either $p$ is PI, or $p^2 = -\omega_2$
(which is central), implying by Lemma~\ref{div-3} that $p$ is
central.
\end{proof}

\subsubsection[Evaluations of multilinear trace zero polynomials on $3\times 3$ matrices]{Evaluations of multilinear trace zero polynomials on $\boldsymbol{3\times 3}$ matrices}

\begin{proof}[Proof of Theorem~\ref{multi_tr=0_3}] If the polynomial $\omega_2 (p)$ (defined in the proof of
Theorem~\ref{semi_tr0_3-3}) is identically zero, then the
characteristic polynomial is $p^3-\omega_3(p)=0$, implying $p$ is
either scalar (which can happen only if $\Char (K)=3$) or
$3$-scalar. Therefore we may assume that the polynomial $\omega_2
(p)$ is not identically zero. Let
\begin{gather*}f_{\alpha ,\beta }(M)=\alpha \omega_2(M)^3+\beta \omega_3(M)^2.\end{gather*} It is enough to
show that for any $ \alpha , \beta \in K$ there exists a
non-nilpotent matrix $M= p(a_1,\dots,a_m)$ such that $f_{\alpha
,\beta }(p(a_1,\dots,a_m))=0$, since this will imply that the
image of $H$ (defined in Lem\-ma~\ref{extr-3}) contains all $-\frac {\beta}{\alpha}$ and thus $K
\cup \{\infty\}$. (For example, if $\alpha= 0$ and $\beta \ne 0$, then $ \omega _3(M) = 0$, implying $ \omega _2(M) \ne 0$
since $ \omega _1(M) = 0 $ and $M$~is non-nilpotent, and thus $H = \infty.)$ Therefore,
for any trace vanishing polynomial (i.e., a polynomial
$x^3+\gamma _1 x+\gamma _0$) there is a matrix in $\Image p$ for
which this is the characteristic polynomial. Hence whenever
$\lambda_1+\lambda_2+\lambda_3=0$ there is a matrix with
eigenvalues $\lambda_i$.

We may assume that $a=p(Y_1,\dots,Y_m)$ and $b=p\big(\tilde Y_1,Y_2\dots,Y_m\big)$ are not proportional, for generic matrices $\tilde Y_1,Y_1,\dots,Y_m$, cf.~Lemma~\ref{dim2-2r}. Consider the polynomial $\varphi_{\alpha,\beta }(t)=f_{\alpha ,\beta }(a+tb)$. There are three cases to
consider:

Case I. $\varphi_{\alpha ,\beta }= 0$ identically. Then $f_{\alpha,\beta }(a)=0$, and $a$ is not nilpotent by Proposition~\ref{Am1-2}. Here we use the Deligne trick.

Case II. $\varphi_{\alpha ,\beta }$ is a constant
 $\tilde \beta \ne 0$. Then
$f_{\alpha ,\beta }(b+ta)=t^6\varphi_{\alpha ,\beta }\big(t^{-1}\big)=\tilde
\beta t^6$; thus $f_{\alpha ,\beta }(b)=0$, and $b$ is not nilpotent
by Proposition~\ref{Am1-2}. Here we use the Deligne trick.

Case III. $\varphi_{\alpha ,\beta }$ is not constant. Then it
has finitely many roots. Assume that for each substitution $t$ the
matrix $a+tb$ is nilpotent; in particular, $\omega_2(a+tb)=0$.
Note that $\omega_2(a+tb)$ equals the sum of principal $2\times
2$ minors and thus is a quadratic polynomial (for otherwise
$\omega_2(b)=0$ which means that $\omega_2(p)$ is identically
zero, a contradiction). Hence $\omega_2(a+tb)$ has two roots,
which we denote as $t_1$ and $t_2$. If $t_1=t_2$, then $t_1$ is
 uniquely defined and thus, in view of Remark~\ref{patch1-int}, is a rational function in the entries of $a$ and $b$, and $a+t_1b$ is
a nilpotent rational function (because we assumed that one of
$a+t_1b$ and $a+t_2b$ is nilpotent, but here they are equal.) At
least one of $t_1$ and $t_2$ is a root of $\varphi_{\alpha ,\beta
}$.

If only $t_1$ is a root, then $t_1$ is uniquely defined and thus, by
Remark~\ref{patch1-int}, is a rational function; hence, $a+t_1b$ is
a nilpotent polynomial, contradicting Proposition~\ref{Am1-2}. Here we use the Deligne trick. Thus,
we may assume that both $t_1$ and $t_2$ are roots of
$\varphi_{\alpha ,\beta }$. But $\varphi_{\alpha ,\beta }(t_i)$ is
nilpotent, and in particular $\omega_3(a+t_ib)=0$. Thus there exists
exactly one more root $t_3$ of $\omega_3(a+tb)$, which is uniquely
defined and thus, by Remark~\ref{patch1-int}, is rational. Hence we
may consider the polynomial $q(x_1,\dots,x_m,\tilde x_1)=a+t_3b$,
which must satisfy the condition $\tr (q)=\det (q)=0$. This is
impossible for homogeneous $q$ by Theorem \ref{semi_tr0_3-3}, and
also impossible for nonhomogeneous $q$ since the leading homogenous
component $q_d$ would satisfy $\tr (q_d)=\det (q_d)=0$, a
contradiction.
 \end{proof}

\subsection{The Euler graph approach}\label{Euler-ch}

We recall the following elementary graph-theoretic Lemma~\ref{graph-2r}: {\it Let $p$ be a multilinear polynomial.
If the $a_i$ are matrix units, then $p(a_1,\dots,a_m)$ is
either~$0$, or~$c\cdot e_{ij}$ for some $i\neq j$, or a~diagonal matrix.}

\begin{proof} We connect a vertex $i$ with a
vertex $j$ by an oriented edge if there is a matrix $e_{ij}$ in our
set $\{a_1,a_2,\dots,a_m\}$. The evaluation $p(a_1,\dots,a_m)\neq
0$ only if there exists an Eulerian cycle or an Eulerian path in the
graph. There exists an Eulerian path only if the degrees of all
vertices but two are even, and the degrees of these two vertices are
odd. Also we know that there exists an Eulerian cycle only if the
degrees of all vertices are even. Thus when $p(a_1,\dots,a_m)\neq
0$, there exists either an Eulerian path or cycle in the graph. In
the first case we have exactly two vertices of odd degree such that
one of them ($i$) has more output edges and another ($j$) has more
input edges. Thus the only nonzero terms in the sum of our
polynomial can be of the type $ce_{ij}$ and therefore the result
will also be of this type. In the second case all degrees are even.
Thus there are only cycles and the result must be a diagonal matrix.
\end{proof}
\begin{Theorem}\label{dense}
If there exist $\alpha ,\ \beta $, and $\gamma$ in $K$ such that
$\alpha +\beta +\gamma,\ \alpha +\beta
\varepsilon+\gamma\varepsilon^2$ and $\alpha +\beta
\varepsilon^2+\gamma\varepsilon$ are nonzero, together with matrix
units $e_1,e_2,\dots,e_m$ such that $p(e_1,e_2,\dots,e_m)$ has
eigenvalues~$\alpha$,~$\beta $ and~$\gamma$, then $\Image p$ is
dense in $M_3$.
\end{Theorem}
\begin{proof}Define $\chi$ to be the permutation of the set of matrix units,
sending the indices $1\to 2$, $2\to 3$, and $3\to 1$. For example,
$\chi(e_{12})=e_{23}$.
For triples $T_1,\dots,T_m$ (each $T_i=(t_{i,1},t_{i,2},t_{i,3})$)
consider the function
\begin{gather}
f(T_1,\dots,T_m)=p\big(t_{1,1}x_1+t_{1,2}\chi(x_1)+t_{1,3}\chi^{-1}(x_1),t_{2,1}x_2+t_{2,2}\chi(x_2)
+t_{2,3}\chi^{-1}(x_2),\nonumber\\
\hphantom{f(T_1,\dots,T_m)=p\big(}{} \dots,t_{m,1}x_m +t_{m,2}\chi(x_m)+t_{m,3}\chi^{-1}(x_m)\big).\label{mapping_triples}
\end{gather}

Opening the brackets, we have $3^m$ terms, each of which we claim is a diagonal matrix. Each term is a monomial with
coefficient of the type
\begin{gather*}\chi^{k_{\pi(1)}} \chi^{k_{\pi(2)}}\cdots \chi^{k_{\pi(m)}}x_{\pi(1)}x_{\pi(2)}\cdots x_{\pi(m)},\end{gather*}
where $k_i$ is $-1$, $0$ or $1$, and $\pi$ is a permutation. Since we
substitute only matrix units in $p$, by Lemma~\ref{graph-2r} the
image is either diagonal or a matrix unit with some coefficient.
For each of the three vertices $v_1$, $v_2$, $v_3$ in our graph define
the index $\iota _\ell $, for $1 \le \ell \le 3$ to be the
number of incoming edges to $v_\ell$ minus the number of outgoing
edges from $v_\ell$. Thus, at the outset, when the image is
diagonal, we have $\iota _1 = \iota _2 = \iota _3 = 0$.

We claim that after applying $\chi$ to any matrix unit the new
$\iota'_\ell$ will all still be congruent modulo 3.
Indeed, if the edge $\vec{12}$ is changed to $\vec{23}$, then
$\iota'_1 = \iota +1$ and $\iota_3' = \iota_3+1$, whereas
$\iota'_2 = \iota_2 -2 \equiv \iota_2 +1$.
 The same with changing
$\vec{23}$ to $\vec{31}$ and $\vec{31}$ to $\vec{12}$. If we make
the opposite change $\vec{21}$ to $\vec{13}$, then (modulo 3) we
subtract~$1$ throughout. If we make a change of the type
$\vec{ii}\mapsto \vec{jj}$, then $\iota_\ell ' = \iota_\ell $ for
each~$\ell$.

If $p\big(\chi^{k_1}x_1,\chi^{k_2}x_2,\dots,\chi^{k_m}x_m\big)=e_{ij}$,
this means that the number of incoming edges minus the number of
outgoing edges of the vertex $i$ is $-1 \pmod 3$ and the number
of incoming edges minus the number of outgoing edges of $j$ is
$1\pmod 3$, which are not congruent modulo $3$. Thus the values of
the mapping $f$ defined in~\eqref{mapping_triples} are diagonal
matrices. Now fix $3m$ algebraically independent triples
$T_1,\dots,T_m,\Theta_1,\dots,\Theta_m,\Upsilon_1,\dots,\Upsilon_m$.
Assume that $\Image f$ is $2$-dimensional. Then $\Image df$ must
also be $2$-dimensional at any point. Consider the differential
$df$ at the point $(\Theta_1,T_2,\dots,T_m)$. Thus,
\begin{gather*}f(\Theta_1,T_2,\dots,T_m), \qquad f(T_1,T_2,\dots,T_m), \qquad
f(\Theta_1,\Theta_2,\dots,T_m)\end{gather*} belong to $\Image df$. Thus
these three matrices must span a linear space of dimension not
more than~$2$. Hence they lie in some plane~$P$. Now take
\begin{gather*}f(\Theta_1,\Theta_2,T_3,\dots,T_m),\qquad
f(\Theta_1,T_2,T_3,\dots,T_m), \qquad
f(\Theta_1,\Theta_2,\Theta_3,T_4,\dots,T_m).\end{gather*} For the same reason
they lie in a plane, which is
 the plane $P$ because it has
two pure quaternions from~$P$. By the same argument, we conclude that all
the matrices of the type
$f(\Theta_1,\dots,\Theta_k,T_{k+1},\allowbreak \dots,T_m)$ lie in P. Now we
see that \begin{gather*}f(\Theta_1,\dots,\Theta_{m-1},T_m),\qquad
f(\Theta_1,\dots,\Theta_m), \qquad f(\Upsilon_1,\Theta_2,\dots,\Theta_m)\end{gather*} also lie in $P$.
Analogously we obtain that also
\begin{gather*}f(\Upsilon_1,\dots,\Upsilon_k,\Theta_{k+1},\dots,\Theta_m)\in
P\end{gather*} for any~$k$.

 Hence for $3m$ algebraically independent triples
\begin{gather*}T_1,\dots,T_m;\qquad \Theta_1,\dots,\Theta_m;\qquad \Upsilon_1,\dots,\Upsilon_m,\end{gather*}
we have obtained that $f(T_1,\dots,T_M)$,
$f(\Theta_1,\dots,\Theta_m)$ and $f(\Upsilon_1,\dots,\Upsilon_m)$
lie in one plane. Thus any three values of~$f$, in particular
$\diag\{\alpha ,\beta ,\gamma\}$, $\diag\{\beta ,\gamma,\alpha \}$
and $\diag\{\gamma,\alpha ,\beta \}$, must lie in one plane. We
claim that this can happen only if \begin{gather*}\alpha +\beta +\gamma=0,\qquad
\alpha +\beta \varepsilon+\gamma\varepsilon^2=0,\qquad
\text{or}\qquad \alpha +\beta \varepsilon^2+\gamma\varepsilon=0.\end{gather*}

Indeed, $\diag\{\alpha ,\beta ,\gamma\}$, $\diag\{\beta
,\gamma,\alpha \}$ and $\diag\{\gamma,\alpha ,\beta \}$, are
dependent if and only if the matrix
\begin{gather*} \left( \begin{matrix} \alpha & \beta & \gamma \\ \beta & \gamma& \alpha \\ \gamma& \alpha & \beta \end{matrix}\right)\end{gather*}
is singular, i.e., its determinant $3\alpha \beta \gamma -\big(\alpha
^3+\beta ^3+\gamma^3\big) = 0$. But this has the desired three roots
when viewed as a cubic equation in~$\gamma$.

We have a contradiction to our hypothesis.
\end{proof}

\begin{Remark}\label{1-codim}
If there exist $\alpha$, $\beta $, and $\gamma$ such that $\alpha
+\beta +\gamma=0$ but $(\alpha ,\beta ,\gamma)$ is not
proportional to $\big(1,\varepsilon,\varepsilon^2\big)$ or
$\big(1,\varepsilon^2,\varepsilon\big)$,
 with
matrices $ A_1,A_2,\dots,A_m $ such that $p(A_1,A_2,\dots,A_m)$
has eigenvalues $\alpha$, $\beta $ and $\gamma$, then either all
diagonalizable trace zero matrices lie in $\Image p$, or $\Image
p$ is dense in~$M_3(K)$. If $\alpha +\beta
\varepsilon+\gamma\varepsilon^2=0$ but $(\alpha ,\beta ,\gamma)$
is not proportional to $\big(1,\varepsilon,\varepsilon^2\big)$ or
$(1,1,1)$, then all diagonalizable matrices with eigenvalues
$\alpha +\beta ,\ \alpha +\beta \varepsilon$ and $\alpha +\beta
\varepsilon^2$ lie in $\Image p$ or $\Image p$ is dense in~$M_3(K)$.
\end{Remark}
\begin{Remark}\label{char3-classification}
The proof of Theorem \ref{dense} works also for any field $K$ of
characteristic $3$. In this case $\varepsilon=1$. Hence, if there
are $\alpha$, $\beta $, and $\gamma$ in $K$ such that \begin{gather*}\alpha
+\beta +\gamma\ne 0,\end{gather*} together with matrix units
$e_1,e_2,\dots,e_m$ such that $p(e_1,e_2,\dots,e_m)$ has eigenvalues
$\alpha$, $\beta $ and $\gamma$, then $\Image p$ is dense in~$M_3$.
Therefore, for $\Char K=3$, any multilinear polynomial~$p$ is either
trace vanishing or $\Image p$ is dense in~$M_3(K)$.
\end{Remark}

\subsection{Open problems related to the rank 3 case}

\begin{Problem}\label{3s-3}Does there actually exist a multilinear polynomial whose image
evaluated on $3\times3$ matrices consists of $3$-scalar matrices?
\end{Problem}

\begin{Problem}\label{s3s}
Does there actually exist a multilinear polynomial whose image
evaluated on $3\times3$ matrices is the set of scalars plus
$3$-scalar matrices?
\end{Problem}

\begin{Remark}
Problems \ref{3s-3} and \ref{s3s} both have the same answer. If
they both have affirmative answers, such a polynomial would provide a
counter-example to Kaplansky's problem.
\end{Remark}

 \begin{Problem}
Is it possible that the image of a multilinear polynomial
evaluated on $3\times3$ matrices is dense but not all of $M_3(K)$?
\end{Problem}
 \begin{Problem}
Is it possible that the image of a multilinear polynomial
evaluated on $3\times3$ matrices is the set of all trace vanishing matrices excluding discriminant vanishing diagonalazable matrices?
\end{Problem}
\begin{Problem}
Give a classification of all possible evaluations of homogeneous polynomials on~$M_3(K)$ with respect to Zariski closure.
\end{Problem}
\begin{Remark}
The working hypothesis is that there are $6$ Zariski closures of image sets of homogeneous polynomials:
$\{0\}$, $K$, $V_3(K)$, $\ssl_3(K)$, $V_3+K,M_3(K)$. However the problem remains open.
\end{Remark}
\begin{Problem}
Investigate all possible image sets of non homogeneous polynomials in order to obtain a zoo of interesting examples.
\end{Problem}
\begin{Example}
Here are some interesting examples.
\begin{itemize}\itemsep=0pt
\item Let $p(x_1,\dots, x_m)$ be a homogeneous $3$-central polynomial and $h(x)$ be any polynomial in one variable. Then, $h(p(x_1,\dots, x_m))$ is the polynomial having evaluations with triples of eigenvalues belonging to the set $\big(h(c),h(c\varepsilon),h\big(c\varepsilon^2\big)\big)$ for all $c\in K$. This is a $7$-dimensional image.
\item Let $h(x)$ be again an arbitrary polynomial in one variable, and consider the polynomial $h([x,y])$.
This polynomial image is the set of all matrices having triples of eigenvalues belonging to the set $(h(a),h(b),h(-a-b))$ for all pairs $a,b\in K$. This is an $8$-dimensional image.
\end{itemize}
\end{Example}
\begin{Problem}
Give a classification of all possible evaluations of multilinear polynomials on $M_3(\R)$ with respect to the standard topology.
\end{Problem}
\begin{Remark}
Although we cannot answer a question about existence of multilinear $3$-central polynomials, but we know that multilinear $3$-central polynomials with real coefficients do not exist. Indeed, according to Lemma~\ref{diag-not-scal-3} any multilinear polynomial has a~diagonal not scalar matrix evaluation (with real entries), and these real entries are its eigenvalues. Note that a~nonzero $3$-central matrix cannot have three real eigenvalues.
\end{Remark}
\begin{Problem}
Investigate possible evaluations of multilinear polynomials on the simple non special Jordan algebra, which is denoted as $H(C_3)$ in~{\rm \cite{ZSSS}} and $C_{27}$ in~{\rm \cite{Sh}}.
\end{Problem}
\section{The case of high rank}\label{pcp}
The Euler graph approach is useful also for $n>3$:
\subsection{The Euler graph approach for algebras of high rank}

\subsubsection{The idea of the method}\label{euler-d}

Define $\chi$ to be the permutation of the set of matrix units,
sending the index $i\mapsto i+1$ for $1\leq i\leq n-1$, and
$n\mapsto 1$. For example, $\chi(e_{12})=e_{23},\
\chi(e_{57})=e_{68}$. Since we substitute only matrix units into
$p$, Lemma~\ref{graph-2r} shows that the image is either diagonal or
a matrix unit with some coefficient. Consider the corresponding
graph~$\Gamma$. Consider the sum of $(i-j)$ over edges
$(i\rightarrow j)$ of $\Gamma$. If the graph is an Eulerian cycle
then this sum is $0$, and if it is an Eulerian path from $k$ to
$\ell$ then this sum equals $k-\ell$. Take matrix units $a_1,\dots,a_m$ such that $p(a_1,\dots,a_m)$ is a diagonal matrix for some $\a \in K$. We may assume that $ a_1\cdots
a_m=D$. Writing $a_\ell = e_{i_\ell, j_\ell}$, we define
$\iota (a_\ell) = i_\ell - j_\ell$. Thus $i_1 = j_m$
and $\sum \iota (a_\ell) = 0$. Then $ \chi ^k (a_\ell) =
e_{i_\ell+k, j_\ell+k}$ (taken modulo~$n)$, implying
\begin{gather*}\iota \big(\chi ^k (a_\ell) \big) \equiv ( i_\ell+k) - (j_\ell+k) =
i_\ell - j_\ell = \iota (a_\ell) \pmod n. \end{gather*}

Consider
\begin{gather}\label{mapping_sets-pcp}
f(a_1,\dots,a_m)
=p\left( \sum _{k_1=1}^m t_{k_1,1}\chi ^{k_1} ( a_1),\dots,
 \sum _{k_m=1}^m t_{k_m,m}\chi ^{k_m} ( a_m)\right),
\end{gather}
where the $t_{k,\ell}$ are commuting indeterminates.
Opening the brackets, we have $n^m$ terms, each of the form
\begin{gather*}a' = \chi^{k_1}(a_{\pi(1)})\cdots \chi^{k_m}(a_{\pi(m)}),\end{gather*}
which, if nonzero, must have \begin{gather*}\iota(a') \equiv \sum _{\ell = 1}^m \iota \big(\chi^{k_\ell} a_{\pi(\ell)}\big)\equiv \sum _{\ell = 1}^m \iota (a_{\pi(\ell)}) \equiv 0\pmod n, \end{gather*}
implying $a'$ is a diagonal matrix.
 Hence $\Image f\subseteq\Image p$ contains only diagonal matrices.
This helps us to obtain a good subset in the set of diagonal properties, in particular matrices with needed set of eigenvalues, and each obtained matrix is an evaluation of the polynomial $p$.
\subsubsection{Applications of the method}

In this subsection we will show how this method is used in the proof
of several important results: Theorems~\ref{thmB1-pcp}, \ref{no-mpc-pcp} and~\ref{harmonic-4-int}.

\begin{proof}[Proof of Theorem~\ref{thmB1-pcp}]
Define $\chi$ as in Section \ref{euler-d} By Lemma \ref{graph-2r} there
exist matrix units $a_1,\dots,a_m$ with $p(a_1,\dots,a_m)=\a
e_{12}$ for some $\a \in K$. We may assume that $ a_1\cdots
a_m=e_{12}$. Writing $a_\ell = e_{i_\ell, j_\ell}$, we define
$\iota (a_\ell) = i_\ell - j_\ell$. Thus $i_1 = 1$ and $j_m = 2$,
and $\sum \iota (a_\ell) = 1$. Then $ \chi ^k (a_\ell) =
e_{i_\ell+k, j_\ell+k}$ (taken modulo $n)$, implying
\begin{gather*}\iota (\chi ^k (a_\ell) ) \equiv ( i_\ell+k) - (j_\ell+k) =
i_\ell - j_\ell = \iota (a_\ell) \pmod n. \end{gather*}

Consider $f(a_1,\dots,a_m)$ defined in \eqref{mapping_sets-pcp}.
Opening the brackets, we have $n^m$ terms, each of the form
\begin{gather*}
a' = \chi^{k_1}(a_{\pi(1)})\cdots \chi^{k_m}(a_{\pi(m)}),
\end{gather*}
which, if nonzero, must have
\begin{gather*}\iota(a') \equiv \sum _{\ell = 1}^m \iota \big(\chi^{k_\ell} a_{\pi(\ell)}\big)\equiv \sum _{\ell = 1}^m \iota (a_{\pi(\ell)}) \equiv 1\pmod n,
\end{gather*}
implying $a'$ is a matrix of the form $ce_{i,i+1}$ or $ce_{n,1}$.
 Hence $\Image f\subseteq\Image p$ has the form
\begin{gather*} a = \left(
\begin{matrix} 0 & * & 0 & \dots & 0
\\ 0 & 0 & * & \dots & 0 \\ \vdots & \vdots & \ddots & \ddots & \vdots \\
 0 & 0 & 0 & \ddots & * \\ * & 0 & \dots & 0 & 0 \end{matrix}\right)
 .\end{gather*}
Each of the starred entries of $a$ is a polynomial with respect to
$t_{k,i}$ and each of them takes nonzero values because
$e_{k,k+1}$ belongs to the image of $f$ for any $1\leq k\leq n-1$
and also $e_{n,1}\in\Image(f)$. Therefore for generic $t_{k,\ell}$
each of the starred entries is nonzero, so the minimal polynomial
of $a$ is $\lambda ^ n - \alpha$ for some $\a$, implying $a$ has
eigenvalues $\{c,c\varepsilon,\dots,c\varepsilon^{n-1}\}$ where
$c$ is the $n$-th root of the determinant $\alpha$.
\end{proof}

\begin{Remark}\label{remharmonic-pcp} Assume that $K $ has the form
$F[\varepsilon]$, where $\varepsilon$ is a primitive $n$-th root of~1. Let $\overline{\phantom w}$ denote the automorphism of $K$
sending $\varepsilon\mapsto \varepsilon^{-1}$.

Let us introduce the tool of ``harmonic bases'' of the space of
diagonal matrices.
 There is a~base of the matrices $e_k$ for $0\leq k\leq n-1$, where $e_k=\diag\big\{1,\varepsilon^k,\varepsilon^{2k},\dots,
 \varepsilon^{(n-1)k}\big\}$.
 Assume that there exist matrix units $a_1,\dots, a_m$ such that $p(a_1,\dots,a_m)=\diag\{c_0,\dots,c_{n-1}\}$.
This can be written as a linear combination of the~$e_k$. If, for
some $k$, $e_k$ appears in
 this sum with nonzero coefficient, then $e_k$ belongs to the linear span of $\big\langle M, \chi(M), \chi^2(M),\dots\big\rangle$,
 where
 $\chi(\diag\{\lambda_1,\dots,\lambda_n\})=\diag\{\lambda_n,\lambda_1,\dots,\lambda_{n-1}\}$.
 Therefore if we have a set of matrix units $a_i$ such that $p(a_1,\dots,a_m)=M$, then
 we construct a multilinear mapping $f$ whose image will either have at least dimension~3 or the image will be a linear set and
 therefore~$M$ will be a linear combination of no more than two base elements~$e_k$.

 Assume also that the image of $p$ is at most $\big(n^2-n+2\big)$-dimensional.
By Remark~\ref{dim-f1-pcp}, the image of $f$ constructed in the proof of Theorem~\ref{thmB1-pcp}
 is at most $2$-dimensional and thus is a~linear space.
 If $p(a_1,\dots,a_m)=h_0e_0+h_1e_1+\dots+h_{n-1}e_{n-1}$ with $h_k\neq 0$, then
 $e_k$ belongs to the linear span of $\Image f$.
 Hence there are at most two nonzero coefficients, say, $h_k$ and $h_l$ with all of the
 others
 zero.
 We can consider the scalar product
\begin{gather*} \langle \{ \alpha _1, \dots, \alpha _n \} \{ \beta _1, \dots, \beta _n \}\rangle
= \sum _{i=1}^n \alpha_i \overline{\beta_i}.\end{gather*}
 We compute $\langle \{c_0,\dots,c_{n-1}\},e_s \}\rangle $ in two ways, first as
 $nq_s$
 and then as
 \begin{gather*}c_0+c_1\varepsilon^{-s}+\dots+c_{n-1}\varepsilon^{-(n-1)s}\end{gather*}
 since $\overline{\varepsilon^{l}}=\varepsilon^{-l}$.
\end{Remark}

\begin{proof}[Proof Theorem~\ref{no-mpc-pcp}] There exist matrix units $a_i$ such that
 $p(a_1,\dots,a_m)=\diag\{c_0,\dots,\allowbreak c_{n-1}\}$
 is diagonal but not scalar.
 If at least one of the $c_i$ were zero,
 then each $c_i$ would be zero and therefore this matrix is zero,
 a contradiction. Thus,
 without loss of generality $c_0\neq c_1$ are nonzero.
 Therefore there also exist matrix units $\tilde a_i$ such that $p(\tilde a_1,\dots,\tilde a_m)=\diag\{c_1,c_0,c_2,\dots,c_{n-1}\}$.

 Assume that the image of $f$ as constructed in \eqref{mapping_sets-pcp} is
 $1$-dimensional. Then, for each~$i$,
 $\diag\{c_0,\dots,c_{n-1}\}$ is proportional
 to $\diag\{c_i,c_{i+1},\dots,c_{n-1},c_0,\dots,c_{i-1}\}$.
We can construct the mappings $f$ and $\tilde f$ as before, and
their images cannot be both $1$-dimensional since otherwise
 \begin{gather*}\tau=\frac{c_1}{c_0}=\frac{c_2}{c_1}=\frac{c_3}{c_2}=\dots=\frac{c_n}{c_{n-1}}=\frac{c_0}{c_n},\end{gather*}
 and also $\tilde\tau=\frac{c_0}{c_1}=\frac{c_2}{c_0}=\frac{c_3}{c_2}= \cdots$. Hence, since $n-1 \ge 3$. \begin{gather*}\tau^2=\frac{c_2}{c_1}\cdot\frac{c_1}{c_0}=\frac{c_2}{c_0}=\tilde\tau=\frac{c_3}{c_2}=\tau.\end{gather*} Thus $\tau \in \{0,1\}$.
 If $\tau=1$ then
 $p(a_1,\dots,a_m)$ is scalar, a contradiction. If $\tau=0$ then $c_1=0$, a~contradiction.

We conclude that $\Image f$ is least $2$-dimensional and $ \Image p$ is at least $\big(n^2-n+2\big)$-dimensional.

 Assume that the image of $p$ is at most $\big(n^2-n+2\big)$-dimensional.
 As we showed in Remark~\ref{remharmonic-pcp}, the matrix $\diag\{c_0,\dots,c_{n-1}\}$
 can be written as $\alpha e_k+\beta e_l$,
which is not scalar. Without loss of generality we may assume that
$c_0\neq c_1$
 (because there exists $r$ such that
 $c_r\neq c_{r+1})$, and we now consider
 the matrix \begin{gather*}\diag\{c_r,c_{r+1},\dots,c_{n-1}\,c_0,c_1,\dots,c_{r-1}\}\end{gather*}
with its different coefficients $\tilde
\alpha=\varepsilon^{rk}\alpha$ and $\tilde \beta
=\varepsilon^{rl}\beta $).

We define the matrices
$q_k:=\diag\big\{\varepsilon^k,1,\varepsilon^{2k},\varepsilon^{3k},\dots,
 \varepsilon^{(n-1)k}\big\}$.
 Switching the indices 1 and 2, we obtain
 matrix units $\tilde a_i$ such that
 \begin{gather*}p(\tilde a_1, \dots,\tilde a_m)=\diag\{c_1,c_0,c_2,c_3,\dots,c_{n-1}\} = \alpha q_k+\beta
 q_l.\end{gather*}
By Remark~\ref{remharmonic-pcp}, $\alpha q_k+\beta q_l$ also can be written as a linear combination of two elements
 of the base $e_s$
 (say, $\tilde \alpha e_{\tilde k}+\tilde \beta e_{\tilde l}$).
 Note that
 \begin{gather*}\langle q_k,e_s\rangle=\varepsilon^k+\varepsilon^{-s}-1-\varepsilon^{k-s}=
 \big(\varepsilon^k-1\big)\big(1-\varepsilon^{-s}\big)+ \langle e_k,e_s\rangle .\end{gather*}

Thus, if $k\neq s$, then \begin{gather*}\langle
q_k,e_s\rangle=\varepsilon^k+\varepsilon^{-s}-1-\varepsilon^{k-s}=
 \big(\varepsilon^k-1\big)\big(1-\varepsilon^{-s}\big)\end{gather*} since $\langle
 e_k,e_s\rangle= 0$,
 and if $k=s$ then \begin{gather*}\langle q_k,e_s\rangle=\big(\varepsilon^k-1\big)\big(1-\varepsilon^{-s}\big)+n.\end{gather*}
 Hence, if $s\notin\{k,l\}$ then \begin{gather*}\langle \alpha q_k+\beta q_l,e_s\rangle=\big(1-\varepsilon^{-s}\big)\big(\alpha\big(\varepsilon^k-1\big)+\beta \big(\varepsilon^l-1\big)\big).\end{gather*}
 We denote $\delta =\alpha\big(\varepsilon^k-1\big)+\beta \big(\varepsilon^l-1\big)$.
 Recall that $c_1\neq c_0$, and thus $\delta \neq 0$. Therefore either $s=\tilde k$, or else $s=\tilde l$, or
 $\langle \alpha q_k+\beta q_l,e_s\rangle=0$ (and thus
 $s$ is either $k$, or $l$, or $1-\varepsilon^{-s}=0$ (and thus $s=0$) -- only five possibilities.
 But for $n\geq 6$ there are at least three nonzero coefficients, a contradiction.

 Thus we may assume that $n=5$.
 We have exactly five possibilities for $s$, which therefore must be distinct.
 Therefore the $k$-th and $l$-th coefficients of $\alpha
 q_k+\beta q_l$ will be zero, i.e.,
 \begin{gather*}\big(1-\varepsilon^{-k}\big)\delta +5\alpha=\big(1-\varepsilon^{-l}\big)\delta +5\beta =0,\end{gather*}
 where $\delta =\alpha\big(\varepsilon^k-1\big)+\beta \big(\varepsilon^l-1\big)$.
 In particular \begin{gather*}\frac{\alpha}{\beta }=\frac{1-\varepsilon^{-k}}{1-\varepsilon^{-l}}.\end{gather*}
 Now let us take matrix units $a_i'$ such that $p(a_1',\dots,a_m')=\diag\{c_2,c_1,c_0,c_3,c_4\}$.
 Then $\alpha r_k+\beta r_l$ can also be written as a linear combination of two of the $e_s$, where
 $r_k=\diag\big\{\varepsilon^{2k},\varepsilon^k,1,\varepsilon^{3k},\varepsilon^{4k}\big\}$.
If $k\neq s$, then \begin{gather*}\langle
r_k,e_s\rangle=\varepsilon^{2k}+\varepsilon^{-2s}-1-\varepsilon^{2k-2s}=
 \big(\varepsilon^{2k}-1\big)\big(1-\varepsilon^{-2s}\big).\end{gather*}
 We perform the same calculations as before, and obtain
 \begin{gather*}\frac{\alpha}{\beta }=\frac{1-\varepsilon^{-2k}}{1-\varepsilon^{-2l}}.\end{gather*}
 Hence \begin{gather*}\frac{1-\varepsilon^{-k}}{1-\varepsilon^{-l}}=
 \frac{1-\varepsilon^{-2k}}{1-\varepsilon^{-2l}},\end{gather*} implying
 \begin{gather*}\frac{1+\varepsilon^{-k}}{1+\varepsilon^{-l}}=1,\end{gather*} and hence $k=l$, a contradiction.
 \end{proof}

\begin{proof}[Proof of Theorem~\ref{harmonic-4-int}] First note that $4^2-4+2 = 14$, so $\dim \Image p \ge 14$.
 Assume that $p$ is a multilinear polynomial evaluated on $4\times 4$ matrices with $14$-dimensional image.
 Let $a_1,\dots,a_m$ be any matrix units such that $p(a_1,\dots,a_m)$ is diagonal but not scalar.
 Let $p(a_1,\dots,a_m)=\diag\{c_0,c_1,c_2,c_3\}$ and $c_0\neq c_1$.
 We use the same notation as in the proof of Theorem \ref{harmonic-4-int}.
 Recall that $e_k=\diag\big\{1,i^k,i^{2k},i^{3k}\big\}$ and $q_k=\diag\big\{i^k,1, i^{2k},i^{3k}\big\}$.
 As in the proof of Theorem~\ref{harmonic-4-int},
 $\langle \alpha q_k+\beta q_l , e_s\rangle= \delta \big(1-i^{-s}\big)$ if $s\notin\{k,l\}$,
 or $\delta \big(1-i^{-s}\big)+4\alpha $ if $s=k$ and
 $\delta \big(1-i^{-s}\big)+4\beta $ if $s=l$.
 Therefore $k$ and $l$ are nonzero (for otherwise we have two nonzero possibilities for $s\notin \{k,l\}$
 and one other nonzero coefficient would be zero: $\langle \alpha q_k+\beta q_l , e_0\rangle=4\alpha$ (if we assume $k=0$ without loss of
 generality).
 Therefore $p(a_1,\dots,a_m)$ belongs to the linear span $\langle e_1,e_2,e_3\rangle$.
 Hence we have three options:
 \begin{itemize}\itemsep=0pt
\item $p(a_1,\dots,a_m)=\alpha e_1+\beta e_2$,
\item $p(a_1,\dots,a_m)=\alpha e_1+\beta e_3$,
\item $p(a_1,\dots,a_m)=\alpha e_3+\beta e_2$.
\end{itemize}
We will not treat the last case since its calculations are as in the first case.
Let us consider the first case $p(a_1,\dots,a_m)=\alpha e_1+\beta e_2$.
Therefore $p(\tilde a_1,\dots,\tilde a_m)=\alpha q_1+\beta q_2$ which can be written explicitly as
\begin{gather*}\left(\frac{1}{2}\alpha-\frac{1+i}{2}\beta\right)e_1+\frac{i-1}{2}\alpha e_2+\left(\frac{i}{2}\alpha+\frac{i-1}{2}\beta\right)e_3,\end{gather*}
thus $\alpha\in\{0,(1+i)\beta,-(1+i)\beta\}$.
If $\alpha=(1+i)\beta$ then $p(a_1,\dots,a_m)=\beta\diag\{i-2,i+2,-i,-i\}$.
If $\alpha=-(1+i)\beta$ then $p(a_1,\dots,a_m)=\beta\diag\{-i,-i,i+2,i-2\}$.
In both cases $\alpha=\pm (1+i)\beta$, so there are matrix units $\tilde a_i$ such that
$p(\tilde a_1,\dots,\tilde a_m)=\diag\{-i,i+2,-i,i-2\}$ which can be written explicitly as
$-ie_1-ie_2+ie_3$, and we have three nonzero coefficients. We conclude that the image is at least $15$-dimensional.
If $\alpha=0$, then the value $p(a_1,\dots,a_m)$ has eigenvalues $(c,c,-c,-c)$ as in the conditions of the theorem.

Assume now $p(a_1,\dots,a_m)=\alpha e_1+\beta e_3$.
Therefore $p(a_1,\dots,a_m)=\diag\{x,y,-x,-y\}$.
Then we consider $p(\tilde a_1,\dots,\tilde a_m)=\diag\{x,-x,y,-y\}$ which can be written explicitely as
\begin{gather*}\frac{(x-y)(1+i)}{4}e_1+\frac{x+y}{2}e_2+\frac{(x-y)(1-i)}{4}e_3,\end{gather*}
therefore $x=\pm y$ and once again we have a matrix from the conditions of the theorem.

Therefore there is a set of matrix units $a_i$ with $p(a_1,\dots,a_m)=\diag\{c,c,-c,-c\}$.
Now we construct the mapping $f$ and the image will be $2$-dimensional if and only if it is the set
$\diag\{\lambda_1,\lambda_2,-\lambda_1,-\lambda_2\}$.
Therefore $\Image p$ contains all the matrices with such eigenvalues, which is a $14$-dimensional variety.
Hence, if $\dim\Image p = 14$, then $\Image p$ is exactly this variety.
 \end{proof}

\subsubsection{Questions related to the Euler graph approach}

By Lemma \ref{dim2-2r}, the image of any multilinear polynomial
either is
 a vector space or is at least $3$-dimensional. This is why,
when in Section~\ref{Euler-ch} we considered a $3$-dimensional case
and we obtained the needed result up to Zariski closure. However the
question of using this approach for dimension larger than $3$
remains open:
\begin{Question}
Is it possible to use this method in order to enlarge the dimension
investigated?
\end{Question}

Assume we have a periodic sequence with period $n$ (i.e., sequence $\{a_k\}$ such that $a_{k+n}=a_k$ for all $k\in\N$).
 The dimension of the space spanned by its shifts (i.e., sequences $\{b_k=a_{k+i}\}$, $i=n$ equals the number of harmonics
 in its discrete Fourier transform. From the other side, using conjugation we can permute numbers in the period, obtaining different periodic sequences.
\begin{Question}
 What is the maximal possible number of harmonics one can obtain for a fixed sequence?
\end{Question}

\subsection{Non-trivial images: power-central polynomials on matrices}\label{pcps}

Images of noncommutative polynomials can be nontrivial and
investigation of them requires advanced theory. As a good example of
such polynomials one con consider power-central polynomials.

Note that a problem of existence of $3$-central polynomials occurred
in Section \ref{3a}.

Any multilinear non-central polynomial $p$ (in several
noncommuting variables) takes on values of degree $n$ in the matrix
algebra $M_n(F)$ over an infinite field $F$ is called power central.

These considerations motivate the following application:

\begin{Definition} A polynomial $p$ is $\nu$-\textit{central} if $p^\nu$ is central, for $\nu\ge 1$ minimal such. The
polynomial $p$ is \textit{power-central} if $p$ is $\nu$-central,
for some $\nu >1$.
\end{Definition}

\begin{Remark}\label{genericten0-pcp}
The existence of a $\nu$-central polynomial is equivalent to $\UD$
containing an element whose $\nu$-power is central, with $\nu$
minimal such. On the other hand, any such element can be
specialized to an arbitrary division algebra of dimension $n^2$
over its center. Thus, to prove the non-existence of a
$\nu$-central polynomial, it suffices for suitable $\ell $
dividing $\nu$ to construct a~division algebra with center~$\hat K
\supseteq K$ having the property that if $d^\nu \in \hat K$ then
$d^\ell \in \hat K$.
\end{Remark}

Examining dimensions of images, we conclude from
Theorem~\ref{no-mpc-pcp} that multilinear power-central
 polynomials do not exist whenever $n \ge 4$.
Note that it does not contradict to the Theorem~\ref{thmC-pcp}, where we will show that if a multilinear polynomial $p$ is $\nu$-central on~$M_n(K)$
then $\nu=n$. Both statements hold. What we can conclude that power central multilinear polynomials evaluated on matrix algebras $M_n(K)$ can occur only for $n=2$ or $n=3$. For $n=2$ there exists an example (in particular, the polynomial $p(x,y)=[x,y]$, see Example~\ref{ex_multilin-int}(ii) for details), for $n=3$ the question whether such polynomial exists remains being open (see Problem~\ref{3s-3}).

To put these results into perspective, we also consider the
existence of $\nu$-central polynomials which need not be
multilinear.

Using the structure theory of division algebras,
Saltman~\cite{Sa} proved that in characteristic $0$, $\nu$-central
polynomials do not exist for odd $\nu>1$ unless $n$ is prime.

For $n=\nu$ prime, as explained in
\cite[Theorem~3.2.20]{Row1}, this is equivalent to Amitsur's
generic division algebra being cyclic, one of the major open
questions in the theory of division algebras. Although we have
nothing more to say about this question, we have results for other cases.

\begin{Remark}\label{genericten00-pcp} A homogeneous $3$-central polynomial for $n=3$ was constructed in~\cite[Theorem~3.2.21]{Row1}, and a homogeneous $2$-central
polynomial for $n=4$ was constructed in
\cite[Proposition~3.2.24]{Row1}.
\end{Remark}

\begin{Theorem}\label{thmC-pcp}If a multilinear polynomial $p$ is $\nu$-central on $M_n(K)$
then $\nu=n$.
\end{Theorem}
\begin{proof}By Theorem~\ref{thmB1-pcp}, $\Image p$ contains a matrix $M$ of
type $c_n e_{n,1} + \sum\limits_{i=1}^{n-1} c_i e_{i,i+1}$ where
$c_1\cdots c_n\neq 0$. Hence $n \le \nu$. To show that $n = \nu$,
consider the continuous mapping \begin{gather*}\theta\colon \ (x_1,\dots,x_m)\mapsto
(\lambda_1:\lambda_2:\dots:\lambda_n)\end{gather*} to the projective variety,
where $\lambda_i$ are the eigenvalues of $p(x_1,\dots,x_m)$. Since
$p$ is $\nu$-central, $\theta$~is discrete and therefore
constant. By Theorem~\ref{thmB1-pcp} there are matrices
$a_1,\dots,a_m$ such that $p(a_1,\dots,a_m)=M$ and thus
$\theta(a_1,\dots,a_m)=\big(1:~\varepsilon:~\varepsilon^2:~\dots:~\varepsilon^{n-1}\big)$.
Therefore~$p$ is $n$-central, as desired.
\end{proof}

As a special case of Remark~\ref{dim-f0-pcp}, if $p$ is power-central then $\Image p$ has
dimension $n^2-n+1$.

\subsubsection{Considerations arising from division algebras}

A major tool is Amitsur's theorem (Proposition~\ref{Am1-2}). Here is an easy consequence of the theory of division
algebras.

\begin{Lemma}\label{div-pcp} If a polynomial $p$ is $\nu$-central for $M_n(K)$ for $\nu>1$, then $\nu$
cannot be relatively prime to $n$.
\end{Lemma}
\begin{proof} We can view $p$ as an element of the generic
division algebra $\widetilde{\UD}$ of degree~$n$, and we adjoin an
$\nu$-root of 1 to $K$ if necessary. Then $p$ generates a subfield
of $\widetilde{\UD}$, of dimension divi\-ding~$\nu$. Hence the
dimension is~1; i.e., $\nu=1$.
\end{proof}

We shall need a general fact about polynomial evaluations.
\begin{Lemma}\label{noncom-pcp} For any polynomial $p(x_1, \dots, x_m)$ which has an evaluation of degree $n$ on $M_n(K)$,
there is an index $i,\ 1\le i\le m$, and matrices $a_1, a_2,
\dots, a_m,a_i'$ such that the evaluations $ p(a_1, \dots,
a_{i-1},a_i,a_{i+1},\dots, a_m)$ and $ p(a_1, \dots,
a_{i-1},a_i',a_{i+1},\dots, a_m)$ do not commute.
\end{Lemma}
\begin{proof} We go back and forth to generic matrices and $\widetilde{\UD}$. First of all, for all generic matrices $Y_1, \dots, Y_m, Y_1', \dots, Y_m'$,
and each $i$, clearly $ p(Y_1, \dots, Y_i,Y'_{i+1},\dots, Y'_m)$
has degree $n$ over $K_1$, and thus has distinct eigenvalues, from
which it follows at once that $ p(Y_1, \dots, Y_m)$ and $
p(Y_1',\dots, Y'_m)$ do not commute (since one could diagonalize $
p(Y_1, \dots, Y_m)$ while $ p(Y_1',\dots, Y'_m)$ remains
non-diagonal).

But, for each $i$, $K_1( p(Y_1, \dots, ,Y_i,Y'_{i+1},\dots,
Y'_m))$ has dimension $n$ over $K_1$ and thus is a maximal
subfield of $\widetilde{\UD}$. It follows that $ p(Y_1, \dots,
Y_i,Y'_{i+1},\dots, Y'_m) $ and $ p(Y_1, \dots,
Y_{j-1},Y'_{j},\dots, Y'_m)$ commute iff \begin{gather*}K_1( p(Y_1, \dots,
Y_i,Y'_{i+1},\dots, Y'_m)) = K_1( p(Y_1, \dots,
Y_{j-1},Y'_{j},\dots, Y'_m)).\end{gather*} In particular, $K_1( p(Y_1,
\dots, Y_m)) \ne K_1( p(Y'_1, \dots, Y'_m) )$, implying \begin{gather*}K_1(
p(Y_1, \dots, Y_i,Y'_{i+1},\dots, Y'_m))\ne K_1( p(Y_1, \dots,
Y_{i-1},Y'_{i},\dots, Y'_m))\end{gather*} for some $i$, and thus $ p(Y_1,
\dots, Y_i,Y'_{i+1},\dots, Y'_m) $ and $ p(Y_1, \dots,
Y_{i-1},Y'_{i}\dots, Y'_m)$ do not commute. In other words,
\begin{gather*}[ p(Y_1, \dots, Y_i,Y'_{i+1},\dots, Y'_m) , p(Y_1, \dots,Y_{i-1},Y'_{i}\dots, Y'_m)] \ne 0,\end{gather*} implying there is a specialization
\begin{gather*}Y_1 \mapsto a_1, \dots, Y_i \mapsto a_i , Y_i' \mapsto a_i', Y'_{i+1} \mapsto a_{i+1}, \dots, Y'_m \mapsto a_m\end{gather*}
yielding $[ p(a_1, \dots, a_{i-1},a_i,a_{i+1},\dots, a_m),p(a_1,
\dots, a_{i-1},a_i',a_{i+1},\dots, a_m)] \ne 0$.
\end{proof}

\begin{Example}\label{genericten-pcp} Suppose $K$ has characteristic $\ne 2$, and $n = 2^{t-1} q$ where $q$ is odd, and
construct~$D$ to be a tensor product of ``generic'' symbols
$\big(\lambda_u^{n_u} , \mu_u^{n_u} \big)$ as in \cite[Example~7.1.28 and
Theorem~7.1.29]{Row1}, where $n_1 = n_2 = \dots = n_{t-1} = 2$
and $n_t = q$. In other words, $D$ is the algebra of central
fractions of the skew polynomial ring $R :=
K(\varepsilon)[\lambda_u\mu_u\colon 1 \le u \le t]$ where $\varepsilon$
is a primitive $q$ root of~1 and the indeterminates commute except
for $\lambda_u \mu_u = - \mu_u \lambda_u$ for $1 \le u \le t-1$
and $\lambda_t \mu_t =\varepsilon \mu_t \lambda_t$. We write a
typical element of $R$ as $\sum \a _{\bfi,\bfj} \lambda ^{\bfi}
\mu ^{\bfj}$, where $\lambda ^{\bfi}$ denotes $\prod\limits_{u=1}^t
\lambda_u^{i_u}$. There is a natural grade given by the
lexicographic order on the exponents of the monomials, and it is
easy to see that if $d^\nu \in K$ then the leading term $\hat
d^\nu \in K$.

In particular, for $\nu =2$, if $d^2 \in K$ then $\hat d$ must
have the form $\a \lambda ^{\bfi} \mu ^{\bfj}$ where $i_t = j_t =
0$. On the other hand, we claim that if $ d ^2 \in K$ and $\hat d
\in K$, then $d \in K$. Indeed, taking $d'$ to
 be the next leading term in $d$, we have
 \begin{gather*}d^2 = \big(\hat d + d'\big)^2 = \hat d^2 + 2\hat d d' + \cdots,\end{gather*}
 implying $d' \in K$, and continuing, we conclude $d\in K$, as desired.

 It follows that if $d$ is $2$-central then $\hat d$ is $2$-central.

 Now we claim that $D$ does not have $4$-central elements. Indeed,
 if $d$ is 4-central then $d^2$ is $2$-central, implying $\hat d^2$ is
 $2$-central,
 and thus $\hat d$ is $4$-central, implying $\hat d$ must have
the form $\a \lambda ^{\bfi} \mu ^{\bfj}$ where $i_t = j_t = 0;$
we conclude that $\hat
 d$ is $2$-central, implying $\hat d^2 \in K$, and thus $d^2 \in
 K$ by the claim.
\end{Example}

\begin{Theorem}[Rowen--Saltman]\label{generictena-pcp} $4$-Central
polynomials do not exist.
\end{Theorem}
\begin{proof} Combine Remark~\ref{genericten0-pcp}
with Example~\ref{genericten-pcp}.\end{proof}

This leaves us to look for $2$-central polynomials.

\begin{Proposition}\label{generictena1-pcp} There exist homogeneous $2$-central
polynomials with respect to $M_n(K)$ if $n = 2q$ or $n = 4q$ for
$q$ odd.
\end{Proposition}
\begin{proof} $\widetilde{\UD}$ is a tensor product of a division algebra $D_1$ of degree 2 or 4
and a division algebra of degree $q$, and we observed earlier that
$D_1$ has a $2$-central element. \end{proof}

The situation for $8|n$ remains open, and is equivalent to another
important question in division algebras about the existence of
square-central elements. The following observation might be
relevant, although we do not use it further.

\subsubsection{2-central polynomials linear in the first indeterminate}

Having settled the issue for homogeneous polynomials except for $n = 8q$, we turn to 1-linear polynomials, where the story ends
differently. We use the division algebra approach.

\begin{Lemma} Any division algebra $D$ with a $2$-central subspace $V$ of dimension $2$ contains a~$K$-central quaternion subalgebra. In particular, $n: = \deg(D)$ cannot be odd. Also, $4$ does not divide $n$ if $D$ also has exponent~$n$.
\end{Lemma}

\begin{proof} Take $v, v' \in V $. Then $v^2,
{v'}^2 \in K$. But also, by assumption, $v+v' $ is also
square-central, so
\begin{gather*}v^2 +vv' + v'v +{v'}^2 = (v+v')^2 \in K,\end{gather*} implying $v'v = -vv' +
\alpha$ for some $\alpha \in K$. But then $K + Kv + Kv' + Kvv'$ is
a central $K$-subalgebra of $D$ and has dimension at least $3$, but
has elements of degree~$2$, so has dimension~$4$.

The last assertion follows easily from the theory of finite
dimensional division algebras. If~$D$ has a quaternion division
algebra then~$2$ must divide $n$ and the exponent of~$D$ is the
least common multiple of~$2$ and~$\frac n2$.
\end{proof}

We are ready to improve Theorem~\ref{no-mpc-pcp}(i) in certain cases.

\begin{Proposition}\label{thmB-pcp} If $p(x_1, \dots, x_m)$ is a $2$-central polynomial for $n\times n$
matrices, linear in $x_1$, and there are non-commuting values
$p(a_1, \dots, a_m)$ and $p(a_1', \dots, a_m)$ for matrices $a_1,
a_1', a_2, \dots, a_m$, then the generic division algebra
$\widetilde{\UD}$ of degree $n$ has a $K_1$-central quaternion
subalgebra. In particular, $n$ cannot be odd, and $4$ does not
divide $n$.
\end{Proposition}

\begin{proof} Let \begin{gather*}w = p(Y_1, \dots, Y_m), \qquad w' = p(Y_1', Y_2 \dots,
Y_m),\end{gather*} where the $Y_i$ and $Y_1'$ are generic matrices. Then
$w^2, {w'}^2 \in K_1$. But also, by definition, $w+w' =
p(Y_1+Y_1', \dots, Y_m)$ is also $2$-central, so \begin{gather*}w^2 +ww' + w'w
+{w'}^2 = (w+w')^2 \in K_1,\end{gather*} implying $w'w = -ww' + \alpha$ for
some $\alpha \in K_1$. But then $K_1 + K_1w + K_1w' + K_1ww'$ is a
central $K_1$-subalgebra of $\widetilde{\UD}$ and has dimension~$4$.

The last assertion follows since the generic division algebra
$\widetilde{\UD}$ of degree $n$ has exponent~$n$, whereas if
$\widetilde{\UD}$ has a central quaternion division subalgebra,
then 2 must divide $n$ and the exponent of $\widetilde{\UD}$ is
the least common multiple of~$2$ and~$\frac n2$.
\end{proof}

This conclusion is the opposite of Proposition~\ref{generictena1-pcp},
when $n=4q$ for $q$ odd.
 The hypothesis clearly holds when $p$ is multilinear. Indeed,
in view of Lemma~\ref{noncom-pcp}, two consecutive terms of the chain
\begin{gather*}p(a_1, \dots, a_m), p(a'_1,a_2, \dots, a_m), \dots, p(a'_1,
\dots, a'_m)\end{gather*} do not commute. On the other hand, it also holds
when $p(x_1, x_2, \dots, x_2)$ is non-central, since we could take
$Y_1' = Y_1 Y_2 Y_1^{-1}$.

\subsection{Open problems}

The investigations of possible image sets may be difficult, this is the reason to try first to construct interesting examples of polynomials with non-standard image sets:
\begin{Question}
What interesting examples of evaluations of homogeneous polynomials $($with respect to Zariski topology$)$ on high rank matrix algebras can occur?
\end{Question}
\begin{Remark}These examples can be nontrivial (see Section~\ref{pcps}) and there
is a little hope to complete description at this moment. Thus, at
this moment it would be nice to provide interesting examples and
constructions.
\end{Remark}

Another famous problem is the Freiheitsatz for associative algebras in nonzero
characteristic:

\medskip

\noindent
{\bf Freiheitsatz.} {\it For $P \in A = k\langle{x_1,\dots,x_n,
t}\rangle$ which involves $t$ nontrivially, the algebra
$k\langle x_1,\dots,\allowbreak x_n \rangle$ can be naturally embedded into the
quotient $k\langle{x_1,\dots, x_n, t}\rangle/\operatorname{Id}(P)$.}

\medskip

For $\Char K=0$ the Freiheitsatz was solved by Makar-Limanov~\cite{ML} and generalized by Kolesnikov \cite{Ko1,Ko2}.

{\it Working hypothesis:} The Freiheitsatz is true in arbitrary characteristic.

 {\it Methods:} The Freiheitsatz for associative algebras was established in the
case char$(k)=0$ by Makar-Limanov, as a consequence of his
construction of an algebraically closed skew-field. Later on,
P.~Kolesnikov developed Makar-Limanov's ideas and, in particular,
improved his exposition. The proof is based on solvability of
equations in the algebra of Malcev--von Neumann series which are
related to differential operators. The construction of the algebra
in which the corresponding equations were solved is none other
than the $*$-operation related to Kontsevich's formal
quantization.

Makar-Limanov has proposed proving that the co-rank of an arbitrary
polynomial $p$ on $M_n(k)$ is bounded by some reasonable function of~$n$. This would yield the Freiheitsatz almost immediately. The
difficulty is that $p$ need not be homogeneous. We have some
information obtained by examining generic matrices, and the hope is
that they will be amenable to new geometric techniques. Topological
maps defined on Banach spaces were investigated by Bre\v{s}ar in
\cite{Bre} and his approach can be useful in this question.

\subsubsection{Images of polynomial maps and matrix equations}

Let $P$ be a noncommutative polynomial.

Suppose that for all sufficiently large $n$ we can solve the system
\begin{gather} \label{eq1}
\{P_i(x_1,\dots ,x_s,t)=0\}.
\end{gather}
Let $\xi_i$, $i=1,\dots, s$, $\nu$ be a solution of~(\ref{eq1}).
Then after substituting $x_i\mapsto a_i+\xi_i$, $t\mapsto
t+\nu$, we obtain an equation without a free term. Further,
assume that after such a substitution there appears a term $B(t)$
linear in~$t$. Let us regard $B$ as an element of an operator
algebra. If it turns out that the operator~$B$ is invertible, we
will be able to solve~(\ref{eq1}) using the method of consecutive
approximations in the product of the matrix algebra and the free
algebra, and thus see that this equation does not impose any
restrictions on the original $x_1,\dots, x_s$.

It is worth noting that it is enough to know how to solve
equations (and prove invertibility of operators) modulo matrices
of bounded rank and as $n\to\infty$.

Therefore the following question is relevant:
\begin{Problem}
Give a classification of possible image sets of arbitrary $($or homogeneous$)$ polynomial evaluated on matrices of rank much higher than the polynomial degree.
\end{Problem}

For the multilinear case, one can consider the union of all ranks matrices $M_\infty(K)$: the set of infinite matrices with finite number of nonzero entries.
 The working hypothesis of the L'vov--Kaplansky conjecture) is quite reasonable:
\begin{Problem}
Let $p$ be any multilinear polynomial evaluated on $M_\infty(K)$.
Then $\Image p$ is the set of trace vanishing matrices
$\ssl_\infty(K)$, if and only if $p$ can be written as a sum of
commutators; otherwise it is the entire set $M_\infty(K)$.
\end{Problem}

\section{Lie polynomials and non-associative generalizations}\label{lie}

Next we consider a Lie
polynomial $p$. We describe all the possible images of $p$ in $M_2(K)$
and provide an example of such $p$ whose image is the set of
non-nilpotent trace zero matrices, together with 0. We provide an
arithmetic criterion for this case. We also show that the standard
polynomial $s_k$ is not a Lie polynomial, for $k>2$.

This section, which consists of two parts, studies the
L'vov--Kaplansky conjecture for Lie algebras. Even the case of Lie
identities has room for further investigation, although it has
already been
 studied in two important books \cite{Bak,Ra4}. In the first part we
are interested in images of Lie polynomials on $M_n(K)$, viewed as
a Lie algebra, and thus denoted as $\gl_n(K)$ (or just $\gl_n$ if
$K$ is understood). Since $[f,g]$ can be interpreted as $fg-gf$ in the free associative algebra,
 we identify any Lie polynomial with an associative polynomial. In this way, any set that can arise as the image of a Lie
polynomial on the Lie algebra $\gl_n$ also fits into the framework
of the associative theory of~$M_n(K)$, and our challenge here is to
find examples of Lie polynomials that achieve the sets described in Sections~\ref{2q7} and~\ref{3a}.

As we shall see, this task is not so easy as it may seem at first
glance. At the outset, the situation for Lie polynomials is subtler than for
regular polynomials, for the simple reason that the most prominent
polynomials in the theory, the standard polynomial $s_n$ and the
Capelli polynomial~$c_n$, turn out not to be
 Lie polynomials. We first consider Lie identities, proving that the standard
polynomial $s_k $ is not a Lie polynomial for $k>2$. A key role is played by~$\sl_n$,
the Lie algebra of $n\times n$ matrices over $K$ having trace~0.

Then we
classify the possible images of Lie polynomials evaluated on
$2\times 2$ matrices, based on
Section~\ref{2q} where the field $K$ was required to be
quadratically closed, and Section~\ref{2r}, where results were
provided over real closed fields, some of them holding more
generally over arbitrary fields. We also consider the $3\times 3$ case (Section~\ref{3a}).

\subsection[Homogeneous Lie polynomials on $\gl_n$ and $\sl_n$]{Homogeneous Lie polynomials on $\boldsymbol{\gl_n}$ and $\boldsymbol{\sl_n}$}

We refine Conjecture~\ref{Polynomial image-int} to Lie polynomials, and ask:

\begin{Question} What is the possible image set $\Image f$ of a Lie polynomial $f$ on
$\gl_n$ and~$\sl_n$? For which Lie polynomials $f$ do we achieve
this image set? For example, what are the Lie identities of
smallest degree on $\gl_n$ and~$\sl_n$?
\end{Question}

In order to pass to the associative theory, we make use of the
\textit{adjoint algebra} $\ad L = \{ \ad_a : L \to L : a \in L\}$
given by $\ad_a (b) = [a,b]$. Note that \begin{gather*}\dim _K (\ad L) < \dim _K
\End _K (L) = (\dim _K L)^2.\end{gather*} The map $a \mapsto \ad_a$ defines a
well-known Lie homomorphism $L \to \ad L$.

We write $ [a_1, \dots, a_t ]$ for $[a_1, \dots,[a_{t-1}, a_t]]$,
and $ \big[a^{(k)}, a_t \big]$ for $ [a, \dots, a, a_t ]$ where $a$ occurs
$k$ times. By \textit{ad-monomial} we mean a term $\alpha \, \ad
x_{i_1} \cdots \ad x_{i_t}$ for some $\alpha\in K$. By
\textit{ad-polynomial} we mean a sum of ad-monomials.

\begin{Remark}\label{corres}
\begin{gather*}\ad _{a_1}\cdots \ad_{a_t}(a) = [a_1, \dots, a_t ,a ].\end{gather*} In this way, any ad-monomial corresponds to a Lie monomial,
and thus any ad-polynomial $f(\ad _{x_1}, \dots, \ad _{x_t})$ gives
rise to a Lie polynomial $f( x_1 , \dots, x_t, y)$ taking on the
same values, and in which $y$ appears of degree~1 in each Lie
monomial in the innermost set of Lie brackets.\end{Remark}

Recall that an associative polynomial $f (x_1, \dots, x_k)$ is
\textit{alternating in the last $m+1$ variables} if $f$ becomes 0
whenever two of the last $m+1$ variables are specialized to the same
quantity. This yields:

\begin{Proposition}\label{spec0} Suppose $L$ is a Lie algebra of dimension
$m$, and $f (x_1, \dots, x_k)$ is a multilinear polynomial
alternating in the last $m+1$ variables. Then
\begin{gather*}f(\ad _{x_1}, \dots , \ad _{x_k})(y)\end{gather*} corresponds to a Lie
identity of $L$ of degree $\deg f +1$.
\end{Proposition}
\begin{proof}\sloppy The alternating property implies
$f (x_1, \dots, x_k)$ vanishes on $\ad L$,
 cf.~\cite[Proposi\-tion~1.2.24]{Row},
 so every substitution of $f(\ad _{x_1}, \dots , \ad _{x_k})(y)$
vanishes.
\end{proof}

Since the alternating polynomial of smallest degree is the standard
polynomial~$s_{m+1}$, we have a Lie identity of degree $m+2$ for any
Lie algebra of dimension~$m$. In particular, $\dim (\sl_n) = n^2-1$,
yielding:

\begin{Corollary}\label{min1} $\sl_n$ satisfies a Lie identity of degree
$n^2+1$.
\end{Corollary}

Corollary~\ref{min1} gives rise to the following question:

\begin{Question} What is the minimal degree $m_n$ of a Lie identity of
$\sl_n$? \end{Question}

By Corollary~\ref{min1}, $m_n \le n^2+1$, and in particular $m_2
\le 5$. Even the answer $m_2 = 5$ given in \cite[Theorem~36.1]{Ra4},
is not easy, although a reasonably fast combinatoric approach is
given in \cite[p.~165]{Bak}, where it is observed that any since any
Lie algebra $L$ satisfying a Lie identity of degree $<5$ is
solvable, one must have $m_2 \ge 5$, yielding $m_2 = 5$.
\v{S}penko~\cite[Proposition 7.5]{S} looked at this from the other
direction and showed that if $p$ is a Lie polynomial of degree
 $\le 4$ then $\Image p=\ssl_2$.

Conversely, we have:

\begin{Proposition}\label{spec} Suppose $f(x_1, \dots, x_t, y)$ is a Lie
polynomial in which $y$ appears in degree~$1$ in each of its Lie
monomials. Then $f$ corresponds to an ad-polynomial taking on the
same values on $L$ as $f$.\end{Proposition}
\begin{proof} In view of Remark~\ref{corres}, it suffices to show
that any Lie monomial $h$ can be rewritten in the free Lie algebra
as a sum of Lie monomials in which $y$ appears (in degree 1) in the
innermost set of Lie brackets. This could be done directly by means
of the Jacobi identity, but here is a~slicker argument.

Write $h = [h_1,h_2]$, and we appeal to induction on the degree of
$h$. $y$ appears say in~$h_2$. If \mbox{$h_1 = y$} then we are done since
$h = -[h_2,y]$ corresponds to $-\ad _{h_2}$. Likewise if $h_2 = y$.
In general, by induction, $h_1$ corresponds to some ad-monomial $\ad
_{x_{i_1}}\cdots \ad _{x_{i_k}}(y)$ and $h_2$ corresponds to some
ad-monomial $\ad _{x_{i_{k+1}}}\cdots \ad _{x_{i_\ell}}(y)$,
 so $[h_1,h_2]$ corresponds
to $\ad _{x_{i_1}}\cdots \ad _{x_{i_k}}((\ad _{x_{i_{k+1}}}\cdots
\ad _{x_{i_\ell}})(y)) \allowbreak =\ad _{x_{i_{1}}}\cdots \ad _{x_{i_\ell}}(y)$
as desired.
\end{proof}

\begin{Corollary}\label{mustbe} Any homogeneous Lie polynomial of degree $\ge 3$ must be
an identity $($viewing the Lie commutator $[a,b]$ as $ab-ba)$ of the
Grassmann algebra $G$.
\end{Corollary}
\begin{proof} Each term includes $[x_i,x_j,x_k]$, which is well known to be an identity of~$G$.
 \end{proof}

\begin{Example}\label{usefulex}\quad
\begin{enumerate}\itemsep=0pt
\eroman \item The standard polynomial $s_2$ itself is a Lie
polynomial.

\item $s_4$ vanishes on $\sl_2$ (viewed inside the associative algebra $M_2(K)$),
since $\sl_2$ has dimension~3. But surprisingly, this is not the
polynomial of lowest degree vanishing on~$\sl_2$, as we see next.

\item Bakhturin~\cite[Theorem~5.14]{Bak} points out that $f = [(x_1x_2 +
x_2 x_1),x_3]$ vanishes on $\sl_2$. In other words, $a_1a_2 + a_2
a_1$ is scalar for any $2 \times 2$ matrices $a_1,a_2$ of trace 0.
Indeed, $a_i^2 $ is scalar for $i=1,2$, implying $a_1a_2 + a_2 a_1$
is scalar unless $a_1$, $a_2$ are linearly independent, in which case
$a_1a_2 + a_2 a_1$ commutes with both $a_1 $ and $ a_2$, and thus
again is scalar. But $f$ is not a Lie polynomial, as seen via the
next lemma.
\end{enumerate}
\end{Example}

This discussion motivates us to ask when a polynomial is a Lie polynomial. Here is a very easy criterion which is of some use.

\begin{Lemma}\label{Liep1} Any Lie polynomial which vanishes on $\sl_n$ is an
identity of $\gl_n$.
\end{Lemma}
\begin{proof} Immediate, since $\gl_n ' = \sl_n'$.
\end{proof}

The standard polynomial $s_4$ is not a Lie polynomial. Here are
three ways of seeing this basic fact.

 \begin{enumerate}\itemsep=0pt \eroman \item Confront Example \ref{usefulex}
 with the fact that $m_2 = 5$, whereas $\deg s_2 = 4$.

\item A computational approach. We have 15 multilinear Lie
monomials of degree~4, namely $\frac{1}{2}\binom 42 = 3$ of the form
\begin{gather}\label{firsttype} [[x_{i_1}, x_{i_2}],[x_{i_3}, x_{i_4}]]\end{gather}
and $2 \binom 42 = 12$ of the form
\begin{gather}\label{secondtype}[[[x_{i_1}, x_{i_2}],x_{i_3}],
x_{i_4}].\end{gather} But
\begin{align*} [[x_{i_1}, x_{i_2}],[x_{i_3},
x_{i_4}]] & = \ad_{[x_{i_3}, x_{i_4}]} \ad _{x_{i_2}}
(x_{i_1})\\
& =\ad_{x_{i_3}} \ad_{ x_{i_4}}\ad _{x_{i_2}}(x_{i_1})- \ad_{ x_{i_4}}
\ad_{ x_{i_3}}\ad _{x_{i_2}}(x_{i_1}), \end{align*} so we can
rewrite the equations \eqref{firsttype} in terms of~\eqref{secondtype}. Furthermore, with the help of the Jacobi
identity, \eqref{secondtype} can be reduced to seven independent Lie
monomials, and one can show that these do not span~$s_4$. Even
though this might seem unduly complicated, it provides a general
program to verify that a given polynomial is not Lie.

\item The third approach is simpler and works for~$s_k$, for any $k>2$.

P.M.~Cohn was the first to tie the standard polynomial to the
infinite dimensional Grassmann algebra $G$ with base $e_1, e_2,
\dots$, by noting that $s_k(e_1, \dots, e_k) = k!e_1 \cdots e_k \ne
0$ when $k! \ne 0$.
\end{enumerate}

\begin{Theorem} The standard polynomial $s_k$ is not a Lie polynomial,
for any $k>2$.
\end{Theorem}
\begin{proof} Otherwise, by Corollary~\ref{mustbe} it would be an
identity of $G$, contradicting Cohn's observation (taking $\Char K =
0$).
\end{proof}

\subsection[The case $n = 2$]{The case $\boldsymbol{n = 2}$}

Recall from Corollary~\ref{min1} that there is a Lie identity of degree~5.

\begin{Theorem}\label{main} If $p$ is a homogeneous Lie polynomial evaluated on the
matrix ring $M_2(K)$, where $K$ is an algebraically closed field,
then $\Image f$ is either $\{0\}$, or $K$ $($the set of scalar
matrices$)$, or the set of all non-nilpotent matrices having trace zero, or $\ssl_2(K)$, or $M_2(K)$.
\end{Theorem}
\begin{Remark}Nonzero scalar matrices can be obtained in Theorem~\ref{main} only
when $\Char K=2$, and the last case $M_2(K)$ is possible only if $\deg f=1$.
\end{Remark}
\begin{proof}[Proof of Theorem~\ref{main}]
According to Theorem \ref{imhom-2} the image of $p$ must be either
 $\{0\}$, or $K$, or the set of all non-nilpotent matrices having
trace zero, or $\ssl_2(K)$, or a dense subset of~$M_2(K)$ (with
respect to Zariski topology). Note that if at least one matrix
having nonzero trace belongs to the image of~$p$, then $\deg p=1$
and thus $\Image p=M_2(K)$.
\end{proof}
\begin{Theorem} For any algebraically closed field $K$ of characteristic $\ne 2$,
 the image of any Lie polynomial~$p$ $($not necessarily homogeneous$)$ evaluated on $\ssl_2(K)$
 is either $\ssl_2(K)$, or~$\{0\}$, or the set of trace zero non-nilpotent matrices.
\end{Theorem}
\begin{proof} For $p$ not a PI,
 we can write $p=f_j+f_{j+1}+\dots+f_d$, where each $f_i$
 is a homogeneous Lie polynomial of degree $i$, and $f_d$ is not PI.
 Therefore for any $c\in K$ we have
 \begin{gather*}p(cx_1,cx_2,\dots,cx_m)=c^jf_j+\dots+c^df_d.\end{gather*}
 Since $f_d$ is not $PI$, we can take specializations of $x_1,\dots, x_m$ for which $\det(f_d)\neq 0$.
Fixing these specializations, we consider
$\det(c^jf_j+\dots+c^df_d)$
 as a polynomial in $c$ of degree $j + \cdots + d$.
Since the leading coefficient is not zero and $K$ is algebraically
closed, its image is $K$.
 Thus for any $k\in K$ there exist $x_1,\dots,x_m$ for which $\det(f)=k$.
 Hence (for $\Char K\neq 2$) any matrix with nonzero eigenvalues $\lambda$ and~$-\lambda$ belongs to $\Image f$.
 Therefore $\Image f$ is either $\ssl_2$ or the set of trace zero non-nilpotent matrices.
\end{proof}

Let us give examples of Lie polynomials having such images:
\begin{Example}[Alexei Kanel-Belov]\label{badex}
If $\Char K=2$, then $\Image p = K$ also is possible:
We take
\begin{gather*}p(x,y,z,t)=[[x,y],[z,t]].\end{gather*}
Any value of $p$ is the Lie product of two trace zero matrices
$s_1=[x,y]$ and $s_2=[z,t]$. Both can be written as
$s_i=h_i+u_i+v_i$, where the $h_i$ are diagonal trace zero
matrices (which are scalar since $\Char K=2$), the $u_i$ are
proportional to $e_{12}$, and the $v_i$ are proportional to
$e_{21}$. Thus $[s_1,s_2]=[u_1,v_2]+[u_2,v_1]$ is scalar.

 Over an arbitrary field, $\Image p$
can indeed be equal to $\{0\}$, or $K$, or the set of all
non-nilpotent matrices having trace zero, or $\ssl_2(K)$, or
$M_2(K)$.
\begin{enumerate}\itemsep=0pt \eroman \item $\Image x=M_2(K)$.

\item $\Image [x,y]=\ssl_2$.

\item Next, we construct a Lie polynomial whose image evaluated on
$\ssl_2(K)$ is the set of all non-nilpotent matrices having trace
zero. We take the multilinear polynomial $h(u_1,\dots,u_8)$
constructed in \cite{DK2} by Drensky and Kasparian which is central
on $3\times 3$ matrices. Given $2\times 2$ matrices $x_1,\dots,
x_9$ we consider the homogeneous Lie polynomial
\begin{gather*}p(x_1,\dots,x_9)=h(\ad_{ [x_9 ,x_9 ,\dots, x_9,x_1]},\ad_{x_2},\ad_{x_3},\dots,\ad_{x_8})(x_9).\end{gather*}
For any $2\times 2$ matrix $x$, $\ad_x$ is a $3\times 3$ matrix
since $\ssl_2$ is $3$-dimensional; hence, for any values of $x_i$,
the value of $p$ has to be proportional to $x_9$. However for $x_9$
nilpotent, this must be zero, since $\big[x^{(3)},y\big]=0$ for any $y\in\ssl_2(K)$ if $x$ is nilpotent. (When we open the brackets we have the sum of $8$ terms and each term equals $x^kyx^{3-k}$. But for any integer~$k$, either $k\geq 2$ or $3-k\geq 2$.) Thus the image of~$p$ is exactly the set of non-nilpotent trace zero matrices.

Another example of a homogeneous Lie polynomial with no nilpotent
values is $p(x,y)=[[[x,y],x],[[x,y],y]]$. (See \cite[Example~4.9]{BGKP} for details.)
\end{enumerate}
\end{Example}

\subsection[The case $n = 3$]{The case $\boldsymbol{n = 3}$}\label{3aa}

New questions arise concerning the possible evaluation of Lie
polynomials on $M_n(K)$.

According to Theorem \ref{semi_tr0_3-3}, if $p$ is a homogeneous polynomial with trace vanishing image, then $\Image p$ is one of the following:

\begin{itemize}\itemsep=0pt
\item $\{0\}$, \item the set of scalar matrices (which can occur
only if $\Char K=3$),
 \item a dense subset of $\sl_3(K)$, or \item
the set of $3$-scalar matrices, i.e., with
eigenvalues $\big(c,c\omega,c\omega^2\big)$,
where $\omega$ is our cube root of~$1$.
\end{itemize}

Drensky~and Rashkova~\cite{DR} have found several identities of
$\sl_3$ of degree 6, but they cannot be Lie polynomials, since
otherwise they would be identities of $\gl_3$ and thus a multiple of
$s_6$, which is not a Lie polynomial. Thus, one must go to higher
degree.

In the associative case, the fact that the generic division algebra
has a $3$-central element implies that there is a homogeneous
$3$-central polynomial for $M_3(K)$, i.e., all of whose values
take on eigenvalues $c$, $\omega c$, $c\omega ^2$, where $\omega$ is a
cube root of~$1$. But any matrix with these eigenvalues is either
scalar or has trace~$0$. This leads us to the basic questions needed
to complete the case $n=3$:

\begin{Question} Is there a Lie polynomial $p$ whose values on $\sl_3 (\mathbb C)$ are
dense, but does not take on all values?
\end{Question}

\begin{Question} Is there a Lie polynomial $p$ whose values on
$\sl_3$ all take on eigenvalues $c$, $\omega c$, $c\omega ^2$, where
$\omega$ is a primitive cube root of $1$?
\end{Question}

\subsubsection{A group theoretical question and its relation to the Lie theoretical problem}

There is a group conjecture (see \cite[Question~2]{KBKP} for the
more general case):

\begin{Conjecture}\label{word-PSL2} If the field $K$ is algebraically closed
of characteristic $0$, then the image of any nontrivial group word
$w(x_1,\dots,x_m)$ on the projective linear group ${\rm PSL}_2(K)$ is
${\rm PSL}_2(K)$.
\end{Conjecture}

\begin{Remark} Note that if one takes the group $\SL_2$ instead of ${\rm PSL}_2$, Conjecture~\ref{word-PSL2}
fails, since the matrix $-I+e_{12}$ does not belong to the image of
the word map $w=x^2$.
\end{Remark}

\begin{Example}
When $\Char K=p>0$, the image of the word map $w(x)=x^p$ evaluated
on ${\rm PSL}_2(K)$ is not ${\rm PSL}_2(K)$. Indeed, otherwise the matrix
$I+e_{12}$ could be written as $x^p$ for $x\in{\rm PSL}_2(K)$.
 If the eigenvalues of $x$ are equal, then $x=I+n$ where $n$ is nilpotent. Therefore $x^p=(I+n)^p=I+pn=I$.
 If the eigenvalues of $x$ are not equal, then $x$ is diagonalizable and therefore $x^p$ is also diagonalizable, a contradiction.
\end{Example}

\begin{Lemma}[Liebeck, Nikolov, Shalev, cf.~also \cite{BanZar,G}] $\Image w$ contains all matrices from ${\rm PSL}_2(K)$ which
are not unipotent.
\end{Lemma}

\begin{proof}According to \cite{B} the image of the word map $w$ must be Zariski
dense in $\SL_2(K)$. Therefore the image of $\tr w$ must be Zariski
dense in $K$. Note that $\tr w$ is a homogeneous rational function
and $K$ is algebraically closed. Hence, $\Image (\tr w)=K$. For any
$\lambda\neq\pm 1$ any matrix with eigenvalues $\lambda$ and
$\lambda^{-1}$ belongs to the image of $w$ since there is a matrix
with trace $\lambda+\lambda^{-1}$ in $\Image w$ and any two matrices
from $\SL_2$ with equal trace (except trace $\pm 2$) are similar.
Note that the identity matrix $I$ belongs to the image of any word map.
\end{proof}

However the question whether one of the matrices $(I+e_{12})$ or
$(-I-e_{12})$ (which are equal in ${\rm PSL}_2$) belongs to the image of
$w$ remains open. We conjecture that $I+e_{12}$ must belong to
$\Image w$. Note that if there exists $i$ such that the degree of
$x_i$ in $w$ is $k\neq 0$ then we can consider all $x_j=I$ for
$j\neq i$ and $x_i=I+e_{12}$. Then the value of $w$ is
$(I+e_{12})^k=I+ke_{12}$ and this is a unipotent matrix since $\Char
K=0$, and thus $\Image w={\rm PSL}_2(K)$. Therefore it is interesting to
consider word maps $w(x_1,\dots,x_m)$ such that the degree of each
$x_i$ is zero.

This is why Conjecture~\ref{word-PSL2} can be reformulated as follows:

\begin{Conjecture}\label{word-PGL}
Let $w(x_1,\dots,x_m)$ be a group word whose degree at each $x_i$ is
$0$. Then the image of $w$ on $G=\GL_2(K)/\{\pm 1\}$ must be
${\rm PSL}_2(K)$.
\end{Conjecture}

One can consider matrices $z_i=\frac{x_i}{\sqrt{\det x_i}}$ and note
that $w(z_1,\dots,z_m)=w(x_1,\dots,x_m)$.

For Conjecture \ref{word-PGL} we take $y_i=x_i-I$. Then we can open
the brackets in
\begin{gather*}w(1+y_1,1+y_2,\dots,1+y_m)=1+f(y_1,\dots,y_m)+g(y_1,\dots,y_m),\end{gather*}
where $f$ is a homogeneous Lie polynomial of degree $d$, and $g$ is
the sum of terms of degree greater than $d$. Therefore it is
interesting to investigate the possible images of Lie polynomials,
whether it is possible that the image of $l$ does not contain
nilpotent matrices. Unfortunately we saw such an example
(Example~\ref{badex}, although its degree must be at least~5 by \v{S}penko
\cite[Lemma~7.4]{S1}). More general questions about surjectivity of
word maps in groups and polynomials in algebras are considered in
\cite{KBKP}.

\begin{Remark}
Our next theorem describes the situation in which the trace
vanishing polynomial does not take on nonzero nilpotent values. It
implies that any nontrivial word map $w$ evaluated on ${\rm PSL}_2$ is
not surjective iff its projection to $\ssl_2$ given by $\ssl_2\colon
x\mapsto x-\frac{1}{2}\tr x$ is a~multiple of any prime divisor of
$\det(\pi(w))$. This might help in answering Conjecture
\ref{word-PSL2}.
\end{Remark}
\begin{Theorem}\label{nonilp}Let $p(x_1,\dots,x_m)$ be a trace vanishing polynomial, evaluated on
$M_n(K[\xi ]). $ Let $\bar p = p(y_1, \dots, y_m)$. Then $p$ takes
on no nonzero nilpotent values on any integral domain containing
$K$, iff each prime divisor $d$ of $\det(\bar p)$ also divides each
entry of~$\bar f$.
\end{Theorem}
\begin{proof} $(\Rightarrow)$ If some prime divisor $d$ of $\det(\bar
p)$ does not divides~$\bar p$, then $\bar p$ does not specialize to~$0$ modulo~$d$. Therefore we have a nonzero matrix in the image of~$p$ which has determinant zero and also trace zero, and thus is
nilpotent, a contradiction.

 $(\Leftarrow)$ Assume that $p$ takes on a
nonzero nilpotent value over some integral domain extension of
$K$. Thus $\det \bar p$ goes to 0 under the corresponding
specialization of the $\xi_{i,j}^k$, so some prime divisor $d$ of
$\det(\bar p)$ goes to $0$, and $\bar p$ is not divisible by~$d$.
\end{proof}

\subsection{Open problems for nonassociative algebras}

We described the evaluations of Lie polynomials on $\sl_2$ and
$\gl_2$. One can consider the same problem for other important Lie
algebras:
\begin{Problem}
Give a classification of possible images of multilinear/homogeneous/arbitrary Lie polynomials on the Lie algebra
\begin{itemize}\itemsep=0pt
\item $\so_4$ $($an algebra of rank $2)$,
\item $\so_6$ $($an algebra of rank $3)$,
\item $\ssp_4$ $($an algebra of rank $2)$,
\item $\ssp_6$ $($an algebra of rank $3)$,
\item $\G_2$ $($an algebra of rank $2)$,
\item $\ff_4$ $($an algebra of rank $4)$.
\end{itemize}
\end{Problem}
Analogous problems occur for the non-Lie case:
\begin{Problem}
Investigate possible evaluations of multilinear/homogeneous/arbitrary polynomials on simple Jordan algebras.
\end{Problem}
\begin{Problem}
Investigate possible evaluations of
multilinear/homogeneous/arbitrary polynomials on the Malcev algebra $($for example, see {\rm \cite{ZSSS})}.
\end{Problem}
\begin{Problem}
There are nonassociative operator algebras $($in particular, the
algebra of left multiplications$)$. It is important to investigate
possible evaluations of polynomials on these algebras.
\end{Problem}

\subsection*{Acknowledgements}

We would like to thank M.~Bre\v{s}ar, B.~Kunyavskii and express our special gratitude to E.~Plotkin for interesting and fruitful discussions. We would also like to thank the anonymous referees whose suggestions contributed significantly to the level
of the paper. The second and third named authors were supported by the ISF (Israel
Science Foundation) grant 1994/20. The first named author was supported by the Russian Science Foundation grant No.~17-11-01377.
The second and fourth named authors were supported by Israel Innovation Authority, grant no.~63412: Development of A.I.~based platform for e commerce.


\LastPageEnding

\end{document}